\documentclass[12pt]{amsart}
\input{preamble.tex}
\input{figures.tex}
 
\newcommand{\intset}[1]{[1;#1]}
\begin{document}

\author{Catharina Stroppel}
\address{Mathematisches Institut, Universit\"at Bonn, Endenicher Allee 60, 53115 Bonn, Germany
\href{http://www.math.uni-bonn.de/ag/stroppel/}{catharina.stroppel.de}}
\email{stroppel@math.uni-bonn.de}

\author{Paul Wedrich}
\address{Fachbereich Mathematik, Universit\"at Hamburg, 
Bundesstra{\ss}e 55, 
20146 Hamburg, Germany
\href{https://paul.wedrich.at/}{paul.wedrich.at}}
\email{paul.wedrich@uni-hamburg.de}

\title[Braiding on Soergel bimodules: semistrictness and naturality]{Braiding on type A Soergel bimodules:\\ semistrictness and naturality}
\begin{abstract}  
We consider categories of Soergel bimodules for the symmetric groups $S_n$ in
their $\gln$-realizations for all $n$ in $\N_0$ and assemble them into a locally
linear monoidal bicategory. Chain complexes of Soergel bimodules likewise form a
locally dg-monoidal bicategory which can be equipped with the structure of a
braiding, whose data includes the Rouquier complexes of shuffle braids.
The braiding, together with a uniqueness result, was established  in an infinity-categorical setting in recent work with Yu Leon Liu, Aaron Mazel-Gee and David Reutter. 

In the present article, we construct this braiding explicitly and describe its
requisite coherent naturality structure in a concrete dg-model for the
morphism categories. To this end, we first assemble the
Elias--Khovanov--Williamson diagrammatic Hecke categories as well as categories of
chain complexes thereover into locally linear semistrict monoidal
$2$-categories. Along the way, we prove strictness results for certain standard categorical constructions, which may be of independent interest. In a second step, we provide explicit (higher) homotopies for the naturality
of the braiding with respect to generating morphisms of the
Elias--Khovanov--Williamson diagrammatic calculus. Rather surprisingly, we
observe hereby that higher homotopies appear already for height move relations
of generating morphisms. Finally, we extend the homotopy-coherent naturality
data for the braiding to all chain complexes using cohomology-vanishing
arguments. 
\end{abstract}

\maketitle

\setcounter{tocdepth}{2}

\tableofcontents

%%%% Final checklist before finishing:
%\begin{itemize}
% 	\item colon instead of : when specifying maps
% 	\item (sub)section titles pdf compatible
%	\item search for "the the", "of of"
%	\item delete this checklist
% \end{itemize}
%%

\section{Introduction}
Many invariants in low-dimensional topology use (type A) braid groups and their
representations. These representations often factor through an action of the corresponding Hecke
algebras of the symmetric groups, particularly in the context of
knot and link invariants. A similar observation applies to categorified topological
invariants, a concept pioneered by Khovanov through his categorification of the
Jones polynomial, \cite{Kho}. His construction can be seen as an instance of
actions of categorified (type~A) Hecke algebras on categories, see
\cite{MR2275632, ICM}.   

Categorifications of Hecke algebras and their actions go back to
Kazhdan--Lusztig \cite{KLHecke} and Soergel \cite{Soergel} in the context of the
Kazhdan--Lusztig conjectures in Lie theory, and led to the introduction of the
(additive monoidal categories of) \emph{Soergel bimodules}. In this paper, we
will primarily consider Soergel bimodules for the symmetric groups $S_n$ in
their $\gln$-realizations. Soergel's construction, which recovers the Hecke
algebra with the regular actions after passage to the split Grothendieck group,
has been lifted by Rouquier~\cite{0409593} to a faithful action of the braid
group on the homotopy category of chain complexes of Soergel bimodules. The action of each braid group generator from $\Braidg_n$ is hereby
given by tensoring with a specific complex of Soergel bimodules for $S_n$. As
shown in ~\cite{0409593}, these basic Rouquier complexes satisfy the braid
relations in the homotopy category. The problem of finding an upgrade to a
homotopy-coherent model such that the $\Braidg_n$-actions satisfy all naturality
constraints is the main focus of our paper. We give an affirmative answer and
also provide a concrete framework which combines the aforementioned categorified Hecke algebras with their regular
actions of braid groups for all $n$ into a semistrict monoidal $2$-category with a braiding. 

\smallskip
To formulate the results, we first assemble the non-categorified Hecke
algebras with their regular (braid group) actions together into a braided
monoidal ($1$-)category:

For $n \in \N_0 \coloneqq \{ 0, 1, 2, \ldots \}$ the \textit{Hecke algebra} $H_n$ for the symmetric group $S_n$ 
is a quotient of the group algebra over $\Z[q^{\pm 1}]$ of the Artin braid
group $\Braidg_n$. Taken together, they give rise to a
braided monoidal category $H$ consisting of the following data:

\begin{enumerate}
\item
The objects are given by natural numbers $n \in \N_0$.
\item
The endomorphism algebra of the object $n \in H$ is $H_n$. Other homs are trivial.
\item The monoidal structure is given on objects by $m \otimes n
\coloneqq m + n$, and on morphisms by the map $H_m \times H_n
\to H_{m+n}$ corresponding to the subgroup $S_m \times
S_n \hookrightarrow S_{m+n}$.
\item
The braiding natural isomorphism $m \otimes n \xrightarrow{\sim} n \otimes m$
is given by the image in $\End_H(m+n) \coloneqq H_{m+n}$ of the positive
$(m,n)$-shuffle braid in $\Braidg_{m+n}$.
\end{enumerate}

This braided monoidal category $H$ is (although usually not formulated in this
way) a central player in quantum topology, since all Reshetikhin--Turaev link
invariants of type A \cite{MR1036112} and related topological quantum field
theories (TQFTs) are controlled by it through quantum Schur--Weyl duality. This,
in particular, includes the Jones polynomial \cite{MR766964}, the quantum $\mathfrak{sl}_k$ link invariants \cite{MR1036112}, as well as the TQFTs of
Witten--Reshetikhin--Turaev \cite{MR990772,MR1091619} and
Crane--Yetter--Kauffman \cite{MR1452438}.
\smallskip 

\subsection{A semistrict monoidal 2-category with braiding from categorification}

As mentioned already, the Hecke algebra $H_n$ admits a categorification to the
monoidal additive category $\Sbim_n$ of Soergel bimodules~\cite{Soergel}, in the
sense that its split Grothendieck ring recovers $H_n$. More precisely, it recovers $H_n$ in its Kazhdan--Lusztig presentation. For an interpretation of the usual braid generators, one could work with complexes of Soergel bimodules (instead of $\Sbim_n$). 
Indeed, the images of the braid group generators in
$H_n$ can be categorified by the Rouquier complexes~\cite{0409593} in the bounded
chain homotopy category $\Kb(\Sbim_n)$. The triangulated Grothendieck ring
still recovers $H_n$. That is, the Rouquier complexes satisfy the braid relations up to homotopy and induce a braid group action on the Grothendieck group.

Important for us is the existence of a diagrammatic
monoidal category $\DS_n$ which, upon Cauchy completion, gives a presentation
of $\Sbim_n$ in terms of generators and relations, \cite{EK,EW}.  
These
Elias--Khovanov--Williamson diagrammatic categories provide strict monoidal
models of $\Sbim_n$, which are especially suited for computations. The main
takeaway of this paper is a categorification of $H$: 

\begin{maintheorem}[Monoidality and Braiding]%[\Cref{mainthm}]
\label[Main Theorem]{thmA} There exists a semistrict monoidal $2$-category
$\Kbloc(\DS)$ including the following data:
    \begin{enumerate}

        \item
        
        The objects are given by natural numbers $n \in \N_0$.
        
        \item
        
        The endomorphism category of the object $n$ is $\Kb(\DS_n)$. Other homs are trivial.
        
        \item The monoidal structure is given on objects by $m \otimes n
        \coloneqq m + n$, and on morphism by a functor $\Kb(\DS_m)\times \Kb(\DS_n)
        \to \Kb(\DS_{m+n})$ corresponding to the parabolic subgroup $S_m\times
        S_n\hookrightarrow S_{m+n}$.
        \end{enumerate}
    Moreover, this semistrict monoidal $2$-category is equipped with braiding equivalences $m \otimes n \xrightarrow{\sim} n
    \otimes m$, given by the Rouquier complexes of the positive
    $(m,n)$-shuffle braids in $\Braidg_{m+n}$. These braiding equivalences are natural in both arguments.

\end{maintheorem}
The $2$-category $\Kbloc(\DS)$ and its braiding equivalences control categorified quantum link
invariants such as Khovanov's categorification of the Jones polynomial
\cite{Kho}, see e.g.  \cite{MWW,ICM} and references therein for more context.

The naturality of the braiding equivalences expressed in \hyperlink{thmA}{Main Theorem} and the compatibility of the braiding morphisms with the monoidal
structure\footnote{For $\Kbloc(\DS)$, this compatibility is a
consequence of the Rouquier Canonicity from \Cref{thm:Rouquier-canonicity}.} are two main steps in exhibiting a \emph{braided} monoidal structure on the $2$-category $\Kbloc(\DS)$, see \Cref{conj:main} and \Cref{thm:last}. 

\subsection{(Semi)strictness}  The precise formulation of (semistrict) braided monoidal $2$-categories \cite{MR1266348, MR1402727} has a long history see e.g. \cite{letter, webpage, Gurski} for some background. Readers only familiar with the $1$-categorical notion of  braided monoidal categories might hope for a small list of data and coherence conditions defining an arbitrary braided monoidal bicategory. This is unfortunately too much to hope for, see  \cite{webpage} or \cite[Appendix~C]{schommer-pries-thesis} (at least if small means than then ten dense pages of data and coherences). The semistrictness  condition however dramatically simplifies the situation. A semistrict monoidal $2$-category has a pleasantly manageable set of data and coherence conditions  \cite[Lemma 4]{MR1402727} and a braiding (encoding the data of a braiding equivalence satisfying naturality and coherences) is still possible to spell out \cite[Definition 6]{MR1402727}. Unfortunately, the hope for an interesting (non-symmetric) braiding and for semi-strictness seem not to be very compatible, thus examples providing both are very rare.

The attentive reader might have observed that the higher category constructed in
\hyperlink{thmA}{Main Theorem} is surprisingly strict. Monoidal categories of
bimodules over a fixed commutative ring are monoidal, but typically not strictly
monoidal, and the same is true for Soergel bimodules in their algebraic
incarnation. Thus, we could have expected a bicategory rather than a
$2$-category. Likewise, it may be surprising that the additional monoidal
structure is semistrict, even more so when complexes of Soergel bimodules are
under consideration. To achieve the level of strictness of \hyperlink{thmA}{Main
Theorem} we employ two tools. First, we work with the diagrammatic presentations
of the monoidal categories of Soergel bimodules, which are strict. Second, we
revisit several standard dg categorical constructions and check that they
preserve strictness of monoidal structures, when engineered carefully. For
example, tensoring chain complexes involves taking total complexes, which is
strictly associative only when summands are indexed lexicographically (see
\Cref{assforposets} and also \cite[Proposition 3.5]{MR4046069}). Curiously, this
view is incompatible with the common conception of chain complexes as
\emph{chain objects indexed by $\Z$, connected by differentials}, see
\Cref{warn:nonstrict}, but rather much closer to the concepts of twisted
complexes and pretriangulated hulls. In \hyperlink{thmA}{Main Theorem} the
notation $\Kb(\DS_n)$ refers the bounded homotopy category of chain complexes
built from $\DS_n$ with its inherited \emph{strict} monoidal structure. The
general strictness results we obtain are applicable in other contexts and should
be of interest independent of our paper. 

\subsection{Outline of the construction and byproducts}
The starting point for our construction is \Cref{thm:ssmtwocat} in \Cref{sec:ssmondiag} which states the following:  
\begin{byproduct}
    \label{byproda}
\emph{
The diagrammatic categories $\DS_n$ assemble into a semistrict monoidal
$2$-category\footnote{As spelled out in \cite[Lemma 4]{MR1402727}, equivalently, a monoid in the category Gray \cite[\S3, \S8]{Gurski}.}  $\DS$ with respect to an external tensor product $\boxtimes\colon
\DS_m\times \DS_{n}\to \DS_{m+n}$. }
\end{byproduct}
Very roughly, the
$\boxtimes$-product\footnote{
The almost tautologically good behaviour of $\boxtimes$ on $\DS$ reflects the
fact that the definitions of the various $\DS_n$ appear very homogeneous in the
parameter $n$. Arguably, it might be more natural to first construct $\DS$ as
free locally linear semistrict monoidal $2$-category on the set of objects
$\N_0$ by specifying generating $1$-morphisms, generating $2$-morphisms, and
certain linear relations on composite $2$-morphisms, and then to recover $\DS_n$
as endomorphisms of $n$.}
of two diagrammatic morphism changes the colors of the
second diagram by $m$ steps and then superposes the diagrams, see
\Cref{tensor1hom}. This color shifting is crucial for naturality. Algebraically, the analog of the color shifting is the following: given polynomial algebras $R_m=\mathbb{C}[X_1,\ldots, X_m]$ and $R_n=\mathbb{C}[X_1,\ldots,X_n]$ and an endomorphism $f$ of the algebra $R_n$, the tensor product  $1\boxtimes f$  endomorphism of $R_{m}\boxtimes R_n=R_{m}\boxtimes_\mathbb{C}R_n=\mathbb{C}[X_1,\ldots,X_{m+n}]$ satisfies the equality $(1\otimes f)(X_{i+m})=f(X_{i})$  involving the shift of indices by $m$ for  $1\leq i\leq n$. 

In \Cref{Strictificationproof} we replace the 
graded linear hom-categories $\DS_n$ by their pretriangulated hulls  $\Chb(\DS_n)$  and show: 
\begin{byproduct}
The pretriangulated hulls $\Chb(\DS_n)$ form the hom-categories  of a locally graded dg-$2$-category
$\Chbloc(\DS)$, in which the horizontal composition is still strict.
\end{byproduct}

 We then proceed to the homotopy categories hom-wise, to obtain the locally
linear $2$-category $\Kbloc(\DS)$, whose $1$-morphisms are chain complexes over
$\DS_n$ and $2$-morphisms are chain maps up to homotopy. Importantly, the
composition of $1$-morphisms is still strict. Extending the tensor product
$\boxtimes$ from \Cref{byproda}, we obtain a locally linear version of
\cite[Definition 2 and 3]{MR1402727}:
%This enables us 
% to reassemble in \Cref{thm:ssmtwocatb}  the $\Kb(\DS_n)$ for all $n$:
\begin{byproduct} 
The $\Kb(\DS_n)$ form the homomorphism categories of a semistrict monoidal locally linear $2$-category $\Kbloc(\DS)$.
\end{byproduct}
The proof of the \hyperlink{thmA}{Main Theorem} requires more, namely showing that the Rouquier complexes of shuffle
braids are natural in both arguments.

The main issue to be addressed here is that
in a $2$-categorical setting, the \emph{naturality} of the braiding is not a
property but \emph{additional structure}, subject to additional coherence conditions. 
Not only do we need to know that morphisms slide through the braiding, we
actually have to \emph{provide} specific $2$-isomorphisms that implement this sliding 
and are moreover compatible with other $2$-morphisms. In
the setting of $\Kbloc(\DS)$, where $1$-morphisms are chain complexes over
various $\DS_n$, we more specifically need 

\begin{itemize}
    \item the data of a \emph{slide chain map}  $\slide_{X_1,X_2}$ for every pair of $1$-morphisms $X_1,X_2$, i.e.~a homotopy equivalence that expresses that the
    $1$-morphisms $X_1$ and $X_2$ slide through the Rouquier braiding.
    \item the existence of a \emph{slide homotopy} $h_{f_1\boxtimes f_2}$ for every pair $f_1\colon X_1 \to Y_1$, $f_2\colon X_2
    \to Y_2$ of $2$-morphisms, i.e.~a homotopy that relates
    the chain maps given by applying $f_1$ and $f_2$ before and after sliding the target resp. source $1$-morphisms through the braiding.
    \end{itemize}
    We refer to  \Cref{fig:slidemapsandhomotopies} for a schematic illustration of these requirements. For readers familiar with
    link homology theories, these two levels (slide map and homotopy) roughly correspond to Reidemeister
    moves and Carter--Saito movie moves respectively.
   
\begin{figure}[h]
\[ 
\begin{tikzcd}[scale=1,column sep=2.5cm]
        \begin{tikzpicture}[anchorbase,xscale=-.3,yscale=.3,tinynodes]
            \draw[thick] (-2,2) to (1,2) \pr (4,0);
            \draw[thick] (-2,3) to (1,3) \pr (4,1);
            \draw[thick] (-2,4) to (1,4) \pr (4,2);
            \draw[wh] (1,1) \pr (3,4) to  (4,4);
            \draw[thick] (-1.5,1) to (1,1) \pr (3,4) to (4,4);
            \draw[wh] (1,0) to (2,0) \pr (4,3);
            \draw[thick] (-1.5,0) to (1,0) to (2,0) \pr (4,3);
            \draw[thick,fill=white] (0,-.2) rectangle (1.5,1.2);
            \draw[thick,fill=white] (-1.5,1.8) rectangle (0,4.2);
            \node at (.75,.3) {$X_1$};
            \node at (-.75,2.9) {$X_2$};
        \end{tikzpicture}
        \ar[r, "\slide_{X_1,X_2}"]
        \ar[d, swap, "\id \hcomp (f_1\boxtimes f_2)"]
    &
    \begin{tikzpicture}[anchorbase,xscale=.3,yscale=-.3,tinynodes]
        \draw[thick] (-1.5,2) to (1,2) \pr (4,0);
        \draw[thick] (-1.5,3) to (1,3) \pr (4,1);
        \draw[thick] (-1.5,4) to (1,4) \pr (4,2);
        \draw[wh] (1,1) \pr (3,4) to  (4,4);
        \draw[thick] (-1.5,1) to (1,1) \pr (3,4) to (4,4);
        \draw[wh] (1,0) to (2,0) \pr (4,3);
        \draw[thick] (-1.5,0) to (1,0) to (2,0) \pr (4,3);
        \draw[thick,fill=white] (0,-.2) rectangle (1.5,1.2);
        \draw[thick,fill=white] (-1.5,1.8) rectangle (0,4.2);
        \node at (.75,.7) {$X_1$};
        \node at (-.75,3.1) {$X_2$};
    \end{tikzpicture}
        \ar[d,"(f_2\boxtimes f_1)\hcomp \id"]
          \ar[dl,Rightarrow,shorten >=5ex,shorten <=5ex, "h_{f_1\boxtimes f_2}"]
\\
        \begin{tikzpicture}[anchorbase,xscale=-.3,yscale=.3,tinynodes]
            \draw[thick] (-2,2) to (1,2) \pr (4,0);
            \draw[thick] (-2,3) to (1,3) \pr (4,1);
            \draw[thick] (-2,4) to (1,4) \pr (4,2);
            \draw[wh] (1,1) \pr (3,4) to  (4,4);
            \draw[thick] (-1.5,1) to (1,1) \pr (3,4) to (4,4);
            \draw[wh] (1,0) to (2,0) \pr (4,3);
            \draw[thick] (-1.5,0) to (1,0) to (2,0) \pr (4,3);
            \draw[thick,fill=white] (0,-.2) rectangle (1.5,1.2);
            \draw[thick,fill=white] (-1.5,1.8) rectangle (0,4.2);
            \node at (.75,.3) {$Y_1$};
            \node at (-.75,2.9) {$Y_2$};
        \end{tikzpicture}
            \ar[r, "\slide_{Y_1,Y_2}"]
        &
        \begin{tikzpicture}[anchorbase,xscale=.3,yscale=-.3,tinynodes]
            \draw[thick] (-2,2) to (1,2) \pr (4,0);
            \draw[thick] (-2,3) to (1,3) \pr (4,1);
            \draw[thick] (-2,4) to (1,4) \pr (4,2);
            \draw[wh] (1,1) \pr (3,4) to  (4,4);
            \draw[thick] (-1.5,1) to (1,1) \pr (3,4) to (4,4);
            \draw[wh] (1,0) to (2,0) \pr (4,3);
            \draw[thick] (-1.5,0) to (1,0) to (2,0) \pr (4,3);
            \draw[thick,fill=white] (0,-.2) rectangle (1.5,1.2);
            \draw[thick,fill=white] (-1.5,1.8) rectangle (0,4.2);
            \node at (.75,.7) {$Y_1$};
            \node at (-.75,3.1) {$Y_2$};
        \end{tikzpicture} 
         \end{tikzcd}
        \]
\caption{Slide chain maps and slide homotopies.} 
\label{fig:slidemapsandhomotopies}
\end{figure}
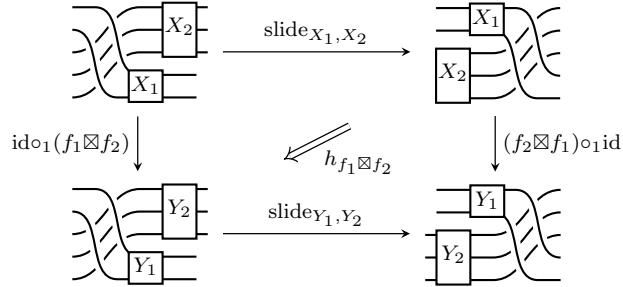

We take a two step approach to fulfill the requirements for the
naturality of the braiding on $\Kbloc(\DS)$. \hfill\\
\emph{Step 1:} We treat the case of $1$-morphisms from $\DS\hookrightarrow \Kbloc(\DS)$ in \Cref{sec:natprebraid}. \hfill\\
\emph{Step 2:} We extend the naturality to arbitrary
$1$-morphisms in $\Kbloc(\DS)$ in \Cref{sec:extensiontocomplexes}.

Constructing the slide maps for $1$-morphisms in $\DS\hookrightarrow \Kbloc(\DS)$ is straightforward, see \Cref{prop:sliding-objects}. Establishing the existence
of slide homotopies for $2$-morphisms in $\DS\hookrightarrow \Kbloc(\DS)$ requires more work and is done in \Cref{thm:naturality}. The proof of this \emph{First Naturality Theorem} is by an explicit
construction that  involves a \empty{reduction step} to the case of diagrammatic
generators, see \Cref{sec:slidereduction}, and then relies on \emph{explicit computations of
homotopies}, see \Cref{sec:slidegens}. As a result,  we get slightly more
than absolutely necessary:
\begin{byproduct}\label{byprochoices}
For $\DS\hookrightarrow \Kbloc(\DS)$, we provide \emph{explicit} slide chain maps and slide homotopies depending on some auxiliary
choices, see \Cref{exa:hiho}. 
\end{byproduct}

In \Cref{ssec:hhomotopies} we work specifically in the dg enhancement
$\Chbloc(\DS)$ of $\Kbloc(\DS)$ and interpret these homotopies as second layer
of a homotopy-coherent naturality structure of the braiding. In fact, the
mentioned dependence on auxiliary choices is trivialized by higher homotopies,
which are themselves canonically determined up to even higher homotopy ad
infinitum, as we prove in \Cref{ssec:hhomotopies} by a cohomology-vanishing
argument. 
\begin{byproduct} 
The explicit slide chain maps and homotopies from \Cref{byprochoices} are canonical up to coherent homotopy, and thus independent of auxiliary choices.
\end{byproduct}

For general $1$-morphisms in $\Kbloc(\DS)$, i.e. chain complexes over $\DS$, it seems difficult to 
construct explicit slide chain maps directly. By working in the homotopy-coherent
setting of $\Chbloc(\DS)$ instead, we describe a general procedure to construct the
slide chain maps for $1$-morphisms, the slide homotopies for $2$-morphisms, and
the full infinite hierarchy of associated higher homotopies in one sweep, see 
\Cref{thm:cabledcross-data}. Noteworthy, the construction of the slide map
for a chain complex of length $k+1$ requires the knowledge of the slide maps for the
chain objects as well as of the higher homotopies for the differential up to height
$k$, see \Cref{exa:cone}. 

Inspired by  \cite{MR1402727}, our approach is based on a centralizer construction on a dg level, more precisely  on a \emph{twisted $A_\infty$-monoidal centralizer} construction 
for pairs of parallel dg-monoidal functors which we introduce in
\Cref{def:twistedcentralizer}. This notion is inspired by the dg-center from \cite{GHW}. Afterwards, we realize the Rouquier
complexes of shuffle braids together with their naturality data as objects of such a twisted $A_\infty$-monoidal centralizer as a consequence of our \emph{Second Naturality Theorem}, stated as \Cref{thm:cabledcross-data}, in  \Cref{cormain}.
The extension of the naturally data from $\DS$ to $\Kbloc(\DS)$ then essentially
boils down to an extension along a deformation retract of $2$-sided bar
complexes for (monoidal) dg-categories, \Cref{prop:barcxdefr}. To summarize: 
%\begin{center}

\begin{byproduct}
The extension problem of the naturality data for the braiding from the additive category to the dg-category of complexes thereover is unobstructed and controlled by deformation retract data for $2$-sided bar
complexes of (monoidal) dg-categories.
\end{byproduct}

 In the concrete case of interest, we obtain:
\begin{byproduct}
    For all complexes in $\Chbloc(\DS)$ we provide explicit slide chain maps, slide homotopies, and an infinite hierarchy of higher homotopies that witness the homotopy-coherent naturality of the braiding by Rouquier complexes of shuffle braids.
\end{byproduct}

Finally, in \Cref{cor:slideandhomotopy} we again descend from $\Chbloc(\DS)$ to
$\Kbloc(\DS)$ and see that the slide chain maps witness the naturality of the braiding morphisms in, which proves \hyperlink{thmA}{Main
Theorem}. Finally, we observe that this provides almost all coherence conditions required for a braided monoidal
structure on $\Kbloc(\DS)$.

\subsection{Direct precursors of this work}
Our \hyperlink{thmA}{Main Theorem} has two direct precursors.
In \cite[Section 6]{MWW}, the second-named author with Morrison and Walker
observed that the braid (and tangle) invariants obtained in the context of the
$\glN$ link homology theory of Khovanov and Rozansky can be assembled into a
semistrict braided monoidal $2$-category. In this setting, the
axioms of a braided $2$-category are consequences of the functoriality of the
categorical tangle invariant under tangle cobordisms, which was shown in
\cite{MR3877770, EST}. The analogous observation also applies to the categorical
braid invariant provided by Rouquier complexes of Soergel bimodules (which are in fact complexes of Bott--Samelson bimodules), see
\cite[Remark~6.3]{MWW}. In this case, the axioms hold by the functoriality of
Rouquier complexes, which was verified by Elias--Krasner~\cite{MR2721032}. 
Our \hyperlink{thmA}{Main Theorem} is a substantial improvement of these results.

Indeed, the braided $2$-category from \cite[Remark~6.3]{MWW} is entirely combinatorial on
the level of objects and $1$-morphisms. The only allowed $1$-morphisms hereby are Rouquier
complexes of braids. This is in contrast to our setup, where the $1$-morphisms in $\Kbloc(\DS)$ are
\emph{arbitrary complexes of (diagrammatic) Bott-Samelson bimodules}.

The second precursor of the \hyperlink{thmA}{Main Theorem} appears in the recent work \cite{liu2024braided} of
the authors with Liu, Mazel-Gee, and Reutter. The main result there is the construction
of a fully homotopy-coherent braided-monoidal $(\infty,2)$-category
based on complexes of Soergel bimodules, summarized in the following two theorems from \cite{liu2024braided}.

\begin{thmA}[Monoidality Theorem]
\label[Monoidality Theorem]{oldthmA} 
    There is a monoidal $(\infty,2)$-category $\newKbloc(\SBim)$ with objects labelled by natural numbers $n \in \mathbb{N}_0$ and whose endomorphim $\infty$-categories are the $\k$-linear, stable, idempotent-complete $\infty$-categories of chain complexes over $\SBim_n$, with a $\Z$-action by grading shift.
\end{thmA} 
    \begin{thmB}[Braiding Theorem]
    \label[Braiding Theorem]{oldthmB} 
    There exists a braided monoidal (i.e., $\EE_2$-algebra) structure on $\newKbloc(\Sbim)$ that enhances its monoidal structure, such that the braiding $1 \otimes 1 \xra{\sim} 1 \otimes 1$ in $\newKbloc(\Sbim)$ admits an equivalence with the Rouquier complex corresponding to the positive braid group generator in $\Braidg_2$.
\end{thmB}

By taking the homotopy $2$-category of $\newKbloc(\Sbim)$, i.e. quotienting by all
$3$-morphisms, we obtain an $\EE_2$-monoidal $(2,2)$-category $\cH$, which we
expect to be essentially equivalent to $\Kbloc(\DS)$ from the \hyperlink{thmA}{Main Theorem}. For many purposes the differences are minor and negligible and one may freely pick the most suitable of the two approaches. 
However, we would like to point out the main differences between the two constructions:
\begin{itemize}
\item $\Kbloc(\DS)$ is based on the diagrammatic categories $\DS_n$ instead of
the (algebraic) Soergel bimodule categories $\SBim_n$. In fact, $\DS_n$ models
\emph{Bott-Samelson bimodules}, which yield Soergel bimodules \emph{after} the formal process of Cauchy
completion. %Sticking completely to Bott-Samelson bimodules simplifies, since it allows to work diagrammatically and explicitly. 
In contrast, the construction of $\cH$ \emph{requires} idempotent completeness from the start.
\item $\Kbloc(\DS)$ is realized as a \emph{semistrict} monoidal
$2$-category, i.e. the horizontal composition is \emph{strictly} associative and the
tensor product is as strictly associative as generically possible. In contrast,
$\cH$ is built in a world in which it does not even make sense to ask for such strictness
properties.
\item In $\Kbloc(\DS)$, all data associated to the braiding can
be exhibited in a \emph{completely explicit and algorithmically accessible way}. This
includes the $2$-morphisms that witness the naturality of the braiding, namely homotopy
equivalences that allow an arbitrary complex over $\DS_n$ to be slid through the
Rouquier complex of the $(m,n)$-shuffle braid, see \Cref{cor:slideandhomotopy}. This is in strong contrast to $\cH$, where the higher data is encoded in the language of $(\infty,2)$-categories and obtained from abstract lifting arguments.
\end{itemize} 

We moreover like to stress (the main reason which makes a rigorous comparison of the two approaches a non-trivial task)  that the construction of $\Kbloc(\DS)$ here follows an entirely different
strategy than the construction of $\cH$ in \cite{liu2024braided}. 

To  expand on this, we recall that the main tool for the constructions in
\cite{liu2024braided} is the monoidal fiber functor $\Hloc\colon \newKbloc(\Sbim)
\ra \stkBZ$ to the symmetric monoidal $\infty$-category $\stkBZ$ of small stable
idempotent-complete $\k$-linear $\infty$-categories. Informally, this functor
can be interpreted as taking the homology of complexes of Soergel bimodules or,
equivalently, as descending from the homotopy category to the  derived category
of chain complexes of (all) bimodules over the respective polynomial rings. This
latter category comes with a \emph{symmetric} braiding. In \cite{liu2024braided}, 
the higher data of this symmetric braiding on $\stkBZ$ gets pulled back along
the above fiber functor to a (non-symmetric!) braiding on $\newKbloc(\Sbim)$. A
crucial observation behind this different behaviour (symmetric versus non-symmetric) of the braiding is that
tensoring with the Rouquier complex associated to the positive shuffle braid in
$\Braidg_n$ associated to $w\in S_n$  simplifies drastically in $\stkBZ$.
Namely, in the derived category of all bimodules, it  just twists one of the two
polynomial ring actions by $w$, see \Cref{rem:Rouquierswap}. 

As a consequence, the complicated coherence data is \emph{inherited} and \emph{the
computational input for the construction of the braiding on $\newKbloc(\Sbim)$ is
minimal}. Namely,  it amounts to showing that the braiding is natural up to homotopy for
a restricted class of $1$-morphisms, but no higher morphisms. We encode this in 
\cite{liu2024braided} by the notion of a \emph{prebraiding} on
the functor $h_1 \BSBim_n\to h_1 \newKbloc(\Sbim)$, where $h_1$ expresses the
truncation to $1$-morphisms. This turns out to be sufficient to lift to a
braiding due to connectivity properties of the map  $\mathbb{A}_1\otimes
\EE_1\to \EE_2$ of $\infty$-operads, see \cite{liu2024mathbbenalgebrasmcategories} for a further study of such connectivity bounds.
An extension of Theorems A and B from \cite{liu2024braided} 
to Soergel bimodules valued in a connective $\mathbb{E}_\infty$-ring spectrum
with a complex orientation is obtained in \cite{liu2024braidingcomplexorientedsoergel}.

Our construction here follows a different route than \cite{liu2024braided}. Instead of importing higher
coherences via a somewhat abstract fiber functor, we now provide naturality
data by hand as described above.  In practice, instead of using the fiber functor
and checking a property at the level of $1$-morphisms (existence of certain
homotopy equivalences), we provide data at the level of $1$-morphisms, check a
property at the level of $2$-morphisms (existence of certain homotopies), and
then verify that the data necessary at the level of $2$-morphisms and higher
lives in a contractible space. 

\subsection{Related work}
Elements of naturality of the Rouquier braiding and related structures have
previously been observed and discussed in various contexts, including skein
algebra categorification \cite[Discussion after Conjecture 1.8]{1806.03416},
cabling operations in link homology \cite[\S~1.3]{GHW}, and wrapping and
flatting functors \cite[\S~6.4]{MR4526088},
\cite{elias2018gaitsgorys}. Closest to our work is the paper of
Mackaay--Miemimitz---Vay~\cite{MR4669324}, which appeared while our paper and
its sibling \cite{liu2024braided} were in preparation and which concerns
evaluation birepresentations of affine type A Soergel bimodules. Technical
ingredients in their construction are slide maps and the existence of slide
homotopies for diagrammatic Soergel bimodules through Rouquier complexes of
Coxeter braids, which share common consequences with our computations in
\Cref{sec:diags}. Higher homotopical data, which is necessary for extending this
data to complexes of Soergel bimodules, has not been considered in
\cite{MR4669324}. 

In the final stages of this work, we were informed by Ben Elias and Matthew
Hogancamp about their upcoming article \cite{EH24} on Drinfeld centralizers and
Rouquier complexes. This paper develops a theory of $A_\infty$-Drinfeld
centralizers\footnote{We independently introduce and use elements of such a
theory in \Cref{sec:extensiontocomplexes}} and, as an application, establishes
the centrality of Rouquier complexes of half- and full-twist braids. An analog
of our \Cref{cor:slideandhomotopy} can then be deduced \cite[Corollary
1.10]{EH24}. While the motivation and the setup of \cite{EH24} are different
from ours and the development independent, we expect that their results and ours
are compatible and complementary. In particular, the centrality of the full
twist suggests that a combination with our results here or in
\cite{liu2024braided} could lead to a framed $\mathbb{E}_2$-algebra structure on
complexes of type A Soergel bimodules, factorization homology over
\emph{oriented} surfaces, and thus homotopy-coherent wrapping and flattening
functors as indicated above.

\vspace{10pt}
\noindent \textbf{Acknowledgements.} The authors would also like to thank Yu
Leon Liu, Aaron Mazel-Gee, and David Reutter for many useful dicussions and for
the collaboration towards \cite{liu2024braided}, and Ben Elias and Matthew
Hogancamp for exchanging drafts of our papers. PW would like to thank Eugene Gorsky, Matthew Hogancamp, Scott
Morrison, and Kevin Walker for joint work, which has influenced the approach and
results in this paper, and David Penneys for useful discussions. CS would like
to thank Daniel Bermudez, Jonas Nehme and Liao Wang for valuable feedback on a
first draft.

\vspace{10pt}
\noindent \textbf{Funding.}
CS is supported by the Gottfried
Wilhelm Leibniz Prize of the German Research Foundation.
PW acknowledges support from the Deutsche Forschungsgemeinschaft (DFG, German
Research Foundation) under Germany's Excellence Strategy - EXC 2121 ``Quantum
Universe'' - 390833306 and the Collaborative Research Center - SFB 1624 ``Higher
structures, moduli spaces and integrability''.

Finally, the authors acknowledge the important role of the Spring 2020 MSRI programs ``Higher Categories and Categorification'' and ``Quantum Symmetries'' supported by the National Science Foundation grant DMS-1440140 and the 2019 Erwin--Schr\"{o}dinger institute workshop ``Categorification in quantum topology
and beyond'' in catalyzing their collaboration.

\vspace{10pt}
\noindent \textbf{Notation.}
We fix as ground field $\k=\C$. All vector spaces, linear maps, tensor products etc. are over $\k$ if not specified otherwise. Algebras are always  unital associative $\k$-algebra and algebra homomorphisms are unital. 

For $a,b\in \N_0$ we abbreviate $[a,b]:=\{n\in \Z\mid a \leq n\leq b\}$, which includes $[a,b]=\emptyset$ if $b<a$. We moreover denote by $S_n$ the symmetric group of order $n!$.

\section{Soergel bimodules (of type \texorpdfstring{$A$}{A})}
\label{S:SBim}

We start by recalling the monoidal graded $\k$-linear
categories $\SBim_n$ of Soergel bimodules, \cite{Soergel, Soergel2}, for $S_n$ acting
on its natural representation (a version for the reflection representation is
another possible but less natural choice). We also consider the diagrammatic
categories $\DS_n$ from \cite{EK, EW}, which present categories of Bott-Samelson bimodules in terms of generators and relations and allow
to recover $\SBim_n$ as Cauchy completion, see \Cref{def:SBim}.

We go one step further and show in both settings, algebraic and diagrammatic, that the Soergel
bimodules for all $n\in \N_0$ assemble into a monoidal bicategorical structure.

\subsection{Soergel bimodules and Bott--Samelson bimodules}
\label{ssec:Sbim}
For this subsection we fix  a nonegative integer $n$ and let $R_n=\C[x_1,x_2,\ldots, x_n]$ denote the
polynomial ring over $\C$ in $n$ variables viewed as polynomial functions on
$\C^n$ in the standard way. Permuting the basis vectors of $\C^n$ induces a left
action of the symmetric group $W=S_n$ on $R_n$ such that the simple
transposition $s_i:=(i,i+1)$ acts by swapping the variables $x_i$ and
$x_{i+1}$. Denote $\alpha_i=x_{i}-x_{i+1}$ for $1\leq i\leq n-1$. Then
restriction to the span of the $\alpha_i$'s gives the usual geometric
representation of $W$ viewed as Coxeter group generated by the simple
transpositions. For any subgroup $G$ of $W$ let $R_n^G$ be the subalgebra of
$G$-invariants in $R_n$. In case $G=\langle s_i\rangle$ for some $1\leq i\leq
n-1$ we abbreviate $R_n^G=R_n^i$. We will view $R_n$ as a graded (by which we
mean $\Z$-graded) algebra by putting the generators $x_i$ in degree $2$. Then $R_n^i$ is a graded subalgebra and we have a canonical decomposition $R_n=
R_n^i\oplus \alpha_iR_n^i\cong R_n^i\oplus R_n^i\langle 2\rangle$ as graded
$R_n^i$-bimodules. Here and in the following we denote, for a given graded (bi)module $M=\oplus_{i\in\Z} M_i$ and fixed $j\in \Z$,  
by $M\langle j\rangle$ the graded (bi)module which equals $M$ as (bi)module but with the grading shifted up by
$j$, i.e. $M\langle j\rangle_i=M_{i-j}$. The grading shifting functors $\langle
j\rangle$, $j\in\Z$ equip the category of graded $(R_m,R_n)$-bimodules for fixed
$m,n$ with an action of the group $\Z$. As homomorphisms between graded
bimodules we can consider homogeneous homomorphisms of arbitrary degree or only
those of degree zero. In the first case, we consider categories of graded
bimodules as enriched in graded $\k$-modules, i.e. as \emph{graded $\k$-linear}
categories. In the second case, we obtain a category enriched in
$\k$-modules, i.e. a \emph{$\k$-linear category}.

\begin{definition} We abbreviate $R=R_n$. For $\underline{i}=(i_1,\ldots, i_k)\in \intset{n-1}^k$ with $k \in \N_0$ and $j\in\Z$, the corresponding \emph{Bott--Samelson bimodule} is the graded $R_n$-bimodule
\begin{equation}
\label{eq:BS}
B_{\underline{i}}\langle
j\rangle:=R\otimes_{R^{s_{i_k}}}R\otimes_{R^{s_{i_{k-1}}}}\cdots\otimes_{R^{s_{i_1}}}R\langle
j-k\rangle.
\end{equation} 
If $k=0$ this is by convention the bimodule $R\langle
j\rangle$. In the case $k=1$ we abbreviate
\begin{equation}
\label{eq:BSi}
B_i:=     R\otimes_{R^{s_{i}}}R\langle-1\rangle.
\end{equation} 
\end{definition}
\begin{definition}
    \label{def:versionsBS}
For fixed $n\in \N$ we define the following categories attached to $n$:
\begin{itemize}
    \item  $\BSbimcl^\gr_n$, the \emph{graded category of Bott--Samelson
    bimodules}, is defined as the graded $\k$-linear full subcategory of graded
    $R_n$-bimodules with objects $B_{\underline{i}}\langle j\rangle$ as in
    \eqref{eq:BS} (for arbitrary $k$, $j$, $\underline{i}$). This category comes
    equipped with the $\Z$-action by grading shifts and a compatible graded $\k$-linear structure.
    \item $\BSbim_n$, the 
\emph{category of Bott--Samelson bimodules}, is obtained by restricting to the degree zero part of the morphism spaces in
    $\BSbimcl^\gr_n$. We call this the \emph{degree zero subcategory} or the \emph{usual category of Bott-Samelson bimodules}, and 
    consider it equipped with the $\Z$-action by grading shifts and the inherited $\k$-linear structure.
    \item $\overline{\BSbim}^\gr_n$, the \emph{category of
    unshifted Bott--Samelson bimodules}, is the graded $\k$-linear full subcategory of graded
    $R_n$-bimodules $B_{\underline{i}}$, where $j=0$ in \eqref{eq:BS}.
\end{itemize}
\end{definition}

The three variations of categories of Bott--Samelson bimodules reflect different ways of handling gradings. The version
$\overline{\BSbim}^\gr_n$ is most convenient when connecting to the diagrammatic
calculus of $\DS_n$, while $\BSbim_n$ features in the definition of Soergel
bimodules, and we will freely pass between these equivalent versions, see \cite[(2.1)]{MazStr}.

\begin{definition}
    \label{def:SBim}
The monoidal $\k$-linear category $\SBim_n$ of \emph{Soergel bimodules for $\mathfrak{gl}_n$}, see \Cref{SBimglsl} for the terminology, is
the Karoubian closure of $\BSBim_n$, i.e. the smallest additive,
itempotent-complete full subcategory of the $\k$-linear category of graded $R_n$-bimodules
and degree zero bimodule homomorphisms, which contains $\BSBim_n$. We consider $\SBim_n$ 
equipped with the $\Z$-action by grading shifts. 
\end{definition}

\begin{definition}
    Given a monoidal category $\mathcal{C}$ with unit $\mathbf{1}$, a \emph{monoid} in $\mathcal{C}$ is a triple $(A,\mm,\epsilon)$ of an object $A$ and morphisms $\mm\colon A\otimes A\to A$ and $\epsilon\colon \mathbf{1}\to A$ in $\mathcal{C}$ satisfying the associativity conditions for $\mm$ and the unit condition for the pair $(\epsilon, \mm)$. Dually,  a \emph{comonoid} in $\mathcal{C}$  is a triple $(A,\Delta,\eta)$ satisfying the opposite conditions. A \emph{Frobenius object} in $\mathcal{C}$ is the data $\mathbb{A}=(A,\mm,\Delta,\epsilon,\eta)$ of a monoid $(A,\mm,\epsilon)$ and a comonoid   $(A,\Delta,\eta)$ which are linked by the \emph{Frobenius condition} 
    \begin{equation}\label{Frob}
    (\id \otimes\mm)\circ(\Delta\otimes\id )=\Delta\circ\mm=\mm\otimes\id \circ(\id \otimes\Delta).
    \end{equation}
\end{definition}

\begin{lemma}\label{counitadj}
If $\mathbb{A}\in\mathcal{C}$ is a Frobenius object, then $A\otimes_-\colon\mathcal{C}\to\mathcal{C}$ is self-adjoint with unit and counit maps given by $\eta\circ\mm$ and $\Delta\circ\eta$ respectively.
\end{lemma}
\begin{proof}
The snake relations for the (co)unit of the adjunction follow directly from the Frobenius condition \eqref{Frob} and the definition of the (co)unit in $A$.
\end{proof}

Consider the Demazure operator $\partial_i\colon R_n\to R_n^{s_i}$, $f\mapsto \frac{f-s_i(f)}{x_i-x_{i+1}}$ for $i\in\intset{n-1}$. The definitions directly provide the following examples (note that $\Delta, \mm$ are of degree $-1$, whereas $\epsilon,\eta$ are of degree $1$ and the (co)unit of adjunction of degree zero).

\begin{lemma}\label{BSFrob}
    Each Bott-Samelson bimodule $B_i$, $i\in\intset{n-1}$ admits the structure of a Frobenius object in $\overline{\BSbim}^\gr_n$ with structure maps
    \[\mm\colon B_i\otimes_{R_n}B_i
%=R_n\otimes_{\R_n^{s_i}}R_n\otimes_{\R_n^{s_i}}R_n
\leftrightarrows B_i\colon\Delta \quad\text{and}\quad \eta\colon B_i\leftrightarrows  R_n
\colon\epsilon, \] 
given by the bimodule homomorphisms
\begin{equation*}
\begin{gathered}
\mm(f_1\otimes f_2\otimes g_1\otimes g_2)=f_1\partial_i(f_2g_1)\otimes g_2, \quad
\Delta(f\otimes g)=f\otimes 1\otimes 1\otimes g,\\
\quad 
\eta(f\otimes g)=fg, \quad
\epsilon(f)=\frac{1}{2}(\alpha_i\otimes 1 +1\otimes\alpha_i)f,
\end{gathered}
\end{equation*}
where $f,f_1,f_2, g, g_1,g_2\in R_n$.
\end{lemma}
\begin{proof}This is a straightforward check using the definitions and the property $\partial_i(fg)=\partial_i(f)g+s_i(f)\partial_i(g)=\partial_i(f)s_i(g)+f\partial_i(g)$ for any $i\in \intset{n-1}$, $f,g\in R_n$.
\end{proof}

\subsection{The monoidal bicategory of Soergel bimodules}
\label{sec:monbicat}
In this section, we assemble the Soergel bimodules for all $n\in\N_0$ into a locally $\k$-linear monoidal bicategory. The set of objects in this category will be $\N_0$ and the only nonzero hom-categories are the endomorphism categories; for $n\in \N_0$ this is $\SBim_n$ with composition being  $\hcomp=_-{\otimes_{R_n}}_-$.

    To this end, for each $n\in \N_0$ we consider the monoidal $\k$-linear
    category $\SBim_n$ as a locally $\k$-linear bicategory (short:
    $\k$-bicategory) with a unique object denoted $n$. The composition of
    1-morphisms is denoted $\hcomp$ and the composition of 2-morphisms $\vcomp$.
    For example, Soergel bimodules $M,M'$ in $\SBim_n$ are now considered as
    1-endomorphisms of the (unique) object $n$ with composition of 1-morphisms
    given by
     \[M\hcomp M':=M\otimes_{R_n} M'\] The composition operation $\circ_2$ is
    the usual composition of bimodule homomorphisms. 

\begin{definition}
\label{def:variableshift}
Given
$a,b,c\in\N_0$ we define the \emph{variable shifting morphisms} as the algebra homomorphism $\jj_{a|c}\colon R_b\hookrightarrow R_{a+b+c}$, $x_i\mapsto x_{i+a}$. 
\end{definition}
\begin{definition}[Variable shifting functors]
Given an $R_b$-bimodule $M$, we can consider  the tensor product $R_a\otimes_\k M \otimes_\k  R_c$ for any fixed $a,b\in\N_0$. 
This is an $R_a\otimes_\k 
R_b\otimes_\k  R_c$-bimodule, and hence a  $R_{a+b+c}$-bimodule via the algebra isomorphism
\begin{equation}\label{eqforshift}
R_a\otimes_\k R_b \otimes_\k R_c \xrightarrow{\jj_{0|b+c}\otimes \jj_{a|c} \otimes \jj_{a+b|c} } 
R_{a+b+c} \otimes_\k R_{a+b+c} \otimes_\k R_{a+b+c} \xrightarrow{\mathrm{multiply}} R_{a+b+c}
\end{equation}
This construction defines a  monoidal functor $\jj_{a|c}$ from $R_b$-bimodules to $R_{a+b+c}$-bimodules.
\end{definition}
Via \eqref{eqforshift}, we get induced monoidal functors:
\begin{equation}
\label{eq:embed}
    \jj_{a|c} \colon \BSbim_b \to \BSbim_{a+b+c}, \quad \jj_{a|c} \colon \Sbim_b \to \Sbim_{a+b+c}
\end{equation}
that act on the generating Bott-Samelson bimodules by a simple index-shift:
\begin{equation}
\label{eq:index-shift}
\jj_{a|c}(B_i)=B_{
i+a}
\end{equation}
\begin{definition}\label{monoidalstr}
Given an $R_m$-bimodule $M_m$ and an
$R_n$-bimodule $N_n$ we define their $\boxtimes$-product as the $R_{m+n}$-bimodule
\[M_m\boxtimes N_n := \jj_{0|n}(M_m) \hcomp \jj_{m|0}(N_n) \] 
and assign to a pair $f\colon M_m\to M'_m$, $g\colon N_n\to N'_n$ of bimodules morphisms  the morphism
\[f \boxtimes g:= \jj_{0|n}(f) \hcomp \jj_{m|0}(g) \colon M_m\boxtimes N_n \to M'_m\boxtimes N'_n\]
of bimodules. It is immediate that $\boxtimes$ is functorial in both arguments and sends pairs of Bott-Samelson bimodules to Bott-Samelson bimodules and likewise for Soergel bimodules. We call the resulting functors \emph{parabolic induction}:
\begin{equation}
    \label{eq:parabolicind}
    \boxtimes \colon \BSBim_m \times \BSBim_n \to \BSBim_{m+n}
    \quad\text{and}\quad
    \boxtimes \colon \SBim_m \times \SBim_n \to \SBim_{m+n} .
\end{equation}
\end{definition}

\begin{remark}
The functors $\boxtimes$ are $\k$-bilinear on the level of bimodule homomorphisms and respect the
$\Z$-action on bimodules in the sense that there are canonical isomorphisms:
\[(M_m\langle k \rangle) \boxtimes (M_n\langle l \rangle) \cong (M_m\boxtimes
M_n)\langle k+l \rangle.\]
    \end{remark}

    \begin{remark}
    The $\k$-module underlying $M_m\boxtimes N_n$ is nothing else than  $M_m\otimes_\k N_n$, but the notation $\boxtimes$ indicates the existence of additional structural morphisms. 
    \end{remark}
To make additional interchange morphisms for  $\boxtimes$ explicit, it is useful to factor the definition of the $\boxtimes$-product into two steps. First we define it when one factor is a regular bimodule
\[R_m\boxtimes N_n := \jj_{m|0}(N)\quad\text{respectively}\quad  M_m \boxtimes R_n := \jj_{0|n}(M), \]
and then extend to the general case using $\hcomp$:
\begin{equation}\label{Defbox}
M_m\boxtimes N_n := (M_m \boxtimes R_n) \hcomp (R_m\boxtimes N_n) = \jj_{0|n}(M) \hcomp \jj_{m|0}(N)
\end{equation}
In the second step we have chosen a particular order of the $\hcomp$-factors,
which is related to the other possible order via the \emph{tensorator} canonical isomorphism.
\begin{equation}
\label{eq:tensorator}
   \boxtimes_{M_m,N_n}\colon (M_m \boxtimes R_n) \hcomp (R_m\boxtimes N_n) \to (R_m\boxtimes N_n) \hcomp (M_m \boxtimes R_n).
\end{equation}
The tensorator is natural in $M_m$ and $N_n$ and yields isomorphisms that witness the more general \emph{interchange law}
\begin{equation}
    \label{eq:interchanglaw}
    (M_m\hcomp M'_m)  \boxtimes (N_n \hcomp N'_n) \cong (M_m\boxtimes N_n) \hcomp (M'_m\boxtimes N'_n).
\end{equation}

Its compatibility with $\hcomp$ is expressed by
the interchange isomorphisms \eqref{eq:interchanglaw} that we describe explicitly
for Bott-Samelson bimodules as follows.

\begin{lemma}[Interchange isomorphisms]
Let $k,l\in\N$, $m\in\intset{n-1}$ and $\underline{i}\in \intset{m}^k$, $\underline{j}\in [m+2,n]^l$. Then there exist a canonical isomorphism 
\begin{equation}
\label{BiBj}
\boxtimes_{\underline{i}, \underline{j}}\colon B_{\underline{i}}\hcomp B_{\underline{j}}\cong B_{\underline{j}}\hcomp B_{\underline{i}}.
\end{equation}
If, moreover, $k',l'\in\N$ and $\underline{i'}\in \intset{m}^{k'}$, $\underline{j'}\in [m+2,n]^{l'}$ and we denote concatenation of tuples by juxtaposition, such as $\underline{i}\underline{i'}$ and $\underline{j}\underline{j'}$, then the following equalities
hold:
\begin{align}\label{slideboxtimescomp}
\begin{split}
\boxtimes_{\underline{i}, \underline{j}\underline{j'}}=
    (\id\hcomp\boxtimes_{\underline{i}, \underline{j'}})\vcomp
    (\boxtimes_{\underline{i}, \underline{j}}\hcomp\id)\colon\quad &
    B_{\underline{i}}\hcomp B_{\underline{j}\underline{j'}}\to B_{\underline{j}\underline{j'}}\hcomp B_{\underline{i}},\\
     \boxtimes_{\underline{i}\underline{i'}, \underline{j}}=
(\boxtimes_{\underline{i}, \underline{j}}\hcomp\id)\vcomp
       (\id\hcomp\boxtimes_{\underline{i'}, \underline{j}})\colon\quad
       &B_{\underline{i}\underline{i'}}\hcomp B_{\underline{j}}\to B_{\underline{j}}\hcomp B_{\underline{i}\underline{i'}} .
       \end{split}
\end{align}
Given bimodule maps $\alpha\colon B_{\underline{i}}\to B_{\underline{i'}}$ and $\beta\colon B_{\underline{j}}\to B_{\underline{j'}}$ for some $i',j'\in\N$, we have 
\begin{equation}\label{slideboxtimes}
\boxtimes_{\underline{i}, \underline{j'}}\vcomp (\id\hcomp \beta)
=(\beta\hcomp \id)\vcomp\boxtimes_{\underline{i}, \underline{j}},
\quad
\boxtimes_{\underline{i'}, \underline{j}}\vcomp (\alpha\hcomp \id)
=(\id\hcomp \alpha)\vcomp\boxtimes_{\underline{i}, \underline{j}}
\end{equation}
\end{lemma}
\begin{proof}
Note that the $\underline{z}=1\otimes z_{i_1}\otimes\cdots\otimes z_{i_{k-1}}\otimes 1$ with $z_{i_t}\in\{1,\alpha_t\}$ generate $B_{\underline{i}}$ as bimodule. Moreover,  $z_{i_t}$ is $s_{j_t}$-invariant for $t\in\intset{n-1}$. Similarly, we have generators 
$\underline{z}'=1\otimes z_{j_1}\otimes\cdots\otimes z_{j_{l-1}}\otimes 1$ for $B_{\underline{j}}$. Using the invariance properties we see that their tensor products $\underline{z}\otimes\underline{z}'$, and $\underline{z}'\otimes\underline{z}$,
generate the bimodule $B_{\underline{i}}\hcomp B_{\underline{j}}$ and $B_{\underline{j}}\hcomp B_{\underline{i}}$ respectively. 
Using the invariance conditions one can verify that the assignment 
$\underline{z}\otimes\underline{z}'
$ on generators defines the desired isomorphism \eqref{BiBj}. The property \eqref{slideboxtimes} follows then directly.  
\end{proof}

In summary we can define the monoidal bicategory of Soergel bimodules: 
\begin{proposition}
    \label{prop:Sbimmonbicat}
The parabolic induction functors \eqref{eq:parabolicind} endow the \emph{bicategory of Soergel bimodules}
\begin{equation}\label{SBimDef}
\SBim:=\SBim_\mathfrak{gl}:= \bigsqcup_{n\geq 0} \SBim_n% \quad\text{and}\quad \DS:= \bigsqcup_{n\geq 0}\DS_n
\end{equation}
with the structure of a locally $\k$-linear monoidal bicategory with
monoidal product denoted $\boxtimes$. By restriction, the same result holds for the \emph{bicategory of Bott--Samelson bimodules} $\BSBim= \bigsqcup_{n\geq 0} \BSBim_n $.
\end{proposition}

We refer to \cite{schommer-pries-thesis} for the axioms of monoidal bicategories
and related structures, as well as a survey of the history of these notions.

\begin{remark} \label{SBimglsl} We call the bicategory defined in \eqref{SBimDef} the bicategory of
Soergel bimodules \emph{of type $\mathfrak{gl}_n$}, since we use the  polynomial
functions on the various $\mathfrak{h}_{\mathfrak{gl}_n}^\star=\C^n$, that is
the Cartan subalgebras for $\mathfrak{gl}_n$. One can similarly define the
\emph{bicategory of Soergel bimodules $\SBim_\mathfrak{sl}$ of type
$\mathfrak{sl}_n$}, by taking the algebras $\overline{R}_n\subset R_n$ of
polynomial functions on the corresponding Cartan subalgebras
$\mathfrak{h}_{\mathfrak{sl}_n}^\star=\C^{n-1}$. Note that
$\overline{R}_n=\k[\alpha_i\mid i\in\intset{n-1}]$, where
$\alpha\in\mathfrak{h}_{\mathfrak{gl}_n}^*$ are the simple roots. Since\footnote{Here we use that we work over a field of
characteristic zero.} 
$R_n=R_n'[x_1+\cdots + x_n]$ we can view the bicategory of Soergel bimodules of type $\mathfrak{gl}_n$ as a
central extension of the one of type $\mathfrak{sl}_n$; the morphism categories
are just obtained via the base change $_-\otimes_{R'_n}R_n$. 
\end{remark}

\begin{remark}
\label{rem:inmorita}
    An alternatively approach towards a monoidal $\k$-bicategory of Soergel bimodules can be based on the ambient monoidal bicategory of algebras, bimodules and bimodule maps. More precisely: associated to any braided monoidal category $\Vcat$ with coequalizers that
    are preserved by the tensor product in each variable separately, there
    exists, following  \cite[Section 8.9]{GPS}, the following monoidal bicategory $\mathrm{Alg}(\Vcat)$, that we call the \emph{Morita
    bicategory}:
    \begin{itemize}
        \item objects in $\mathrm{Alg}(\Vcat)$ are the monoids in $\Vcat$,
        \item $1$-morphisms from $A$ to $B$ in $\mathrm{Alg}(\Vcat)$ are $B$-$A$-bimodule objects in $\Vcat$,
        \item horizontal composition $\hcomp$ is by tensoring bimodules over intermediate monoids,
        \item $2$-morphisms in $\mathrm{Alg}(\Vcat)$ are the bimodule morphisms in $\Vcat$,
        \item the monoidal structure and composition of $2$-morphism are inherited from $\Vcat$.
    \end{itemize}
    Now we specialize to the case where $\Vcat=\gVect$, the
    category  of $\Z$-graded $\k$-vector spaces with trivial (not
    Koszul) braiding. The objects of $\mathrm{Alg}(\gVect)$ are graded algebras
    over $\k$, the $1$-morphisms are graded bimodules, and the $2$-morphisms are
    grading-preserving bimodule homomorphisms. Inside $\mathrm{Alg}(\gVect)$, we now build a $2$-full monoidal sub-bicategory $\BSBim'$ as follows:
    \begin{itemize}
        \item objects in $\BSBim'$ are the graded algebras $R_n=\k[x_1,\dots,x_n]$ for $n\in \N_{0}$,
        \item $1$-morphisms in $\BSBim'$ are generated inside
        $\mathrm{Alg}(\Vcat)$ under $\boxtimes$, $\hcomp$ and grading-shifts
        by the elementary Bott-Samelson bimodule $B_1:=R_2
        \otimes_{R_2^{S_2}}R_2 \langle 1\rangle$, and by the bimodules corresponding to the algebra isomorphisms $\jj_{0|n}\otimes
        \jj_{m|0}$ and their inverses for $m,n\in \N_0$,
          \item $2$-morphisms between $1$-morphisms in $\BSBim'$ are the same as in $\mathrm{Alg}(\gVect)$.
    \end{itemize}
    This is itself a monoidal bicategory, as the objects are closed under
    $\boxtimes$, the $1$-morphisms are closed under $\boxtimes$ and $\hcomp$,  and
    we have the associativity and unitality $1$-morphisms that are necessary for
    a monoidal bicategory. 
    
    We can also let the $1$-morphisms be
    generated by all grading shifts of Bott-Samelson bimodules directly to define a monoidal bicategory $\BSBim''$:
    \begin{itemize}
        \item $1$-morphisms in $\BSBim''$ are generated inside
        $\mathrm{Alg}(\Vcat)$ under $\boxtimes$ and $\hcomp$ by the grading-shifted Bott-Samelson bimodules $B_{\underline{i}}\langle
            j+k\rangle$ and the bimodules corresponding to the algebra isomorphisms $\jj_{0|n}\otimes
        \jj_{m|0}$ and their inverses for $m,n\in \N_0$.
    \end{itemize} 
    Since the tensor product of
    $\gVect$ preserves finite direct sums in both arguments, direct sums are
    also preserved in under $\boxtimes$ in $\mathrm{Alg}(\gVect)$ and so we
    obtain a locally additive and idempotent complete monoidal bicategory
    $\SBim''$ by replacing in the above:
    \begin{itemize}
        \item $1$-morphisms in $\SBim''$ are generated inside
        $\mathrm{Alg}(\Vcat)$ under $\boxtimes$ and $\hcomp$ by Soergel
        bimodules and the bimodules corresponding to the algebra isomorphisms
        $\jj_{0|n}\otimes \jj_{m|0}$ and their inverses for $m,n\in \N_0$.
    \end{itemize} 

    In $\BSBim'$, $\BSBim''$, and $\SBim''$, every object of the form $\boxtimes_{i=1}^m R_{n_i}$ is
    equivalent to some  $R_{\sum_i n_i}$ via variable shifting morphisms, and so the morphism
    categories in these three monoidal bicategories can be identified with morphism categories in $\BSBim$ and $\SBim$, respectively, which we use in this paper.

\end{remark}

\section{Various strictness results for dg-categories}
In this section we  establish certain strictness results for monoidal dg-categories. The Strictness Theorem, \Cref{cor:moninherit}, might be interesting on its own,  but will in particular be applied  to categories of (chain complexes of) Soergel or Bott--Samelson bimodules. Before we define these categories and study their monoidal structures we first consider monoidal structures on sets. 

To this end, we remind ourselves of a few basic facts on total orders on sets.
    Recall that for totally ordered sets $(X,\leq)$ and $(Y,\leq)$, the
    \emph{lexicographic order} on $X\times Y$ is defined by 
\[(x,y)\leq (x',y') 
\quad \iff\quad  
\left(
x<x' \;\text{ or }\; x=x, y<y' 
\right).
\] 
Similarly, one defines the lexicographic order on cartesian products of higher arity.

\begin{lemma}[Associativity of lexicographic order] \label{assforposets}
Assume we are given totally ordered sets $(X,\leq)$, $(Y,\leq)$, $(Z,\leq)$. Consider the canonical identifications of sets 
\begin{equation}\label{assposs}
(X\times Y)\times Z \rightarrow X\times Y \times Z \rightarrow X \times (Y\times
Z). 
\end{equation}
Then the iteratively defined lexicographic order on $(X\times Y)\times Z$ and  $X \times (Y\times
Z)$ agree with the lexicographic order on  $X\times Y \times Z $ and induce identifications \eqref{assposs} of posets.

\end{lemma}
\begin{proof}
For $x,x'\in X$, $y,y'\in Y$, $z,z'\in Z$, we use the definition to check:
\begin{align*}
    ((x,y),z)\leq ((x',y'),z')  
    \quad &\iff\quad  \left( 
    (x,y)<(x',y')     \;\text{ or }\; (x,y)=(x',y'), z<z' \right)
    \\
    \quad &\iff\quad  
        \left( 
    x<x'  \;\text{ or }\;  x=x', y<y' \;\text{ or }\;  (x,y)=(x',y'), z<z' 
    \right)
    \\
 \quad &\iff\quad  
          \left( 
        x<x'     \;\text{ or }\;     x=x', (y,z)\leq (y',z') 
        \right)
        \\
    \quad &\iff\quad  
    (x,(y,z))\leq (x',(y',z')).  
\end{align*}
Thus, the iteratively defined orders agree with the lexicographic order on $X\times Y\times Z$.
\end{proof}
Note that this result does not hold for arbitrary orderings on the product. As a consequence, the lexicographic ordering allows us to  use the Cartesian product to define a strict monoidal structure on ordered sets or a strict monoidal skeleton for the category of finite sets. This will now be used to construct a strict monoidal structure on categories of chain complexes built from objects in a strict monoidal category. 

\subsection{Revisiting dg-categorical constructions}\label{backgrounddg}
We start by collecting  background material on standard constructions with
    differential graded (dg-)categories. We assume familiarity with the
    definitions of dg-categories, functors and natural transformations, see e.g.
    \cite{BondalKapranov, Kellerondg}. The main ingredient that we need in the
    following sections is the construction of the dg-catgory of chain complexes over a
    $\k$-linear category in terms of the pretriangulated hull, see
    \Cref{newdefCb}.   
    We follow the conventions in
    \cite[\S 2-3]{GHW} and use the cohomological convention for the $\Z$-grading
    of chain complexes (in particular, the differential increases the degree).
    We denote by $\Hom^l$ the homogeneous component of cohomological degree
    $l\in \Z$ of a morphism complex. 

\begin{definition}\label{add}
    If $\CS$ is a dg-category, the \emph{additive suspended ordered envelope} $\adds\CS$
    of $\CS$ is the dg-category whose 
    \begin{itemize}
    \item objects are collections
    $\{\susp{a_i}X^i\}_{i\in I}$ where $I$ runs through the isomorphism classes of finite, totally ordered sets, 
    $X^i\in \CS$ and $a_i\in \Z$.    
   \item The morphism complexes in $\adds\CS$ are
        \[
        \Hom^l_{\adds\CS}\left(\{\susp{a_i}X^i\}_{i\in I},\{\susp{b_j}Y^j\}_{j\in J}\right) 
        = \prod_{i\in I}\bigoplus_{j\in J} \Hom^{l+b_j-a_i}_{\CS}(X^i,Y^j)
        \]
        with differential constructed from the differential $d_{\CS}$ of $\CS$ by
        \begin{equation}
        \label{eq:diff-susp-env}
        d_{\adds\CS}((f_{ji})_{(j,i)\in J\times I}) = ((-1)^{b_j}d_{\CS}(f_{ji}))_{(j,i)\in J\times I}.
        \end{equation}
    A morphism in this hom space can be considered as a $J\times I$ matrix
    $(f_{ji})$ of morphisms $f_{ji}\in \Hom_\CS(X^i,Y^j)$.  
    \item Composition of
    morphisms is given by matrix multiplication\footnote{We refer to \cite[\S 3]{MR2174270} for a beautiful explanation of this construction in the context of the additive (though not suspended) completion.} with multiplication of entries
    modelled on composition in $\CS$.
    \end{itemize}
    There is a canonical fully faithful dg-functor $\CS\rightarrow \adds\CS$ sending $X\mapsto \{X\}$ (with a singleton as
    indexing set), and we may identify $\CS$ with its image in $\adds\CS$.
\end{definition}
 
We formulate an obvious observation explicitly to stress the involved orders: 
\begin{lemma}\label{obviousbutcrucial}
The dg-category $\adds\CS$ is additive with biproducts of objects given by disjoint unions of collections such that 
the biproduct of an object indexed by $I$ with an object indexed by $J$ is indexed by $I\sqcup J$ equipped with the induced total order with $I<J$. 
\end{lemma}
    
 One can easily verify that $\adds\CS$ is a suspended category in the sense of \cite{KellerVossieck, KellerChain} (with the obvious suspension functor induced by $\Sigma$). 
       Applying the construction $\adds$ once again to $\CS\rightarrow \adds\CS$ yields an equivalence $\adds \CS \cong \adds\adds\CS$. Thus, $\adds$ is an idempotent operation. 
This justifies,  with \Cref{obviousbutcrucial},  our  terminology  \emph{additive suspended ordered envelope}.

Motivated by \Cref{obviousbutcrucial}, we use, in the setup from there, the notation
    \begin{equation}\label{orderedoplus}
    \bigoplus_{i\in I} \susp{a_i}X^i := \{\susp{a_i}X^i\}_{i\in I} \in \adds\CS,
    \end{equation}
    and also write $0$ for the empty direct sum (corresponding to the case
    $I=\emptyset$).
   
    \begin{definition}
        If $\CS$ is a dg-category, the \emph{twisted envelope} $\Tw(\CS)$ of
$\CS$ is the dg-category with objects $\tw_\a(X)$, where $X\in\CS$ and $\a\in
\End^1_\CS(X)$ satisfying $d_{\CS}(\a)+\a^2=0$, the \emph{Maurer--Cartan equation}. The morphism complexes
in $\Tw(\CS)$ are by definition
\[
\Hom_{\Tw(\CS)}\Big(\tw_\a(X),\tw_\b(Y)\Big) := \Hom_\CS(X,Y)
\]
with differential
$d_{\Tw(\CS)}(f) = d_{\CS}(f) + \b\circ f - (-1)^{|f|} f\circ \a$.
We say that $\CS$ \emph{has twists} if the fully faithful dg-functor
$\CS \rightarrow \Tw(\CS)$ sending $X\mapsto \tw_0(X)$ is an equivalence.
    \end{definition}

For us, the most important twists are cones: If $f\in \Hom^1_\CS(X,Y)$ is a
degree $1$ closed morphism, then the cone of $f$ gives the object $\tw_\a(X\oplus
Y)$ in $\Tw(\CS)$ with $\a=\smMatrix{0& 0\\ f & 0}$.

\begin{remark}
We consider, compatible with \eqref{orderedoplus}, the summands of $\tw_\a(X\oplus Y)$ as ordered. The order is induced by $f$, i.e.  $X=X_1$ and $Y=X_2$ with $1<2$.
\end{remark} 

\begin{remark}
The name twisted  \emph{envelope} is justified, as one can show that $\Tw(\CS)$ has
twists, so that the inclusion $\Tw(\CS)\rightarrow \Tw(\Tw(\CS))$ is an
equivalence,  see \cite{BondalKapranov, Drinfeld, Tabuada, Kellerondg}---also for background on $\Tw(\CS)$ and $\pretr{\CS}$ defined below. 
\end{remark}

The next definition is standard up to total orders on labelling sets built into $\adds\CS$:
\begin{definition}
    A dg-category $\CS$ is \emph{pretriangulated} if it is suspended and closed
    under the formation of cones. The \emph{pretriangulated hull} $\pretr{\CS}$ of a dg-category $\CS$ is the full subcategory of $\Tw(\adds\CS)$ generated by $\adds\CS$ under the formation of cones.
\end{definition}

\subsection{Bounded chain complexes}
\begin{definition}\label{newdefCb}
    Let $\CS$ be a $\k$-linear category. We define the \emph{dg-category of bounded
    chain complexes $\Chb(\CS)$ over} $\CS$ as the pretriangulated hull $\pretr{\CS}$ of $\CS$, considered as a dg-category with morphism complexes concentrated in degree zero. 
\end{definition}

\begin{remark}
Note that the objects of $\Chb(\CS)$ are exactly the bounded chain complexes
assembled from finite direct sums of homological grading shifts of objects of $\CS$, since every such chain complex is an iterated cone from objects of $\adds\CS$ 
   \end{remark} 
   
 Recall that in any dg-category, homogeneous morphisms in the kernel of the
    differential are called \emph{closed} and those in the image of the
    differential are called \emph{exact}. The exact morphisms form an ideal
    inside the closed morphisms. Two parallel morphisms are \emph{homotopic} if
    their difference is exact.  A closed degree zero morphism is a
    \emph{homotopy equivalence} if it is invertible up to homotopy. 
    The \emph{cohomology category} of a dg-category is defined to be the graded
    $\k$-linear category with the same objects, but with graded morphism spaces
    computed as the (co)homology of the morphism complexes. Its degree zero
    subcategory is called the \emph{homotopy category} of the dg-category. It is
    clear that homotopy equivalences in a dg-category are exactly the degree
    zero closed morphisms that induce isomorphisms in the homotopy category.

    \begin{exa}
        Let $\CS$ be a $\k$-linear category. The cohomology category of
        $\Chb(\CS)$ will be denoted $H^*(\Chb(\CS))$. It has the same objects as
        $\Chb(\CS)$, but with graded morphism spaces computed as the homology of
        the morphism complexes of $\Chb(\CS)$. The homotopy category of
        $\Chb(\CS)$ is also called the \emph{bounded homotopy category of $\CS$}
        and denoted $\Kb(\CS)$. Explicitly, the objects of $\Kb(\CS)$ are chain
        complexes and its morphisms are degree zero closed morphisms considered
        up to degree zero exact morphisms, i.e. chain maps up to homotopy. The
        notion of homotopy equivalence in the dg-category $\Chb(\CS)$ is
        equivalent to the common (classical) concept of the same name.
    \end{exa}

    Finally we recall a concrete model for the idempotent completion of a
    category $\CS$.

    \begin{definition}\label{Karconcrete} For a category $\CS$, the \emph{Karoubi envelope} $\Kar(\CS)$ is the category wherein
        \begin{itemize}
            \item objects are pairs $(c,p)$, where $c$ is an object of $\CS$ and $p\colon c\to c$ is an idempotent endomorphism in $\CS$, i.e.~ satisfying $p\circ p = p$,
            \item morphisms $(c,p) \to (d,q)$ in $\Kar(\CS)$ are morphisms $f\colon c \to d$ in $\CS$ with the property $f=f\circ p = q \circ f$,
            \item composition is inherited from $\CS$.
        \end{itemize}
        In case $\CS$ is graded we require idempotents to be  homogeneous of degree zero. The \emph{canonical embedding} $\CS\to \Kar(\CS)$ is the functor defined by  $c\mapsto (c,\id_c)$ on objects and by $f\mapsto f$ on morphisms. If this functor is an equivalence $\CS$ is 
        \emph{idempotent complete}.
    \end{definition}
    \begin{remark}
Note that  $\Kar(\CS)$ inherits structure, such as $\k$-linearity or dg, from $\CS$. 
\end{remark}
\subsection{Strictness for monoidal dg-categories}
\label{Strictificationproof}
It is a well-known fact that if $\CS$ is a monoidal $\k$-linear category, then
$\Chb(\CS)$ and $\Kb(\CS)$ inherit monoidal structures such that the respective
embeddings of $\CS$ are monoidal. In the following, we need a strengthening of
these statements for strict monoidal structures. 

\begin{proposition}
    \label{thm:moninherit}
    Suppose $\CS$ is a monoidal $\k$-linear category. Then the following hold:
    \begin{enumerate}
        \item $\adds\CS$ inherits the structure of a monoidal dg-category, in which the tensor product distributes over biproducts. 
    \item $\Chb(\CS)$ inherits the structure of a monoidal dg-category, in which the tensor product of complexes is given by the totalization of a double complex formed from term-wise monoidal products, using the usual Koszul sign rule for the differential. 
    \item $H^*(\Chb(\CS))$ inherits the structure of a monoidal graded $\k$-linear category. 
    \item $\Kb(\CS)$ inherits the structure of a monoidal $\k$-linear category. 
    \item $\Kar(\CS)$ inherits the structure of a monoidal dg-category.
       \end{enumerate}
       In all cases, the embedding of $\CS$ is monoidal. Analogous results hold for $\k$-bicategories.
    \end{proposition}
We will give a proof of this folklore result below. Our main point, however, is that these monoidal structure can be engineered to be strict whenever $\CS$ is strict:

\begin{theorem}[Strictness Theorem]
    \label{cor:moninherit}
In the setting of \Cref{thm:moninherit} we also have:       

If the monoidal structure on $\CS$ is strict, then the inherited monoidal structure on, respectively,  $\adds\CS$, $\Chb(\CS)$, $H^*(\Chb(\CS))$, $\Kb(\CS)$, $\Kar(\CS)$ from (1)-(5)  is strict. 

Analogous results hold for locally $\k$-linear $2$-categories.
\end{theorem}

To prove \Cref{thm:moninherit} and \Cref{cor:moninherit}, we need some
    bookkeeping for the indexing of direct sums to control the associativity of
    their tensor products. For this we recall \Cref{assforposets}. 

    \begin{lemma}\label{deftensorstrict}
Assume $\CS$ is a monoidal $\k$-linear category. Then the assignment 
  \begin{equation} 
    \label{eq:distr}
           \left( \bigoplus_{i\in I} \susp{a_i}X^i\right) \otimes \left( \bigoplus_{j\in J} \susp{b_j}Y^j\right):=
           \bigoplus_{(i,j)\in I \times J} \susp{a_i+b_j} (X^i\otimes Y^j), 
        \end{equation}
on the level of objects, extends to a dg-functor $\adds \CS\otimes \adds \CS\to
\adds \CS$. Here,  $I\times J$ is equipped with the lexicographic order and
$\adds \CS\otimes \adds\CS$ refers to the tensor product of dg-categories.  
 \end{lemma}
 \begin{proof}
We start by considering an analog of the assignment \eqref{eq:distr} on the
        level of morphism, which is modelled on the Kronecker product of
        matrices with the Koszul signs incorporated. More precisely, assume we
        are given a pair of matrix entries of morphisms  in $\adds\CS$, say
        $\susp{a_i}X^i \to \susp{a'_i}U^i$ of degree $a_i-a'_i$ and
        $\susp{b_j}Y^j \to \susp{b'_j}V^j$ of degree $b_j-b'_j$.  The
        corresponding pair of components $f_i\in \Hom_\CS(X^i,U^i)$ and $g_j\in
        \Hom_{\CS}(Y^j,V^j)$ yield the matrix entry $(-1)^{(b'_j-b_j)a_i} f_i
        \otimes g_j \in \Hom_\CS(X^i\otimes Y^j,U^i\otimes V^j)$, which we can
        interpret again as a morphism in $\adds\CS$ in the obvious way,  namely
        as $\susp{a_i+b_j} (X^i\otimes Y^j)  \to \susp{a'_i+b'_j} (U^i\otimes
        V^j)$ of degree $a_i+b_j -a'_i-b'_j$.  After this reformulation in terms
        of matrices, tensor product of morphisms is nothing else then the
        (signed version) of the Kronecker product.  The signs are built in so
        that the graded interchange law $(f'_i\otimes g'_j)\circ (f_i\otimes
        g_j) = (-1)^{(b_j'-b''_j)(a_i-a'_i)}(f'_i\circ f_i)\otimes (g'_j\circ
        g_j)$ holds for any other pair of components $f'_i\in \Hom_\CS(U^i,R^i)$
        and $g'_j\in \Hom_{\CS}(V^j,S^j)$ representing in $\CS$ morphisms
        $\susp{a'_i}U^i \to \susp{a''_i}R^i$ and $\susp{b'_j}V^j \to
        \susp{b''_j}S^j$. It is now clear that the assignments on pairs of
        objects and pairs of morphism in $\CS$ descend to a dg-functor $\adds
        \CS\otimes \adds \CS\to \adds \CS$.
         \end{proof}

    \begin{proof}[Proof of \Cref{thm:moninherit} and \Cref{cor:moninherit}] We will only prove the case of
    monoidal categories, the generalization to bicategories being analogous. We treat each case (1)-(5).
    
        (1) By  \Cref{deftensorstrict}, the tensor product
        functor $\CS\otimes \CS \to \CS$ can be extended to a dg-functor $\adds \CS\otimes
        \adds \CS\to \adds \CS$. The monoidal unit in $\adds\CS$ will be the image
        of the unit ${\mathbf 1}$ of $\CS$ under the canonical embedding. As unitors in
        $\adds\CS$ we use diagonal matrices, whose entries are the unitors from
        the monoidal structure of $\CS$. By the associativity of forming
        lexicographic orders on triple products of indexing sets $I\times J
        \times K$ given by \Cref{assforposets}, we also see that we can take
        diagonal matrices of associators from $\CS$ as associators in
        $\adds\CS$. The checks of triangle and pentagon axioms are now
        straightforward, namely componentwise for diagonal matrices. \Cref{cor:moninherit}
        holds for  the now constructed monoidal structure on $\adds\CS$, since the triviality of
        unitors and associators is inherited.

        (2) We extend the monoidal structure on $\adds\CS$ to the
        twisted envelope $\Tw(\adds\CS)$.  Namely, we define on objects   
         \[\tw_\a(X)\otimes \tw_\b(Y):= \tw_{\a\otimes \id+\id\otimes
        \b}(X\otimes Y),\] noting  that $\a\otimes \id+\id\otimes \b\in
        \End^1(X\otimes Y)$ satisfies the Maurer--Cartan equation. The tensor
        product on morphisms is now easy to define: it is determined already by
        that on  $\adds\CS$; a straightforward check shows that it induces
        chain maps on morphism complexes with respect to the desired perturbed
        differential $d_{\Tw(\adds\CS)}$.  If $\CS$ and $\adds\CS$ are strict,
        then the strictness of the tensor product is inherited by
        $\Tw(\adds\CS)$, since it acts associatively on Maurer--Cartan elements.
        By definition, the (strict) monoidal structure now extends to the
        pretriangulated hull $\Chb(\CS)$ proving \Cref{cor:moninherit} for $\Chb(\CS)$.
        
        (3) and (4) are now straightforward: the (strict) monoidal structure on $\Chb(\CS)$ descends to $H^*(\Chb(\CS))$ and $\Kb(\CS)$ since the tensor products of closed morphisms are closed and exact morphisms form an ideal under tensoring with closed morphisms. Obviously, the embedding functor from $\CS$ is monoidal in each case.

        (5) we extend the monoidal structure from $\CS$ to $\Kar(\CS)$ as
         follows. For objects we set $(c,p)\otimes_{\Kar(\CS)} (d,q):= (c\otimes
         d, p\otimes q)$ and for morphisms $f\otimes_{\Kar(\CS)} g := f\otimes
         g$. The associators and unitors of $\CS$ can be reused for $\Kar(\CS)$ using \Cref{Karconcrete}:
         they only depend on the object-part $c$ of the pairs $(c,p)$ in each
         argument and represent morphisms in the Karoubi envelope by naturality.
         The triangle and pentagon relations are easily checked.      
     It is also clear that the inherited monoidal structure on $\Kar(\CS)$ will be
        strict whenever the original one on $\CS$ was strict. This also completes the proof of \Cref{cor:moninherit}.
    \end{proof}

    \begin{warn}
        \label{warn:nonstrict}
        The strictness of the monoidal structures of the first four cases in \Cref{cor:moninherit}  heavily relies on the associativity of the lexicographic orders chosen for
        the ordering of direct sums. For other orders, one may be forced to use
        non-trivial permutation matrices of $\CS$-associators instead of
        diagonal matrices as associators which might prevent strictness even if $\CS$
        was strict. 
        
        In particular, when tensoring chain complexes, one should \emph{not}
        pick a total order that refines the order by cohomological degree. To
        see this, let us consider three complexes $X$, $Y$, $Z$ and---only for
        the purpose of this warning---denote by $X^i$ the part of cohomological
        degree $i$ in $X$. Then the part of cohomological degree $i$ in the two
        possible triple products are computed respectively as:
        \begin{align*}
        ((X\otimes Y)\otimes Z)^i &= \bigoplus_{a\in \Z} \bigoplus_{b\in \Z} (X^b \otimes Y^{a-b})\otimes Z^{i-a},\\
       \quad \text{respectively}\quad
        (X\otimes (Y\otimes Z))^i &= \bigoplus_{a\in \Z} \bigoplus_{b\in \Z} X^a \otimes (Y^{b}\otimes Z^{i-a-b}).
        \end{align*}
        The second direct sum is manifestly ordered lexicographically in the
        indices, but in the first sum $(X^1 \otimes Y^{0})\otimes Z^{1}$ would
        appear before $(X^0 \otimes Y^{2})\otimes Z^{0}$. 
    \end{warn}

    \begin{remark}
        Analogous statements as in \Cref{thm:moninherit} and \Cref{cor:moninherit} hold for the (local) additive completion of monoidal $\k$-linear categories (resp. bicategories), and strictness can be preserved, see e.g. \cite[5.1.12-5.1.13]{Penneysnotes} for the statement. The proof is analogous (but easier) to the case $\adds\CS$, since no shifts have to be accounted for. Strictness of locally passing to pretriangulated hulls in locally finite $p$-dg $2$-categories was already observed in the proof of \cite[Proposition 3.5]{MR4046069}.
     \end{remark} 

\subsection{Monoidal locally dg-bicategories associated to Soergel bimodules} 
   Suppose that $\CS$ is a monoidal $\k$-bicategory with monoidal product $\boxtimes$. Then the monoidal structure can be extended to the bicategories obtained in \Cref{thm:moninherit} by arguments as in \Cref{Strictificationproof}.  
      
   We keep the discussion here brief since a more relevant stricter version is the subject of \Cref{thm:ssmtwocatb}, but like to mention the following constructions:
    \begin{itemize}
    \item A monoidal locally dg-bicategory $\Chbloc(\CS)$, by performing the construction $\Chb$ on the $\k$-linear hom-categories of $\CS$. The $\boxtimes$-product of $1$-morphisms is given by the totalization of a double complex formed from term-wise $\boxtimes$ products of $1$-morphisms in $\CS$. 
    \item A monoidal locally graded $\k$-linear bicategory $H^*(\Chbloc(\CS))$, by applying $H^*$ hom-category-wise. 
    \item A monoidal locally $\k$-bicategory $\Kbloc(\CS)$, by applying $\Kb$ hom-category-wise. 
    \end{itemize}
%\end{remark}
For a thorough discussion of the concept of dg-bicategories, see \cite[\S 4.1]{gyenge2021heisenbergcategorycategory}. 

These constructions in particular apply to the monoidal locally $\k$-linear bicategory $\CS=\SBim$ (and also $\CS=\BSBim$) from \Cref{prop:Sbimmonbicat} with its compositions $\vcomp, \hcomp, \boxtimes$.     
 Concretely, for each object $n$ of $\CS$, the endomorphism category $\SBim_n$ is a monoidal $\k$-linear category that we now consider as a monoidal dg-category with morphism complexes concentrated in degree zero. The compositions $\vcomp, \hcomp$ of $\SBim_n$ extend to the pretriangulated hull $\pretr{\SBim_n}=\Chb(\SBim_n)$, see \Cref{newdefCb}, which thus inherits the structure of a monoidal dg-category. Collecting all objects $n$ again, we obtain the locally dg-bicategory $\Chbloc(\SBim)$, which now additionally inherits a monoidal structure from $\boxtimes$. The $1$-morphisms in $\Chbloc(\SBim)$ are the bounded chain complexes\footnote{Note that we model the formation of chain complexes via the pretriangulated hull to obtain the strictness results from \Cref{cor:moninherit}!} of Soergel bimodules for $\mathfrak{gl}_n$ for all $n\in \N_0$. By passing locally to the homotopy categories we obtain $\Kbloc(\SBim)$.  Hence we defined the following crucial players.
 \begin{definition} We use the following notation and terminology:
    \begin{itemize}
   \item $\Chb(\SBim_n)$ the monoidal \emph{dg-category of chain complexes
    of Soergel bimodules for $\mathfrak{gl}_n$},
    \item $\Chbloc(\SBim)$ with its monoidal structure induced by $\boxtimes$ the
    \emph{monoidal dg-bicategory of chain complexes of Soergel bimodules},
    \item $\Kb(\SBim_n)$ the monoidal \emph{homotopy category of
    Soergel bimodules for $\mathfrak{gl}_n$}, 
    \item $\Kbloc(\SBim)$ the \emph{monoidal homotopy bicategory of
    Soergel bimodules}.
    \end{itemize}
    Similarly, for $\BSbim$ instead of $\Sbim$.
    \end{definition}
\begin{remark}\label{rkdelicate}
We can spell out the data of a monoidal bicategory equivalent to $\Kbloc(\SBim)$: the objects are indexed by natural numbers, $1$-morphisms are chain complexes of Soergel bimodules, $2$-morphisms are chain maps up to homotopy, and the monoidal product $\boxtimes$ acts on objects by $m\boxtimes n=m+n$ and on $1$- and $2$-morphism by tensoring over $\k$.  

Our constructions in \Cref{backgrounddg}
 with the resulting Strictness Theorem, \Cref{cor:moninherit}, imply a stricter version of $\Kbloc(\SBim)$ if we replace the Soergel bimodules categories by equivalent strict ones. This will be done in  \Cref{thm:ssmtwocatb} using the diagrammatical presentations from 
\Cref{thm:equivalence} with \Cref{rem:diagtosbim}.
\end{remark}
 
In \cite[Section 6]{liu2024braided}, $\infty$-categorical incarnations of the (monoidal) bicategories of Bott--Samelson bimodules and Soergel bimodules are constructed by following a strategy analogous to that of Remark~\ref{rem:inmorita}. To make the connection, note that we can apply \cite[Construction 1.3.1.6]{HA} to a(n idempotent complete)  dg-category $\Chb(\CS)$ as above and  
obtain  the \emph{dg-nerve} of $\Chb(\CS)$, which is a ($\k$-linear and
idempotent-complete) stable $\infty$-category $\newKb(\CS)=N_{\dg}(\Chb(\CS))$, see \cite[Definition 3.4.1]{liu2024braided}.
Taking its $1$-truncation recovers  the usual homotopy category $\Kb(\CS)$. This construction applies to the hom-categories $\SBim_n$ in $\SBim$. The \hyperlink{oldthmA}{Monoidality Theorem}, formulated in the introduction,  provides, in a fully homotopy coherent way, a monoidal $(\infty,2)$-category whose
$(2,2)$-truncation recovers $\Kbloc(\SBim)$ in the version from \Cref{rkdelicate}, see \cite[Proposition 6.4.2.]{liu2024braided},  together with the bisubcategories $\Sbim$ and $\BSbim$,  \cite[\S6.1, \S6.3]{liu2024braided}. 
In all these  cases the monoidal structure is given by parabolic induction. 
     
\section{Semistrict monoidal 2-categories via diagrammatics}
In this section we apply the Strictness Theorem, \Cref{cor:moninherit}, to Soergel bimodules. To be able to do so, we need a strictified version of the monoidal category of Soergel bimodules. For this 
we consider the diagrammatic calculus for Bott--Samelson bimodules, which was
developed by Elias--Khovanov~\cite{EK} and Elias--Williamson~\cite{EW}. Importantly, this framework 
allows us to do explicit calculations in $\overline{\BSbim}^\gr_n$ for each $n\in \N_0$, which will 
be used to compute higher homotopies. 
As a main result we assemble these categories (respectively homotopy categories thereof) into a semistrict monoidal $2$-category, see \Cref{defsemistrict}.
We also explain how to extend all this, via a concrete model of Karoubi envelope, from Bott--Samelson to Soergel bimodules, \Cref{rem:diagtosbim}.

\subsection{Diagrammatic calculus for Bott--Samelson bimodules}
\label{sec:diag}
The Elias--Khovanov--Williamson diagrammatic calculus is a version of
the usual string diagram calculus,  see e.g. \cite{TVbook},  for pivotal monoidal categories obtained by
choosing a collection of generating objects and of generating morphisms represented
by string diagrams. The strings will be labeled by generating objects, which
correspond to the elementary Bott--Samelson bimodules $B_i$ from \eqref{eq:BSi} or, equivalently, to
simple transpositions $s_i$. In the diagrammatic calculus we work with strictly pivotal and strictly monoidal categories which allows us to   
suppress the associators, unitors and the pivotality isomorphisms.  

 \begin{definition} \label{Def:Dn}
    For fixed $n\in \N_0$ consider the strictly pivotal and strictly
monoidal graded $\k$-linear category $\DS_n^{\free}$, freely generated by $n-1$
Frobenius algebra objects $(c_i, \mm_i,\Delta_i,\epsilon_i,\eta_i)$,
$i\in\intset{n-1}$, where the (co)multiplication and (co)unit morphisms are homogeneous
of degree $1$ and $-1$ respectively, and by additional generating morphisms
\begin{align}
\label{gen1}
  c_{i}\hcomp c_{j}\hcomp c_i &\to c_{j}\hcomp c_i\hcomp c_{j} \, &&\text{if } |i-j|=1,\\
 \label{gen2}   c_i \hcomp c_j &\to c_j \hcomp c_i\, &&\text{if } |i-j|>1,\\
\label{gen3}    x_k\colon  \one & \to \one,
\end{align} 
for
$i,j\in\intset{n-1}$, $k\in \intset{n}$,
of respective degrees $0$, $0$ and $2$. We denote here the monoidal product by $\hcomp$. 
\end{definition}

\begin{remark} \label{xcomm}
Note that $x_ix_j=x_jx_i$, since the endomorphism algebra of $\one$ is commutative, \cite[Proposition~2.2.10]{EGNO}.
\end{remark}
To get a diagrammatic description we encode the objects $c_i$ for $i\in\intset{n-1}$ by colors. Since we will soon consider
parabolic induction functors, where we will allow $n$ to vary, we choose now once and for all a countable sequence 
\begin{equation}
\label{eq:colorsequence}
    \{c_1,c_2,c_3,c_4,c_5, \dots\} :=
\{\textrm{\textcolor{red}{red}},\textrm{\textcolor{blue}{blue}},\textrm{\textcolor{green}{green}},
\textrm{\textcolor{orange}{orange}},\textrm{\textcolor{purple}{purple}},\dots\}
\end{equation}
of such colors. The generating morphisms in $\DS_n^{\free}$ are displayed by the following diagrams (read from bottom to top with vertical stacking as composition $\vcomp$ and with monoidal structure $\hcomp$ displayed by horizontal juxtaposition): 
\begin{equation}\label{diaggenerators}
    \begin{tikzpicture}[anchorbase,scale=.75]
        \node at (0,0) {
           \begin{tikzpicture}[anchorbase,scale=.2]
            \draw[bl] (1,-0.01) \uur (2,1.) (3,-0.01) \uul (2,1.) to (2,2.01); 
            \end{tikzpicture}
        ,\;
            };
        \node at (1,0) {
           \begin{tikzpicture}[anchorbase,scale=-.2]
            \draw[bl] (1,-0.01) \uur (2,1.) (3,-0.01) \uul (2,1.) to (2,2.01); 
            \end{tikzpicture}
        ,\;
            };
        \node at (2,0) {
           \begin{tikzpicture}[anchorbase,scale=.2]
            \draw[bl] (0,-0.01) \pu (0,-1.01);
            \fill[bl] (0,-1) circle (2.5mm); 
            \end{tikzpicture}
            ,\;
        };
        \node at (3,0) {
           \begin{tikzpicture}[anchorbase,scale=-.2]
            \draw[bl] (0,-0.01) \pu (0,-1.01);
            \fill[bl] (0,-1) circle (2.5mm); 
            \end{tikzpicture}
        ,\;
            };
        \node at (4,0) {\sixv ,};
        \node at (5.2,0) {
           \begin{tikzpicture}[anchorbase,scale=-.2]
            \draw[bl] (0,-0.01) \pd (2,-2.01);
            \draw[gr] (2,-0.01) \pd (0,-2.01); 
            \end{tikzpicture}
            ,\;
            };
            \node at (6.2,0) {
           \begin{tikzpicture}[anchorbase,smallnodes,scale=-.2]
            \draw (0,-0.01) rectangle (2,2.01);
            \node at (1,1) {$x_k$};
            \end{tikzpicture}
            };
        \end{tikzpicture}
\end{equation}
Here, the shapes of the first four diagrams (\emph{merge, split, startdot, enddot}) represent the multiplication, the comultiplication, the unit and the counit respectively for the Frobenius object of the given color (exemplified here for the color blue), and the other three diagrams display the additional types \eqref{gen1}-\eqref{gen3} of generating morphisms. Note that the six-valent vertex exists in two mirror versions, 
$\sixv$ and $\rotatebox[origin=c]{180}{\sixv}$, 
 for any two adjacent colors (here displayed using red and blue); similarly the crossing exists in two versions, $ \begin{tikzpicture}[anchorbase,scale=-.2]
            \draw[bl] (0,-0.01) \pd (2,-2.01);
            \draw[gr] (2,-0.01) \pd (0,-2.01); 
            \end{tikzpicture}$ and ${\begin{tikzpicture}[anchorbase,scale=-.2]
            \draw[gr] (0,-0.01) \pd (2,-2.01);
            \draw[bl] (2,-0.01) \pd (0,-2.01); 
            \end{tikzpicture}},$ 
             for any two distant colors (here: blue and green).

Note that the defining relations for a Frobenius algebra object (here of color blue) are 
\begin{equation}
    \label{eq:Drel}
        \begin{tikzpicture}[anchorbase,xscale=.2,yscale=.2]
        \draw[bl] (1,1.01) \uur (2,2.) (2,-0.01) to (2,3.) ; 
        \fill[bl] (1,1) circle (2.5mm); 
        \end{tikzpicture} 
        =
        \begin{tikzpicture}[anchorbase,xscale=-.2,yscale=.2]
        \draw[bl] (1,1.01) \uur (2,2.) (2,-0.01) to (2,3.) ; 
        \fill[bl] (1,1) circle (2.5mm); 
        \end{tikzpicture}  
        =
        \begin{tikzpicture}[anchorbase,xscale=.2,yscale=-.2]
        \draw[bl] (1,1.01) \uur (2,2.) (2,-0.01) to (2,3.) ; 
        \fill[bl] (1,1) circle (2.5mm); 
        \end{tikzpicture} 
        =
        \begin{tikzpicture}[anchorbase,xscale=-.2,yscale=-.2]
        \draw[bl] (1,1.01) \uur (2,2.) (2,-0.01) to (2,3.) ; 
        \fill[bl] (1,1) circle (2.5mm); 
        \end{tikzpicture}  
        =
        \begin{tikzpicture}[anchorbase,xscale=.2,yscale=.2]
        \draw[bl] (2,-0.01) to (2,3.); 
        \end{tikzpicture}
        \quad
        ,
        \quad
        \begin{tikzpicture}[anchorbase,xscale=.2,yscale=.2]
        \draw[bl] (1,-0.01) \uur (2,1.) (3,-0.01) \uul (2,1.) \uur (3,2.01) (5,-0.01)   \uul (3,2.) to (3,3.01); 
        \end{tikzpicture} 
        =
        \begin{tikzpicture}[anchorbase,xscale=-.2,yscale=.2]
        \draw[bl] (1,-0.01) \uur (2,1.) (3,-0.01) \uul (2,1.) \uur (3,2.01) (5,-0.01)   \uul (3,2.) to (3,3.01);
        \end{tikzpicture}    
        \quad
        ,
        \quad
        \begin{tikzpicture}[anchorbase,xscale=.2,yscale=-.2]
        \draw[bl] (1,-0.01) \uur (2,1.) (3,-0.01) \uul (2,1.) \uur (3,2.01) (5,-0.01)   \uul (3,2.) to (3,3.01);
        \end{tikzpicture} 
        =
        \begin{tikzpicture}[anchorbase,xscale=-.2,yscale=-.2]
        \draw[bl] (1,-0.01) \uur (2,1.) (3,-0.01) \uul (2,1.) \uur (3,2.01) (5,-0.01)   \uul (3,2.) to (3,3.01);
        \end{tikzpicture}
        \quad
        ,
        \quad
        \begin{tikzpicture}[anchorbase,xscale=.2,yscale=.2]
        \draw[bl]   (2,-0.01) to (2,1) \ulu (1,2) to (1,3) 
                    (2,1) to [out=45,in=225] (3,2)  
                    (4,-0.01) to (4,1) \uul (3,2) to (3,3); 
        \end{tikzpicture} 
        =
        \begin{tikzpicture}[anchorbase,xscale=.2,yscale=.2]
        \draw[bl] (1,-0.01) \uur (2,1.) (3,-0.01) \uul (2,1.) to (2,2.01) \ulu (1,3) (2,2.01) \uru (3,3); 
        \end{tikzpicture}  
        = 
        \begin{tikzpicture}[anchorbase,xscale=-.2,yscale=.2]
            \draw[bl]   (2,-0.01) to (2,1) \ulu (1,2) to (1,3) 
                        (2,1) to [out=45,in=225] (3,2)  
                        (4,-0.01) to (4,1) \uul (3,2) to (3,3); 
            \end{tikzpicture} .
\end{equation}

\begin{notation}
Boxes decorated by arbitrary polynomials $p\in \k[x_1,\dots,x_n]$ are our shorthand notation for the evident linear combinations of products of generators from \eqref{gen3} using \Cref{xcomm}, and we use the following caps and cups
\[\begin{tikzpicture}[anchorbase,smallnodes,scale=-.2]
        \draw (0,-0.01) rectangle (2,2.01);
        \node at (1,1) {$p$};
        \end{tikzpicture}
        := 
        p(\begin{tikzpicture}[anchorbase,smallnodes,scale=-.2]
            \draw (0,-0.01) rectangle (2,2.01);
            \node at (1,1) {$x_1$};
            \end{tikzpicture},\ldots, 
            \begin{tikzpicture}[anchorbase,smallnodes,scale=-.2]
            \draw (0,-0.01) rectangle (2,2.01);
            \node at (1,1) {$x_n$};
            \end{tikzpicture})
                    \quad, \quad
    \begin{tikzpicture}[anchorbase,scale=.2]
        \draw[white] (2,2.01) to (2,2.0);
        \draw[bl] (1,-0.01) \ur (2,1.) (3,-0.01) \ul (2,1.); 
        \end{tikzpicture}   
        :=
        \begin{tikzpicture}[anchorbase,scale=.2]
        \draw[bl] (1,-0.01) \uur (2,1.) (3,-0.01) \uul (2,1.) to (2,2.01) ; 
        \fill[bl] (2,2) circle (2.5mm); 
        \end{tikzpicture}    
        \quad, \quad
        \begin{tikzpicture}[anchorbase,xscale=.2,yscale=-.2]
        \draw[white] (2,2.01) to (2,2.0);
        \draw[bl] (1,-0.01) \ur (2,1.) (3,-0.01) \ul (2,1.); 
        \end{tikzpicture}   
        :=
        \begin{tikzpicture}[anchorbase,xscale=.2,yscale=-.2]
        \draw[bl] (1,-0.01) \uur (2,1.) (3,-0.01) \uul (2,1.) to (2,2.01) ; 
        \fill[bl] (2,2) circle (2.5mm); 
        \end{tikzpicture} . 
\]
 to display the (co)unit of the adjunctions from \Cref{counitadj}.
\end{notation}
\begin{definition} For $n\in \N_0$, the diagrammatic monoidal category
$\DS_n:=\DS_{\mathfrak{gl}_n}$ is the quotient of $\DS_n^{\free}$ by the tensor
ideal of morphisms generated by the following diagrammatic relations:
\begin{itemize}[leftmargin=*]
\item \emph{Isotopy:} diagrams related by an ambient isotopy relative to the boundary are equal.
\item \emph{Polynomial forcing and needle relations:}
for each color $c_i$, $i\in \intset{n-1}$ and any $p\in \C[x_1,\dots,x_n]$  the relations
\begin{equation}
\label{eq:polyforce}
    \begin{tikzpicture}[anchorbase,smallnodes, scale=.2]
        \draw[bl] (0,-0.01) \pu (0,3.01);
        \draw (-3,0.51) rectangle (-1,2.51);
        \node at (-2,1.5) {$p$};
    \end{tikzpicture}  
    =
    \begin{tikzpicture}[anchorbase,smallnodes,scale=.2]
        \draw[bl] (0,-0.01) \pu (0,3.01);
        \draw (4.5,0.51) rectangle (.5,2.51);
        \node at (2.5,1.5) {$s_i(p)$};
    \end{tikzpicture}  
+
\begin{tikzpicture}[anchorbase,smallnodes,scale=.2]
        \draw[bl] (0,-0.01) \pu (0,1.01);
        \fill[bl] (0,1) circle (2.5mm); 
        \fill[bl] (0,2) circle (2.5mm); 
        \draw[bl] (0,2) \pu (0,3.01);
        \draw (4.5,0.51) rectangle (.5,2.51);
        \node at (2.5,1.5) {$\del_i(p)$};
    \end{tikzpicture}, \quad
    \begin{tikzpicture}[anchorbase,smallnodes,scale=.2]
        \draw[bl] (0,1.01) \pu (0,2.01);
        \fill[bl] (0,1) circle (2.5mm); 
        \fill[bl] (0,2) circle (2.5mm); 
    \end{tikzpicture} 
    = 
    \begin{tikzpicture}[anchorbase,smallnodes,scale=.2]
        \draw (2.5,0.51) rectangle (0,2.51);
        \node at (1.25,1.5) {$\alpha_i$};
    \end{tikzpicture} 
    = 
    \begin{tikzpicture}[anchorbase,smallnodes,scale=.2]
        \draw (7,0.51) rectangle (0,2.51);
        \node at (3.5,1.5) {$x_i-x_{i+1}$};
    \end{tikzpicture} 
    , \quad
    \begin{tikzpicture}[anchorbase,scale=.2]
        \draw[bl] (0,-2) to (0,-1) \lu (-1,0) \ur (0,1) \rd (1,0) \dl (0,-1); 
        \end{tikzpicture}   = 
        0.
    \end{equation}
\item \emph{Compatibilities:} Frobenius structures interact with six- and four-valent vertices: 
\begin{equation}
\label{eq:comp}
        \begin{tikzpicture}[anchorbase,scale=.2]
            \draw[bl] (1,-0.01) \uur (2,1.) (3,-0.01) \uul (2,1.) to (2,2.01); 
            \draw[rd] (2,-0.01) to (2,1.)  \ulu (1,2.01) (2,1.)  \uru (3,2.01);
            \fill[bl] (2,1.9) circle (2.5mm);
        \end{tikzpicture}
        =
        \begin{tikzpicture}[anchorbase,scale=.2]
            \draw[bl] (1,-0.01) \ur (2,1) \rd (3,-0.01); 
            \draw[rd] (1,2.01) \dr (2,1.35) \ru (3,2.01); 
            \draw[rd] (2,-0.01) to (2,.5);
            \fill[rd] (2,.5) circle (2.5mm);
        \end{tikzpicture}
        +
        \begin{tikzpicture}[anchorbase,scale=.2]
            \draw[bl] (1,-0.01) to (1,0.5) (3,-0.01) to (3,0.5); 
            \draw[rd] (2,-0.01) to (2,1.)  \ulu (1,2.01) (2,1.)  \uru (3,2.01);
            \fill[bl] (1,.5) circle (2.5mm);
            \fill[bl] (3,.5) circle (2.5mm);
        \end{tikzpicture}\quad, \quad    
        \begin{tikzpicture}[anchorbase,scale=.2]
        \draw[bl] (0,-1.01) \uur (2,1.) (4,-1.01) \uul (2,1.) to (2,2.01); 
        \draw[rd] (2,-0.01) to (2,1.)  \ulu (1,2.01) (2,1.)  \uru (3,2.01);
        \draw[rd] (1,-1.01) \uur (2,0) (3,-1.01) \uul (2,0);
    \end{tikzpicture}  =
    \begin{tikzpicture}[anchorbase,scale=.2]
        \draw[bl] (1,-0.01) \uur (2,1.) \uur (3,2) to (3,3); 
        \draw[bl] (2,1.) to [out=315,in=180] (3,.4) to [out=0,in=225] (4,1.);
        \draw[rd] (2,-0.01) to (2,1.)  \ulu (1,3.01) (2,1) to (4,1) ;
        \draw[rd] (4,-0.01) to (4,1)  \uru (5,3.01);
        \draw[bl] (5,-0.01) \uul (4,1.) \uul (3,2); 
    \end{tikzpicture}  
\quad,\quad
    \begin{tikzpicture}[anchorbase,scale=.2]
        \draw[gr] (0,-0.01) \pu (2,2);
        \fill[gr] (2,1.9) circle (2.5mm); 
        \draw[bl] (2,-0.01) \pu (0,2.01); 
        \end{tikzpicture}
        =        
        \begin{tikzpicture}[anchorbase,scale=.2]
        \draw[gr] (0,-0.01) \pu (0,0.5);
        \fill[gr] (0,0.5) circle (2.5mm); 
        \draw[bl] (2,-0.01) \pu (0,2.01); 
        \end{tikzpicture}
        \quad,\quad
            \begin{tikzpicture}[anchorbase,scale=.2]
        \draw[gr] (0,-0.01) \uur (1,1) \pu (3,3);
        \draw[gr] (2,-0.01) \uul (1,1); 
        \draw[bl] (4,-0.01) to  (4,1)\pu (0,3.01); 
        \end{tikzpicture}
        =    
        \begin{tikzpicture}[anchorbase,scale=.2]
        \draw[gr] (0,-0.01) \uur (3,2) to (3,3);
        \draw[gr] (2,-0.01) \pu (3.5,1.25) \uul (3,2); 
        \draw[bl] (4,-0.01) \pu (0,2.01) to (0,3.01); 
        \end{tikzpicture}\quad.
\end{equation}
\item \emph{Parabolic relations:} the following relation for any choice of three colors
\begin{equation}
\label{eq:para}
    \begin{tikzpicture}[anchorbase,scale=.2]
    \draw[bl] (0,-0.01) \pu  (4,4); 
    \draw[gr] (2,-0.01) \pu  (0,2)\pu (2,4); 
    \draw[pu] (4,-0.01) \pu (0,4); 
    \end{tikzpicture}
    =
    \begin{tikzpicture}[anchorbase,scale=.2]
    \draw[bl] (0,-0.01) \pu  (4,4); 
    \draw[gr] (2,-0.01) \pu  (4,2)\pu (2,4); 
    \draw[pu] (4,-0.01) \pu (0,4); 
    \end{tikzpicture}
    \quad,\quad
    \begin{tikzpicture}[anchorbase,scale=.2]
    \draw[bl] (1,-0.01) \uur (2,1.) (3,-0.01) \uul (2,1.) to (2,2.01) \pu (4,4); 
    \draw[rd] (2,-0.01) to (2,1.)  \ulu (1,2.01) \pu (3,4) (2,1.)  \uru (3,2.01) \pu (5,4);;
    \draw[or] (5,-0.01) to (5,1) \pu (1,4); 
    \end{tikzpicture}
    =
    \begin{tikzpicture}[anchorbase,scale=.2]
    \draw[bl] (1,-0.01) \pu (3,2.01) \uur (4,3.) (3,-0.01) \pu (5,2.01) \uul (4,3.) to (4,4.01); 
    \draw[rd] (2,-0.01) \pu (4,2.01) to (4,3.)  \ulu (3,4.01) (4,3.)  \uru (5,4.01);
    \draw[or] (5,-0.01) \pu (1,3) to (1,4); 
    \end{tikzpicture}
    \quad,\quad
        \begin{tikzpicture}[anchorbase,scale=.2]
            \draw[gr]   (0,-0.01) \uur (1,1)
                        (3,-0.01) \uul (1,1) \uur (3,3)
                        (5,-0.01) \uul (3,3) \pu (3,5);
            \draw[bl] (1,-0.01) to  (1,1)  to [out=135,in=215] (1,4) to (1,5)
                                    (1,1) to [out=45,in=215] (3,2) to (3,3)\uru (4,5)
                                                (3,3) to [out=135,in=315] (1,4) to (1,5)
                      (4,-0.01) \uul (3,2);
            \draw[rd]   (2,-0.01) \pu (3,2) \ulu (1,4) \ulu (0,5) 
                                                (1,4) \uru (2,5)
                                    (3,2) \uru (5,5);
        \end{tikzpicture}
        =    
        \begin{tikzpicture}[anchorbase,xscale=-.2,yscale=.2]
            \draw[gr]   (0,-0.01) \uur (1,1)
                        (3,-0.01) \uul (1,1) \uur (3,3)
                        (5,-0.01) \uul (3,3) \pu (3,5);
            \draw[bl] (1,-0.01) to  (1,1)  to [out=135,in=215] (1,4) to (1,5)
                                    (1,1) to [out=45,in=215] (3,2) to (3,3)\uru (4,5)
                                                (3,3) to [out=135,in=315] (1,4) to (1,5)
                      (4,-0.01) \uul (3,2);
            \draw[rd]   (2,-0.01) \pu (3,2) \ulu (1,4) \ulu (0,5) 
                                                (1,4) \uru (2,5)
                                    (3,2) \uru (5,5);
        \end{tikzpicture}.
\end{equation}
as long as they correspond to the possible rank three parabolic subgroups of $S_n$, namely
of type $A_1\times A_1 \times A_1$, $A_1\times A_2$, and $A_3$ in $S_n$
respectively.
\end{itemize}
\end{definition}

\begin{remark} 
    Examples of isotopy relations are the snake relations for each color, see
    \Cref{counitadj}, and the more general vertex rotation relations (in any color):                                
\begin{equation}
    \label{eq:Drell}
    \begin{tikzpicture}[anchorbase,scale=.2]
        \draw[bl] (1,-1) \pu (2.25,1) \ul (1.75,1.5) \pl (0.25,0.5) \lu (-.25,1) \pu (1,3); 
        \end{tikzpicture}
        =
        \begin{tikzpicture}[anchorbase,scale=.2]
            \draw[bl] (1,-1)  \pu (1,3); 
            \end{tikzpicture}
            =
            \begin{tikzpicture}[anchorbase,xscale=-.2,yscale=.2]
                \draw[bl] (1,-1) \pu (2.25,1) \ul (1.75,1.5) \pl (0.25,0.5) \lu (-.25,1) \pu (1,3); 
                \end{tikzpicture}
            \quad, \quad
    \begin{tikzpicture}[anchorbase,scale=.2]
        \draw[bl] (1,3)\pd  (0.5,1.51) to [out=270,in=225] (2,1.) 
        (2,-1) \pu (2.75,0.25) \uul (2,1.) to (2,1.2)  \pu (3,3); 
        \draw[rd] 
        (1,-1)\pu (2,0.8) to (2,1.)  \ulu (1.25,1.75) \pu (2,3) 
        (2,1.)  to [out=45,in=90] (3.5,0.51) \pd (3,-1);
    \end{tikzpicture}    =
    \begin{tikzpicture}[anchorbase,scale=.2]
        \draw[rd] (1,-1) to (1,-0.01) \uur (2,1.) 
        (3,-1) to (3,-0.01) \uul (2,1.) to (2,2.01) to (2,3); 
        \draw[bl] (2,-1) to (2,-0.01) to (2,1.)  \ulu (1,2.01) to (1,3)
        (2,1.)  \uru (3,2.01) to (3,3);
    \end{tikzpicture}   \quad, \quad
        \begin{tikzpicture}[anchorbase,scale=.2]
        \draw[gr] (0,-1) \pu (1.5,0.5) \pu (0.5,1.5) \pu (2,3);
        \draw[bl] (2,-1) \pu (2.25,1) \ul (1.75,1.5) \pl (0.25,0.5) \lu (-.25,1) \pu (0,3); 
        \end{tikzpicture}
        =
        \begin{tikzpicture}[anchorbase,scale=.2]
        \draw[gr] (0,-1) to (0,-0.01) \pu (2,2.01) to (2,3);
        \draw[bl] (2,-1) to (2,-0.01) \pu (0,2.01) to (0,3); 
        \end{tikzpicture}\quad.
\end{equation}
\end{remark}
\begin{remark}
     We leave it to the reader to verify that among the useful consequences of the defining relations are that distant colors satisfy the Reidemeister 2 move and that any diagram vanishes which contains a facet bounded by a polygon of a single color:
    \begin{equation}\label{Reidem2}
    \begin{tikzpicture}[anchorbase,scale=.2]
    \draw[gr] (0,-0.01) \pu (2,2) \pu (0,4);
    \draw[bl] (2,-0.01) \pu (0,2) \pu (2,4);
    \end{tikzpicture}
    =
    \begin{tikzpicture}[anchorbase,scale=.2]
        \draw[gr] (0,-0.01)  \pu (0,4);
        \draw[bl] (2,-0.01)  \pu (2,4);
        \end{tikzpicture}
        \quad, \quad
    \begin{tikzpicture}[anchorbase,scale=.2]
            \draw[bl] (0,0)  to  (-1,1) to  (-1,2) to (0, 3) to (3, 3) to (4, 2) to (4,1) to (3,0) to  (3.25,-.5) ;
            \draw[bl] (0,0)  to (-.25,-.5) 
            (-1,1) to (-1.5,.75)  
            (-1,2) to (-1.5,2.25)  
            (0, 3) to (-.25,3.5)  
            (3, 3) to (3.25,3.5)  
            (4, 2) to (4.5,2.25)
            (4,1) to (4.5,0.75);
            \draw[bl, dotted] (0,0)  to  (3,0);
    \end{tikzpicture}
    =0.
\end{equation}
\end{remark}

\begin{remark}
    The $\mathfrak{sl}_n$-version $\DS_{\mathfrak{sl}_n}$ of the diagrammatic category, cf.~\Cref{SBimglsl}, is defined analogously to $\DS_{\mathfrak{gl}_n}$, but with the generators $x_i$ omitted and the polynomial forcing relations \eqref{eq:polyforce} replaced by
    \begin{equation}
        \label{eq:barbell}
    \begin{tikzpicture}[anchorbase,smallnodes, scale=.2]
        \draw[rd] (1,-0.01) \pu (1,3.01);
        \draw[rd] (0,1.01) \pu (0,2.01);
        \fill[rd] (0,1) circle (2.5mm); 
        \fill[rd] (0,2) circle (2.5mm); 
    \end{tikzpicture}  
    =
    -\;
    \begin{tikzpicture}[anchorbase,smallnodes,scale=.2]
        \draw[rd] (-1,-0.01) \pu (-1,3.01);
        \draw[rd] (0,1.01) \pu (0,2.01);
        \fill[rd] (0,1) circle (2.5mm); 
        \fill[rd] (0,2) circle (2.5mm); 
    \end{tikzpicture}  
+2\; 
\begin{tikzpicture}[anchorbase,smallnodes,scale=.2]
        \draw[rd] (0,-0.01) \pu (0,1.01);
        \fill[rd] (0,1) circle (2.5mm); 
        \fill[rd] (0,2) circle (2.5mm); 
        \draw[rd] (0,2) \pu (0,3.01);
    \end{tikzpicture} 
    \quad,\quad
    \begin{tikzpicture}[anchorbase,smallnodes, scale=.2]
        \draw[rd] (1,-0.01) \pu (1,3.01);
        \draw[bl] (0,1.01) \pu (0,2.01);
        \fill[bl] (0,1) circle (2.5mm); 
        \fill[bl] (0,2) circle (2.5mm); 
    \end{tikzpicture}  
    =
    \begin{tikzpicture}[anchorbase,smallnodes,scale=.2]
        \draw[rd] (-1,-0.01) \pu (-1,3.01);
        \draw[bl] (0,1.01) \pu (0,2.01);
        \fill[bl] (0,1) circle (2.5mm); 
        \fill[bl] (0,2) circle (2.5mm); 
    \end{tikzpicture}  
    +
    \begin{tikzpicture}[anchorbase,smallnodes,scale=.2]
        \draw[rd] (-1,-0.01) \pu (-1,3.01);
        \draw[rd] (0,1.01) \pu (0,2.01);
        \fill[rd] (0,1) circle (2.5mm); 
        \fill[rd] (0,2) circle (2.5mm); 
    \end{tikzpicture}  
-\; 
\begin{tikzpicture}[anchorbase,smallnodes,scale=.2]
        \draw[rd] (0,-0.01) \pu (0,1.01);
        \fill[rd] (0,1) circle (2.5mm); 
        \fill[rd] (0,2) circle (2.5mm); 
        \draw[rd] (0,2) \pu (0,3.01);
    \end{tikzpicture} 
    \quad,\quad
    \begin{tikzpicture}[anchorbase,smallnodes, scale=.2]
        \draw[rd] (1,-0.01) \pu (1,3.01);
        \draw[gr] (0,1.01) \pu (0,2.01);
        \fill[gr] (0,1) circle (2.5mm); 
        \fill[gr] (0,2) circle (2.5mm); 
    \end{tikzpicture}  
    =
    \begin{tikzpicture}[anchorbase,smallnodes,scale=.2]
        \draw[rd] (-1,-0.01) \pu (-1,3.01);
        \draw[gr] (0,1.01) \pu (0,2.01);
        \fill[gr] (0,1) circle (2.5mm); 
        \fill[gr] (0,2) circle (2.5mm); 
    \end{tikzpicture}  
\end{equation}
for every color, every pair or adjacent colors (here: red and blue), and every pair of distant colors (here: red and green) respectively.
\end{remark}

The main result of \cite{EK} is  a diagrammatic Bott--Samelson bimodule category:
\begin{thm}\label{thm:equivalence} There is an equivalence of monoidal graded $\k$-linear categories \[\DS_n\xrightarrow{\cong} \overline{\BSbim}^\gr_n\] 
with sends the object $c_i$ of color $i$ to $B_i$, and the generating morphisms to bimodule morphisms as follows. The generating morphisms $
           \begin{tikzpicture}[anchorbase,scale=.2]
            \draw[bl] (1,-0.01) \uur (2,1.) (3,-0.01) \uul (2,1.) to (2,2.01); 
            \end{tikzpicture}$~, $
           \begin{tikzpicture}[anchorbase,scale=-.2]
            \draw[bl] (1,-0.01) \uur (2,1.) (3,-0.01) \uul (2,1.) to (2,2.01); 
            \end{tikzpicture}
       $~, $
           \begin{tikzpicture}[anchorbase,scale=.2]
            \draw[bl] (0,-0.01) \pu (0,-1.01);
            \fill[bl] (0,-1) circle (2.5mm); 
            \end{tikzpicture}$\, ,
            $
           \begin{tikzpicture}[anchorbase,scale=-.2]
            \draw[bl] (0,-0.01) \pu (0,-1.01);
            \fill[bl] (0,-1) circle (2.5mm); 
            \end{tikzpicture}
$ are sent to the bimodule morphisms $\mm,\Delta,\epsilon,\eta$ from the Frobenius structure on $B_i$ from \Cref{BSFrob}, the morphism $
    \begin{tikzpicture}[anchorbase,scale=.2]
        \draw[rd] (1,-0.01) \uur (2,1.) (3,-0.01) \uul (2,1.) to (2,2.01); 
        \draw[bl] (2,-0.01) to (2,1.)  \ulu (1,2.01) (2,1.)  \uru (3,2.01);
    \end{tikzpicture}
$ is sent to $B_i\otimes_{R} B_{i+1} \otimes_{R} B_i \to B_{i+1}\otimes_{R} B_{i} \otimes_{R} B_{i+1}$ given  by
\begin{align*}
R \otimes_{R^{s_i}} R \otimes_{R^{s_{i+1}}} R \otimes_{R^{s_i}} R 
&\to 
R \otimes_{R^{s_{i+1}}} R \otimes_{R^{s_i}} R \otimes_{R^{s_{i+1} }} R 
\\
1 \otimes 1 \otimes 1 \otimes 1 & \mapsto 1 \otimes 1 \otimes 1 \otimes 1\\
1 \otimes x_i \otimes 1 \otimes 1 & \mapsto (x_i+x_{i+1}) \otimes 1 \otimes 1 \otimes 1- 1\otimes 1 \otimes 1 \otimes x_{i+1},
\end{align*}
the generator $\begin{tikzpicture}[anchorbase,scale=-.2] \draw[bl] (0,-0.01) \pd
            (2,-2.01); \draw[gr] (2,-0.01) \pd (0,-2.01); \end{tikzpicture}$ is
            sent to the canonical isomorphism $B_i\otimes B_j \cong B_j \otimes
            B_i$ from (a special case of) \eqref{BiBj}, and finally $\boxed{x_k}$ is sent to the
            endomorphism of multiplying by the variable $x_k$.
\end{thm}
\begin{corollary} \label{rem:diagtosbim}
With the concrete \Cref{Karconcrete} for the Karoubi envelope, \Cref{thm:equivalence} implies that the Soergel bimodule categories from \Cref{def:SBim} admit a diagrammatic description as well.
\end{corollary}
\begin{proof}
Indeed, starting with $\DS_n$,  we first adjoin grading shifts of objects and subsequently restrict to the degree zero subcategory as in \Cref{def:versionsBS}. This results in a diagrammatic category equivalent to $\BSBim_n$. Then one takes the Karoubi envelope using \Cref{Karconcrete}  and finally passes to the additive completion by taking formal direct sums and matrices of morphism spaces (as in  \Cref{add}, see also \cite{MR2174270}). The result is then a diagrammatic model for $\SBim_n$.
\end{proof}

 \begin{remark} 
    \label{rem:dualities}
    \Cref{thm:equivalence} can be upgraded to provide diagrammatic incarnations of
    graded $\k$-linear bicategories of \emph{singular} Bott-Samelson bimodules (i.e  of generalizations of Bott-Samelson bimodules as originally defined in  \cite{Soergel}, \cite{MR2097586}, \cite{MR2844932})
    in terms of \emph{foams} embedded in an axis-parallel unit cube in $\R^3$,
    see e.g. \cite[Appendix A]{HRW1}. For example, the six-valent vertex
    corresponds to the following foam:
\begin{equation}\label{foam}
 \scalebox{3.0}[3.0]{  \sixv}
        \quad \longleftrightarrow \quad\quad
    \begin{tikzpicture} [anchorbase,xscale=-.3,yscale=.4,tinynodes]
	%shading
	\path [fill=yellow, opacity=.1] (0,2) to (0,2+4) to [out=0,in=150] (3.5,1.5+4) to (4.5,1.5+4) to [out=30,in=180] (7.5,2+4) to (7.5,2) to [out=180,in=45] (6.5,1.5) to (5.5,1.5) to [out=150,in=30] (2.5,1.5) to (1.5,1.5) to [out=135,in=0] (0,2);
	\path [fill=yellow, opacity=.1] (.5,1) to (.5,1+4) to [out=0,in=135] (2,.5+4) to (3,.5+4) to [out=30,in=210] (3.5,1.5+4) to (4.5,1.5+4) to [out=315,in=135] (6,.5+4) to (7,.5+4) to [out=45,in=180] (8,1+4) to (8,1) to [out=180,in=315] (6.5,1.5) to (5.5,1.5) to [out=225,in=45] (5,.5) to (4,.5) to [out=150,in=330] (2.5,1.5) to (1.5,1.5) to [out=225,in=0] (.5,1);
	\path [fill=yellow, opacity=.1]  (1,0) to (1,0+4) to [out=0,in=225] (2,.5+4) to (3,.5+4) to [out=330,in=210] (6,.5+4) to (7,.5+4) to [out=315,in=180] (8.5,0+4) to (8.5,0) to [out=180,in=330] (5,.5) to (4,.5) to [out=210,in=0] (1,0);
	\path [fill=white, opacity=1]  (3,.5+4) to [out=270,in=180] (4.5,-.75+4) to [out=0,in=270] (6,.5+4) to (7,.5+4) to (7,0+4) to [out=270,in=90] (5,1) to (5,.5) to (4,.5) to (4,1) to [out=90,in=270] (2,0+4) to (2,0.5+4) to (3,.5+4);
	\path [fill=white, opacity=1] (2.5,1.5) to [out=90,in=180] (4,2.75) to [out=0,in=90] (5.5,1.5) to (6.5,1.5) to (6.5,2) to [out=90,in=270] (4.5,1+4) to (4.5,1.5+4) to (3.5,1.5+4) to (3.5,1+4) to [out=270,in=90] (1.5,2) to (1.5,1.5) to (2.5,1.5);
		\path [fill=red, opacity=.2]  (3,.5+4) to [out=270,in=180] (4.5,-.75+4) to [out=0,in=270] (6,.5+4) to (7,.5+4) to (7,0+4) to [out=270,in=90] (5,1) to (5,.5) to (4,.5) to (4,1) to [out=90,in=270] (2,0+4) to (2,0.5+4) to (3,.5+4);
	\path [fill=blue, opacity=.2] (2.5,1.5) to [out=90,in=180] (4,2.75) to [out=0,in=90] (5.5,1.5) to (6.5,1.5) to (6.5,2) to [out=90,in=270] (4.5,1+4) to (4.5,1.5+4) to (3.5,1.5+4) to (3.5,1+4) to [out=270,in=90] (1.5,2) to (1.5,1.5) to (2.5,1.5);
	%bottom web
	\draw [very thick] (0,2) to [out=0,in=135] (1.5,1.5);
	\draw [very thick] (.5,1) to [out=0,in=225] (1.5,1.5);
	\draw [very thick] (1,0) to [out=0,in=210] (4,.5);
	\draw [double] (1.5,1.5) to (2.5,1.5);
	\draw [very thick] (2.5,1.5) to [out=330,in=150] (4,.5);
	\draw [very thick] (2.5,1.5) to [out=30,in=150] (5.5,1.5);
	\draw [double] (4,.5) to (5,.5);
	\draw [very thick] (5,.5) to [out=45,in=225] (5.5,1.5);
	\draw [double] (5.5,1.5) to (6.5,1.5);
	\draw [very thick] (5,.5) to [out=330,in=180] (8.5,0);
	\draw [very thick] (6.5,1.5) to [out=315,in=180] (8,1);
	\draw [very thick] (6.5,1.5) to [out=45,in=180] (7.5,2);
	%vertical edges
	\draw  (0,2) to (0,2+4);
	\draw  (.5,1) to (.5,1+4);
	\draw  (1,0) to (1,0+4);
	\draw  (7.5,2) to (7.5,2+4);
	\draw  (8,1) to (8,1+4);
	\draw  (8.5,0) to (8.5,0+4);
    	%vertical edges
	%seam
	\draw [very thick, red] (2.5,1.5) to [out=90,in=180] (4,2.75) to [out=0,in=90] (5.5,1.5);
	\draw [very thick,directed=.55, red] (3,.5+4) to [out=270,in=180] (4.5,-.75+4) to [out=0,in=270] (6,.5+4);
	\draw [very thick,rdirected=.4, red] (2,.5+4) to (2,0+4) to [out=270,in=90] (4,1) to (4,.5);
	\draw [very thick,rdirected=.45, red] (3.5,1.5+4) to (3.5,1+4) to [out=270,in=90] (1.5,2) to (1.5,1.5);
	\draw [very thick,directed=.65, red] (4.5,1.5+4) to (4.5,1+4) to [out=270,in=90] (6.5,2) to (6.5,1.5);
	\draw [very thick,directed=.6, red] (7,.5+4) to (7,0+4) to [out=270,in=90] (5,1) to (5,.5);
	% top web
	\draw [very thick] (0,2+4) to [out=0,in=150] (3.5,1.5+4);
	\draw [very thick] (.5,1+4) to [out=0,in=135] (2,.5+4);
	\draw [very thick] (1,0+4) to [out=0,in=225] (2,.5+4);
	\draw [double] (2,.5+4) to (3,.5+4);
	\draw [very thick] (3,.5+4) to [out=30,in=210] (3.5,1.5+4);
	\draw [very thick] (3,.5+4) to [out=330,in=210] (6,.5+4);
	\draw [double] (3.5,1.5+4) to (4.5,1.5+4);
	\draw [very thick] (4.5,1.5+4) to [out=315,in=135] (6,.5+4);
	\draw [double] (6,.5+4) to (7,.5+4);
	\draw [very thick] (7,.5+4) to [out=315,in=180] (8.5,0+4);
	\draw [very thick] (7,.5+4) to [out=45,in=180] (8,1+4);
	\draw [very thick] (4.5,1.5+4) to [out=30,in=180] (7.5,2+4);
    % coordinate system
    \draw[->]  (-.5,-.5) to (1.5,-.5+0);
    \draw[->]  (-.5,-.5) to (-1.5,-.5+1.5);
    \draw[->]  (-.5,-.5) to (-.5,-.5+2);
    \node at (.8,-.8) {$y$};
    \node at (-1.8,-.5+1) {$x$};
    \node at (0,-.5+1.5) {$z$};
\end{tikzpicture}
    \end{equation}
We will not need the foam calculus in this paper, but it might help to guide intuition: 
First, \eqref{foam} clearly shows that there are three possible ways of composing morphisms, similar to the three ways of stacking two cubes, i.e. placing them next to each other.  
We can stack cubes \emph{vertically} (corresponding to $\vcomp$), \emph{horizontally}
(corresponding to $\hcomp$) or \emph{behind} each other (corresponding to parabolic
induction, $\boxtimes$). Second, \eqref{foam}  shows that reflections in coordinate hyperplanes induce symmetries of the diagrammatic calculus.

    For each fixed $n\in \N_0$ the diagrammatic category $\DS_n$ has a subgroup of automorphisms isomorphic to $(\Z/2\Z)^3$ generated by the following involutions:
    \begin{itemize}
        \item The \emph{Dynkin diagram automorphism}: the covariant and monoidal
        automorphism $r_x$ that reverses the sequence of colors $c_i\mapsto
        c_{n-i}$ in objects and diagrams and sends $x_i \mapsto -x_{n+1-i}$. 
        \item The \emph{vertical automorphism}: the  covariant opmonoidal automorphism $r_y$ that reverses
        objects and reflects diagrams in a vertical line in the drawing surface.
        \item The \emph{horizontal automorphism}: the contravariant monoidal automorphism $r_z$ that fixes
        objects and reflects diagrams in a horizontal line in the drawing
        surface.
    \end{itemize}
    As a composite we also obtain:
    \begin{itemize}
        \item The contravariant and opmonoidal duality automorphism $r_{yz}=r_y\circ r_z$,
        which reverses objects and rotates diagrams by $\pi$ in the plane.
    \end{itemize}
    Analogous symmetries for a monoidal bicategory based on the dotted embedded
    cobordism categories underlying Khovanov homology appear in \cite[\S2.2]{HRW4}.
\end{remark}
\begin{cor}\label{cor:selfduality}The generators $c_i$, $i\in[1;n-1]$, of $\DS_n^{\free}$ and also $\DS_n$ are self-dual. Thus, the categories are rigid (with diagrammatic duality $r_{yz}$ as in \Cref{rem:dualities}).
\end{cor}
\begin{proof}This is clear from  the maps and relations in \eqref{eq:Drel} (and \Cref{rem:dualities}).
\end{proof}

\subsection{Semistrict monoidal structure on the diagrammatic categories}
\label{sec:ssmondiag}
Next, we assemble
the monoidal categories $\DS_n$, $n\in \N_0$, into a semistrict monoidal
$2$-category $\DS$. 

\begin{definition}
    \label{def:DS}
    Let $\DS$ denote the locally graded $\k$-linear $2$-category on the set of objects $\N_0$,
    whose morphism categories for $m,n\in \N_0$ are
    \[
    \Hom_{\DS}(m,n)=
    \begin{cases} 
    \DS_n & m=n,\\
    0 & m\neq n.
    \end{cases}
    \]
    Here, $0$ denotes the graded $\k$-linear category on a single object with
    endomorphism algebra $0$. The composition of endomorphisms is given by the
    (strictly associative!) horizontal composition $\hcomp$ in $\DS_n$, and it
    is trivial for non-endo-morphisms.
\end{definition}
\begin{remark}\label{iota} If $\cal{I}$ denotes the $2$-category with one
object, one endomorphism and one $2$-morphism, then there is an obvious functor
of $2$-categories $\iota\colon\cal{I}\to\DS$ which picks out the object $0$ with
the its identity 1- and 2- morphisms.
\end{remark}
\begin{notation}
The $2$-category $\DS$ inherits a diagrammatic calculus from the one of $\DS_n$. Unlike
the one-object $2$-categories $\DS_n$, we now have several objects $n\in \N_0$ and
in the graphical calculus we now label the regions with the relevant
objects. For example,
\[
\begin{tikzpicture}[anchorbase,scale=.3,tinynodes]
            \draw[bl] (1,-0.01) \uur (2,1.) (3,-0.01) \uul (2,1.) to (2,2.01); 
            \node at (1,1.5) {$3$};
            \node at (3,1.5) {$3$};
            \node at (2,.2) {$3$};
            \end{tikzpicture}
\]
encodes a $2$-morphism between the $1$-morphisms $c_1 \hcomp c_1$ and $c_1$, and each of the $1$-morphisms is an endomorphism of the object $3$. Of course, the labelling of regions has to be constant in each diagram (only $1$-endomorphisms are nontrivial). Labelling all regions instead of just one is still useful, since it directly gives information about the colors that can appear in each region (here only blue and red, corresponding to $c_1$ and $c_2$) and which polynomials can float in each region (here polynomials in $x_1$, $x_2$, and $x_3$).

We often omit labels and implicitly assume the reader to pick any which makes the respective expression well-defined.
\end{notation}

Our next goal is to equip $\DS$ with a monoidal structure, that mimics the
$\boxtimes$-product from \S\ref{sec:monbicat}. First we need analogs of the
functors $\jj_{a|c} \colon \BSbim_b \to \BSbim_{a+b+c}$ from \eqref{eq:embed},
which are induced by variable-shifting morphisms from
Definition~\ref{def:variableshift}.

\begin{definition}\label{shifts}
For any triple $a,b,c\in \N_0$ consider the monoidal graded $\k$-linear \emph{color-shifting functor}
\[
\jj_{a|c} \colon \DS_b \to \DS_{a+b+c}
\]
which is determined on objects by sending the generating object $c_i\in \DS_b$, $1\leq i <b$ to the object $c_{a+i}\in \DS_{a+b+c}$ and on morphisms by sending each diagram to the same diagram, but with all indices of variables and all colors in the enumeration from \eqref{eq:colorsequence} shifted by $+a$. (It is straightforward from the definition of $\DS_n$ that color-shifting respects the defining relations of the $\DS_n$.)
\end{definition}

\begin{example} For $(a,b,c)=(1,3,10)$ we have for instance, recalling \eqref{eq:colorsequence},
\[
\jj_{1|10}\left(\begin{tikzpicture}[anchorbase,scale=.2]
        \draw[rd] (1,-0.01) \uur (2,1.) (3,-0.01) \uul (2,1.) to (2,2.01); 
        \draw[bl] (2,-0.01) to (2,1.)  \ulu (1,2.01) (2,1.)  \uru (3,2.01);
    \end{tikzpicture}
    \right)
    =
    \begin{tikzpicture}[anchorbase,scale=.2]
        \draw[bl] (1,-0.01) \uur (2,1.) (3,-0.01) \uul (2,1.) to (2,2.01); 
        \draw[gr] (2,-0.01) to (2,1.)  \ulu (1,2.01) (2,1.)  \uru (3,2.01);
    \end{tikzpicture}
  \quad\text{or, if we also label (some) regions:}\quad
\jj_{1|10}\left(\begin{tikzpicture}[anchorbase,scale=.2,tinynodes]
        \draw[rd] (1,-0.01) \uur (2,1.) (3,-0.01) \uul (2,1.) to (2,2.01); 
        \draw[bl] (2,-0.01) to (2,1.)  \ulu (1,2.01) (2,1.)  \uru (3,2.01);
        \node at (0,1) {$3$};
        \node at (4,1) {$3$};
    \end{tikzpicture}
    \right)
    =
    \begin{tikzpicture}[anchorbase,scale=.2,tinynodes]
        \draw[bl] (1,-0.01) \uur (2,1.) (3,-0.01) \uul (2,1.) to (2,2.01); 
        \draw[gr] (2,-0.01) to (2,1.)  \ulu (1,2.01) (2,1.)  \uru (3,2.01);
                \node at (0,1) {$14$};
        \node at (4,1) {$14$};
    \end{tikzpicture}.
  \]  
\end{example}
\begin{definition}
    \label{def:boxtimesdef}
For $m,n\in \N_0$, consider the \emph{parabolic induction functor} $\boxtimes$ given by 
\begin{align}\label{parabinddiag}
    \boxtimes\colon \DS_m\times \DS_n \to \DS_{m+n}, \qquad 
    (M,N) \mapsto \jj_{0|n}(M) \hcomp \jj_{m|0}(N) 
\end{align}
on objects and on morphisms.
\end{definition}
\begin{example} \label{tensor1hom}
For $(m,n)=(2,3)$ we have, recalling \eqref{eq:colorsequence},
\[
\left(\!\!
\begin{tikzpicture}[anchorbase,scale=.2,tinynodes]
            \draw[rd] (1,-0.01) \uur (2,1.) (3,-0.01) \uul (2,1.) to (2,2.01); 
        \node at (.75,1) {$2$};
        \node at (3.25,1) {$2$};
            \end{tikzpicture}\!\!
\right)\boxtimes 
\left(\!\!\begin{tikzpicture}[anchorbase,scale=.2,tinynodes]
        \draw[rd] (1,-0.01) \uur (2,1.) (3,-0.01) \uul (2,1.) to (2,2.01); 
        \draw[bl] (2,-0.01) to (2,1.)  \ulu (1,2.01) (2,1.)  \uru (3,2.01);
         \node at (.25,1) {$3$};
         \node at (3.75,1) {$x_2$};
        \node at (5.75,1) {$3$};
    \end{tikzpicture}\!\!
    \right)
    =
    \jj_{0|3}\left(
    \begin{tikzpicture}[anchorbase,scale=.2,tinynodes]
            \draw[rd] (1,-0.01) \uur (2,1.) (3,-0.01) \uul (2,1.) to (2,2.01); 
            \end{tikzpicture}
    \right)
    \hcomp
    \jj_{2|0}\left(
\begin{tikzpicture}[anchorbase,scale=.2,tinynodes]
        \draw[rd] (1,-0.01) \uur (2,1.) (3,-0.01) \uul (2,1.) to (2,2.01); 
        \draw[bl] (2,-0.01) to (2,1.)  \ulu (1,2.01) (2,1.)  \uru (3,2.01);
        \node at (3.75,1) {$x_2$};
    \end{tikzpicture}\!\!
    \right)
    =
        \begin{tikzpicture}[anchorbase,scale=.2,tinynodes]
            \draw[rd] (1,-0.01) \uur (2,1.) (3,-0.01) \uul (2,1.) to (2,2.01); 
            \end{tikzpicture}
            \hcomp 
            \begin{tikzpicture}[anchorbase,scale=.2,tinynodes]
        \draw[gr] (1,-0.01) \uur (2,1.) (3,-0.01) \uul (2,1.) to (2,2.01); 
        \draw[or] (2,-0.01) to (2,1.)  \ulu (1,2.01) (2,1.)  \uru (3,2.01);
        \node at (3.75,1) {$x_4$};
    \end{tikzpicture}
    =
                \begin{tikzpicture}[anchorbase,scale=.2,tinynodes]
        \draw[rd] (-2,-0.01) \uur (-1,1.) (0,-0.01) \uul (-1,1.) to (-1,2.01); 
        \draw[gr] (1,-0.01) \uur (2,1.) (3,-0.01) \uul (2,1.) to (2,2.01); 
        \draw[or] (2,-0.01) to (2,1.)  \ulu (1,2.01) (2,1.)  \uru (3,2.01);
        \node at (3.75,1) {$x_4$};
    \end{tikzpicture}\in\DS_5.
  \]  
  Note that it is useful to label the regions in the left expression, since this encodes by how many steps the colors of the second diagram have to be shifted (here by $m=2$).
\end{example}

Next we need the \emph{tensorators}, that is the analogs of the canonical isomorphisms from \eqref{eq:tensorator}. In the diagrammatic setting, we specify them as follows.

\begin{definition}
\label{pseudo7}
    Let $m,n\in \N_0$ and  $c_{\underline{i}}:= c_{i_1}\hcomp \cdots \hcomp c_{i_k}\in\DS_m$,  $c_{\underline{j}}:=c_{j_1}\hcomp\cdots \hcomp c_{j_l}\in \DS_n$.  Then we define an invertible $2$-morphism $\boxtimes_{c_{\underline{i}},c_{\underline{j}}}\colon
      (c_{\underline{i}} \boxtimes \id) \hcomp (\id \boxtimes c_{\underline{j}})\to(\id \boxtimes c_{\underline{j}})\hcomp (c_{\underline{i}} \boxtimes \id)
    $
    % , where 
    % \begin{gather*}
    % (\id \boxtimes c_{\underline{j}})\hcomp (c_{\underline{i}} \boxtimes \id) =  (c_{m+j_1}\hcomp\cdots \hcomp c_{m+j_l}) \hcomp (c_{i_1}\hcomp \cdots \hcomp c_{i_k}) \\
    % \uparrow\\
    % (c_{\underline{i}} \boxtimes \id) \hcomp (\id \boxtimes c_{\underline{j}}) = (c_{i_1}\hcomp \cdots \hcomp c_{i_k}) \hcomp (c_{m+j_1}\hcomp\cdots \hcomp c_{m+j_l})  
    % \end{gather*}
    by the following diagram, modelled on the $(k,l)$-shuffle permutation:
    \begin{equation}\label{pseudonat}
\boxtimes_{c_{\underline{i}},c_{\underline{j}}} :=    
\begin{tikzpicture}[anchorbase,scale=.3,tinynodes]
    \draw[bla] (0,-0.01) \pu  (6,4); 
    \draw[dotted] (1,-0.01) \pu  (7,4);
    \draw[bla] (2,-0.01) \pu  (8,4); 
    \draw[bla] (6,-0.01) \pu (0,4);
    \draw[dotted] (7,-0.01) \pu  (1,4);
    \draw[bla] (8,-0.01) \pu (2,4); 
    \node at (0,-.5) {$c_{i_1}$}; 
    \node at (2,-.5) {$c_{i_k}$}; 
    \node at (5.6,-.5) {$c_{m+j_1}$}; 
    \node at (8.4,-.5) {$c_{m+j_l}$};
    \node at (-.4,4+.5) {$c_{m+j_1}$}; 
    \node at (2.4,4+.5) {$c_{m+j_l}$}; 
    \node at (6,4+.5) {$c_{i_1}$}; 
    \node at (8,4+.5) {$c_{i_k}$};
    \end{tikzpicture}.
    \end{equation}
Here we omit to display the colors, which are encoded by the generating objects
$c_{i_1}, \cdots, c_{i_k}$ and $c_{m+j_1},\dots, c_{m+j_l}$. We just observe
that the color-shift by $+m$ guarantees that there are four-valent vertex
generators for every pair of strands that have to cross.
\end{definition}

Setting $f:=c_{\underline{i}}$ and $g:=c_{\underline{j}}$, the $2$-isomorphism
$\boxtimes_{f,g}$ witnesses the commutativity of the square of $1$-morphisms:
\[
\begin{tikzcd}
m+n 
\arrow{rrr}
{\id\boxtimes g} \arrow[swap]{d}{f\boxtimes \id} &&&   m+n \arrow{d}{f\boxtimes \id} \ar[dlll,Rightarrow,shorten >=5ex,shorten <=5ex,swap, "\boxtimes_{f,g}"] \\
m+n \arrow{rrr}{\id\boxtimes g} &&&   m+n.
\end{tikzcd}
\] 

    The tensorators from \eqref{pseudonat} can be used to make the parabolic
    induction functors from \Cref{def:boxtimesdef} compatible with the monoidal
    structure $\hcomp$. To see this, we first equip $\DS_m\times \DS_n$ with the Cartesian monoidal structure, which is defined componentwise on objects as $(f_2,g_2)\hcomp (f_1,g_1):=(f_2 \hcomp f_1,g_2\hcomp g_1)$ and on morphisms by $(\alpha_2,\beta_2)\hcomp (\alpha_1,\beta_1):=(\alpha_2 \hcomp \alpha_1,\beta_2\hcomp \beta_1)$. The tensorators now provide isomorphisms:
    \[
        (f_2\boxtimes g_2)\hcomp (f_1\boxtimes g_1) \xrightarrow{\cong} (f_2 \hcomp f_1)\boxtimes (g_2\hcomp g_1)
    \]
    which are natural in all four arguments $f_1,f_2,g_1,g_2$, so that
    $\boxtimes$ becomes monoidal with respect to $\hcomp$. Moreover, $\boxtimes$
    is graded $\k$-bilinear on the level of morphism spaces, so that it factors
    through a monoidal functor from the Deligne-Kelly tensor product of monoidal
    graded $\k$-linear categories. In fact, we can give a fully diagrammatic description of this tensor product:

    \begin{definition}
        \label{def:parabolicsub}
       For $m,n\in \N_0$ we let $\DS_{m,n}$ denote the
       diagrammatic monoidal graded $\k$-linear category, whose objects are
       words in generators $c_i$ for $i\in \intset{m+n-1}\setminus\{m\}$ and whose morphisms spaces are spanned by all generating morphisms of $\DS_{m+n}$, which do not involve $c_m$, subject to all diagrammatic relations in $\DS_{m+n}$ not involving $c_m$.
    \end{definition}

    The diagrammatic category $\DS_{m,n}$ models Bott-Samelson bimodules for the parabolic subgroup $S_m\times S_n \hookrightarrow S_{m+n}$. In fact, the parabolic induction functor \eqref{parabinddiag} factors through $\DS_{m,n}$, which motivates the (ab)use of notation $\DS_m\boxtimes \DS_n:= \DS_{m,n}$ and immediately yields a functor $\DS_{m,n}\to \DS_{m+n}$.
    Furthermore, note that an analog of the Dynkin diagram automorphism $r_x$ of $\DS_{m+n}$ induces an equivalence $\DS_{m,n}\to \DS_{n,m}$.

\begin{definition}
    \label{def:swap}
For $m,n\in \N_0$ we let $\swap_{m,n}\colon \DS_{m,n}\to \DS_{n,m}$ denote the equivalence of monoidal graded $\k$-linear categories induced by acting with the $(m-1,n-1)$-shuffle permutation on the ordered set of colors $(c_1,\dots, c_{m-1}, c_{m+1}, \dots c_{m+n-1})$ and by the $(m,n)$-shuffle permutation on the ordered set of variables $(x_1,\cdots, x_{m+n})$. 
\end{definition}

Similarly to how $2$-categories with one object can be viewed as strict monoidal
categories, we now consider $3$-categories with one object. These give rise to
\emph{semistrict monoidal 2-categories} in the sense of \cite[\S4.3]{KapVoe} and \cite[Definitions 2 and
3]{MR1402727}. 
\begin{definition}\label{defsemistrict}
A \emph{semistrict
monoidal $2$-category}, also known as a \emph{Gray monoid} (that is a monoid in the category $\operatorname{Gray}$, see \cite[\S3,\S8]{Gurskibook}), is a $\mathcal{V}$-category with one object, where
$\mathcal{V}$ is the symmetric monoidal category with (strict) $2$-categories as
objects, (strict) $2$-functors as morphisms and with monoidal structure given by
the Gray tensor product (with unit object $\cal{I}$ as in \Cref{iota}). 
\end{definition}
This
definition is strict enough to explicitly spell out the data and coherence
conditions, which is done, for example, in \cite[Lemma 4]{MR1402727}. Informally speaking, it is a monoidal bicategory whose underlying bicategory is in fact a $2$-category and  roughly half of the coherence morphisms for the monoidal structure are identity $2$-morphisms.

\begin{theorem}
    \label{thm:ssmtwocat}
  The data of $\boxtimes$ from Definitions~\ref{def:boxtimesdef} and \ref{pseudo7} and
  $\iota$ from \Cref{iota} equip the locally graded $\k$-linear $2$-category $\DS$ from \Cref{def:DS} with a semistrict
  monoidal structure in the sense of \cite[Lemma~4]{MR1402727}.
\end{theorem}
\begin{proof}
We verify the lists (1)-(7) of data and (i)-(viii) of properties from \cite[Lemma~4]{MR1402727}. The object in (1) (picked out by $\iota$) is just $0\in\N_0$, and
for $m,n\in\N_0$ we set $m\otimes n:=m\boxtimes n=m+n$ which gives (2). Given
diagrammatic $1$-morphisms $f\in\End_\DS(m)$ and $g\in\End_\DS(n)$ for
$m,n\in\N_0$, we can take $\jj_{0|n}(f)$ and $\jj_{m|0}(g)$ as the $1$-morphisms
required for (3) and (4). Similarly for the $2$-morphisms required by (5) and (6).
As $2$-isomorphisms for (7) we use $\boxtimes_{f,g}$ from \Cref{pseudo7}.
Thus we have all the required data and it remains to check the properties. 

Property (i) requires that tensoring with an object $m\in\N_0$ from the left (or
with $n\in\N_0$ on the right) induces a $2$-functor. Indeed it is given by
$n\mapsto m+n$ (respectively $m\mapsto m+n$) on objects and by the
color-shifting functor $\jj_{m|0}\colon \DS_n\to \DS_{m+n}$ (respectively
$\jj_{0|n}\colon \DS_m\to \DS_{m+n}$) on morphism categories. The property (ii)
is obvious. 

For (iii) we clearly have associativity $(m+n)+p=m+(n+p)$ on objects. The remaining properties follow from the following identities of color-shifting functors:
\begin{align*}
\jj_{m|0}\circ \jj_{n|0} &= \jj_{m+n,0} &\DS_p &\to \DS_{m+n+p}& h &\mapsto \id_{m+n} \boxtimes h\\  
\jj_{m|0}\circ \jj_{0|p} &= \jj_{0|p} \circ \jj_{m,0}  &\DS_n &\to \DS_{m+n+p}& g &\mapsto \id_m\boxtimes g \boxtimes \id_{p}\\ 
\jj_{0|n+p} &= \jj_{0|p} \circ \jj_{0,n} &\DS_m &\to \DS_{m+n+p}& f &\mapsto f \boxtimes \id_{n+p}.
\end{align*}
(iv) is obvious by the definition of the functors from \Cref{shifts}. (v) is
obvious, since $\jj_{0|n}(\id_m)=\id_{n+m}=\jj_{m|0}(\id_n)$. (vi) and (vii) are
equalities for the $2$-morphisms $\boxtimes_{f,g}$ which are equivalent to
\eqref{slideboxtimes} via \Cref{thm:equivalence}. Diagrammatically, they look
like
    \begin{equation*}
\begin{tikzpicture}[anchorbase,scale=.3,tinynodes]
    \draw[bla] (0,-2) \pu(0,-0.01) \pu  (6,4); 
    \draw[dotted] (1,-2) \pu(1,-0.01) \pu  (7,4);
    \draw[bla] (2,-2) \pu (2,-0.01) \pu  (8,4); 
    \draw[bla] (6,-2) \pu (6,-0.01) \pu (0,4);
    \draw[dotted] (7,-2) \pu (7,-0.01) \pu  (1,4);
    \draw[bla] (8,-2) \pu (8,-0.01) \pu (2,4); 
    \draw[fill=white] (5.8,-1.75) rectangle (8.2,-.25);
    \node at (7,-1) {$\beta$};
    \end{tikzpicture}
    \;\;=\;\;
    \begin{tikzpicture}[anchorbase,scale=-.3,tinynodes]
    \draw[bla] (0,-2) \pu(0,-0.01) \pu  (6,4); 
    \draw[dotted] (1,-2) \pu(1,-0.01) \pu  (7,4);
    \draw[bla] (2,-2) \pu (2,-0.01) \pu  (8,4); 
    \draw[bla] (6,-2) \pu (6,-0.01) \pu (0,4);
    \draw[dotted] (7,-2) \pu (7,-0.01) \pu  (1,4);
    \draw[bla] (8,-2) \pu (8,-0.01) \pu (2,4); 
    \draw[fill=white] (5.8,-1.75) rectangle (8.2,-.25);
    \node at (7,-1) {$\beta$};
    \end{tikzpicture}, \quad
    \begin{tikzpicture}[anchorbase,xscale=-.3,yscale=.3,tinynodes]
    \draw[bla] (0,-2) \pu(0,-0.01) \pu  (6,4); 
    \draw[dotted] (1,-2) \pu(1,-0.01) \pu  (7,4);
    \draw[bla] (2,-2) \pu (2,-0.01) \pu  (8,4); 
    \draw[bla] (6,-2) \pu (6,-0.01) \pu (0,4);
    \draw[dotted] (7,-2) \pu (7,-0.01) \pu  (1,4);
    \draw[bla] (8,-2) \pu (8,-0.01) \pu (2,4); 
    \draw[fill=white] (5.8,-1.75) rectangle (8.2,-.25);
    \node at (7,-1) {$\alpha$};
    \end{tikzpicture}
    \;\;=\;\;
    \begin{tikzpicture}[anchorbase,xscale=.3,yscale=-.3,tinynodes]
    \draw[bla] (0,-2) \pu(0,-0.01) \pu  (6,4); 
    \draw[dotted] (1,-2) \pu(1,-0.01) \pu  (7,4);
    \draw[bla] (2,-2) \pu (2,-0.01) \pu  (8,4); 
    \draw[bla] (6,-2) \pu (6,-0.01) \pu (0,4);
    \draw[dotted] (7,-2) \pu (7,-0.01) \pu  (1,4);
    \draw[bla] (8,-2) \pu (8,-0.01) \pu (2,4); 
    \draw[fill=white] (5.8,-1.75) rectangle (8.2,-.25);
    \node at (7,-1) {$\alpha$};
    \end{tikzpicture}.
    \end{equation*}
Diagrammatically, these relations hold in fact  for any $2$-morphisms $\alpha$, $\beta$,  because the last two compatibility relations \eqref{eq:comp} and the first two parabolic relations \eqref{eq:para} guarantee that (linear combinations of) diagrams can slide through and cross through any bundle of strands,  provided the colors involved in the respective two bundles of strands are from sets $C_1$ and $C_2$ which are all mutually distant.

Finally, (viii) is via
\Cref{thm:equivalence} equivalent to \eqref{slideboxtimescomp}; diagrammically we have
\begin{equation*}
\begin{tikzpicture}[anchorbase,scale=.2,tinynodes]
    \draw[bla] (0,-0.01) \pu  (6,4) \pu (12,8); 
    \draw[dotted] (1,-0.01) \pu  (7,4) \pu (13,8);
    \draw[bla] (2,-0.01) \pu  (8,4) \pu (14,8); 
    \draw[bla] (6,-0.01) \pu (0,4) \pu (0,8);
    \draw[dotted] (7,-0.01) \pu  (1,4) \pu (1,8);
    \draw[bla] (8,-0.01) \pu (2,4) \pu (2,8); 
    \draw[bla] (12,-0.01) \pu (12,4) \pu (6,8);
    \draw[dotted] (13,-0.01) \pu (13,4) \pu  (7,8);
    \draw[bla] (14,-0.01)\pu (14,4) \pu (8,8); 
    \end{tikzpicture}
    \;\;=\;\;
\begin{tikzpicture}[anchorbase,scale=.2,tinynodes]
    \draw[bla] (0,-0.01) \pu (12,8); 
    \draw[dotted] (1,-0.01) \pu  (13,8);
    \draw[bla] (2,-0.01) \pu (14,8); 
    \draw[bla] (6,-0.01) \pu (0,8);
    \draw[dotted] (7,-0.01) \pu (1,8);
    \draw[bla] (8,-0.01) \pu (2,8); 
    \draw[bla] (12,-0.01) \pu (6,8);
    \draw[dotted] (13,-0.01) \pu (7,8);
    \draw[bla] (14,-0.01) \pu (8,8); 
    \end{tikzpicture}, \quad
    \begin{tikzpicture}[anchorbase,xscale=.2,yscale=-.2,tinynodes]
    \draw[bla] (0,-0.01) \pu  (6,4) \pu (12,8); 
    \draw[dotted] (1,-0.01) \pu  (7,4) \pu (13,8);
    \draw[bla] (2,-0.01) \pu  (8,4) \pu (14,8); 
    \draw[bla] (6,-0.01) \pu (0,4) \pu (0,8);
    \draw[dotted] (7,-0.01) \pu  (1,4) \pu (1,8);
    \draw[bla] (8,-0.01) \pu (2,4) \pu (2,8); 
    \draw[bla] (12,-0.01) \pu (12,4) \pu (6,8);
    \draw[dotted] (13,-0.01) \pu (13,4) \pu  (7,8);
    \draw[bla] (14,-0.01)\pu (14,4) \pu (8,8); 
    \end{tikzpicture}
    \;\;=\;\;
\begin{tikzpicture}[anchorbase,xscale=.2,yscale=-.2,tinynodes]
    \draw[bla] (0,-0.01) \pu (12,8); 
    \draw[dotted] (1,-0.01) \pu  (13,8);
    \draw[bla] (2,-0.01) \pu (14,8); 
    \draw[bla] (6,-0.01) \pu (0,8);
    \draw[dotted] (7,-0.01) \pu (1,8);
    \draw[bla] (8,-0.01) \pu (2,8); 
    \draw[bla] (12,-0.01) \pu (6,8);
    \draw[dotted] (13,-0.01) \pu (7,8);
    \draw[bla] (14,-0.01) \pu (8,8); 
    \end{tikzpicture}.
    \end{equation*}
This finishes the verification of the data and properties from  \cite[Lemma 4]{MR1402727}. 
\end{proof}

\begin{theorem}
    \label{thm:ssmtwocatb}
    Consider the semistrict monoidal structure on the locally graded $\k$-linear $2$-category
    $\DS$ from \Cref{thm:ssmtwocat}. By
applying the construction $\Kb$ hom-category-wise we obtain a semistrict monoidal locally graded $\k$-linear $2$-category $\Kbloc(\DS)$.     
\end{theorem}
\begin{proof}
By \Cref{thm:moninherit} and \Cref{cor:moninherit}, we can perform the construction $\Kb$ on the level of
 hom-categories without losing the strictness of the composition of $1$-morphisms.
 We thus obtain a locally graded $\k$-linear $2$-category $\Kbloc(\DS)$.

To provide the semistrict monoidal structures, we follow closely the proof of \Cref{thm:ssmtwocat} and again turn to the list 
(1)-(7) of data and the list  (i)-(viii) of properties from \cite[Lemma 4]{MR1402727}. The
assignments (1) and (2) are taken verbatim as in the case of $\DS$ and the
verification of (i) and (ii) are identical.
Next we note that the parabolic induction functors from \Cref{def:boxtimesdef} extend from $\DS$ to 
$\Kbloc(\DS)$. They provide the data (3),
(4), (5), and (6) and it is clear that they still satisfy (iii). In place of the tensorator $2$-morphisms \eqref{pseudonat} from \Cref{pseudo7},
we now use degree zero chain maps
\[
\boxtimes_{c,c'}\colon
      (c \boxtimes \id) \hcomp (\id \boxtimes c')\to(\id \boxtimes c')\hcomp (c \boxtimes \id),
\]
whose non-zero components are given by (adjusted by Koszul signs)  $2$-morphisms of the form
\eqref{pseudonat} used previously.  Namely, the terms of the
source complex $(c \boxtimes \id) \hcomp (\id \boxtimes c')$ are in canonical
bijection to pairs of elements corresponding to the factors, i.e. one element comes from the complex $c$ and the other from $c'$.  The same statement holds for the terms in 
the target
complex $(\id \boxtimes c')\hcomp (c \boxtimes \id)$.  Together, we obtain 
a canonical bijection between the terms in the source with those in the target complex. Then take, up to signs,  the
morphisms \eqref{pseudonat} as components of the morphisms spaces between  the partners under the constructed bijection, and let  the components between
all other terms be zero. An additional sign has to be included exactly when both factor terms carry an odd $\Sigma$-shift. 
The remaining relations (iv)--(viii) are inherited from the
corresponding relations in $\DS$.  This provides the data and properties required by  \cite[Lemma~4]{MR1402727}. 
\end{proof}
%\begin{remark}\label{homwisetobi}
 \begin{rem} In the proof of
    \Cref{thm:ssmtwocatb} we have, apart from the strictness assumptions, not used any special properties of $\DS$. Thus, the statement of  \Cref{thm:ssmtwocatb}  holds for any locally (graded) $\k$-linear $2$-category in place of $\DS$. For example, one could work with the diagrammatic model for Soergel bimodules from \Cref{rem:diagtosbim} instead.
\end{rem}

\begin{rem}
    \label{rem:semmonchb}
    We expect that the semistrict monoidal structure on $\DS$ also enables the
    construction of semistrict monoidal structures on the locally graded dg-$2$-category $\Chbloc(\DS)$ and on the locally bigraded $\k$-linear
    $2$-category $H^*(\Chbloc(\DS))$. A proof could start very similarly, but in comparison to 
    \Cref{thm:ssmtwocatb}, more care is needed when verifying (iv), since this relation concerns $2$-morphisms $\alpha$, $\beta$ which are now, in contrast to the case of $\Kbloc(\DS)$, 
    possibly of 
    nonzero homological degree. Additional Koszul signs appear which have to be dealt with.  Since we only need
    the statement for $\Kbloc(\DS)$, we refrain from a further discussion of semistrict
    monoidal structures in the other cases. 
\end{rem}

\section{Braiding on Soergel bimodules}
\label{sec:braiding}
Our goal is to explicitly describe the braiding on Soergel bimodules in a
dg-model. More precisely, we will first work in the monoidal dg-bicategory
$\Chbloc(\SBim)$  and parallel or afterwards in its diagrammatic
analog $\Chbloc(\DS)$. In this section we describe the braiding 1-morphisms
(certain chain complexes) and in the following sections the requisite (higher)
naturality data, i.e. chain maps and higher homotopies of such. 

\subsection{Rouquier complexes}
\begin{definition} For  $n\geq 2$ let $\Braidg_n$ be the braid group with
(Artin) generators  $\Ag_i$, $i\in [1;n-1]$.  Given a braid generator $\Ag_i$ or its
inverse, we will consider the following complexes\footnote{Here and in the following, $\uwave{B}$ indicates that $B$ is concentrated in degree zero of the complex} in $\Chb(\SBim_n)$:
\begin{equation}
    \label{eqn:Rouq-gen-alg}
    \Rouq(\Ag_i) 
    := \left( 0 \xrightarrow{} \uwave{B_i} \xrightarrow{\epsilon} R\langle-1\rangle \xrightarrow{} 0 \right)\, , \quad
    \Rouq(\Ag_i^{-1}) 
    := \left( 0 \xrightarrow{} R\langle 1\rangle \xrightarrow{\eta} \uwave{B_i} \xrightarrow{} 0 \right).
\end{equation}
Here $\eta$ and $\epsilon$ are the counit and unit maps from \Cref{BSFrob}.
Explicitly, $\eta$ is induced by the multiplication map $B_i=R \otimes_{R^{s_i}}
R \langle -1 \rangle \to  R\langle-1\rangle$ and $\epsilon$ is the bimodule map
determined by $1 \mapsto x_i\otimes 1 - 1\otimes x_{i+1}$. Their diagrammatic
analogs in $\Chb(\DS_n)$ are given by $\begin{tikzpicture}[anchorbase,scale=-.2]
\draw[bl] (0,-0.01) \pu (0,-1.01); \fill[bl] (0,-1) circle (2.5mm); 
            \end{tikzpicture}
            $ and $
           \begin{tikzpicture}[anchorbase,scale=.2]
            \draw[bl] (0,-0.01) \pu (0,-1.01);
            \fill[bl] (0,-1) circle (2.5mm); 
            \end{tikzpicture}$. An expression
$\ub=\Ag_{i_1}^{\epsilon_1}\cdots\Ag_{i_r}^{\epsilon_r}$ with
$\epsilon_j\in\{\pm\}$  is called a \emph{braid word} with corresponding
element $\beta\in \Braidg_n$. The word is \emph{positive} if $\epsilon_j=1$ for
$j\in[1;r]$. 

Given  $\ub$, define
\begin{equation*}
    \label{eqn:Rouq-alg}
    \Rouq(\ub) 
    := 
    \Rouq(\Ag_{i_1}^{\epsilon_1}) \hcomp \cdots \hcomp \Rouq(\Ag_{i_r}^{\epsilon_r}),
\end{equation*}
where we make use of the horizontal composition\footnote{This agrees also with the derived tensor product, since Soergel bimodules are free from either side.} $\hcomp=\otimes_R$ in
$\Chb(\SBim_n)$. By convention, the empty braid word gives $\Rouq(\emptyset)=R$. 
\end{definition}
The complexes  $ \Rouq(\ub) $ are called \emph{Rouquier complexes}, since they
were first thoroughly studied by Rouquier who proved in particular in
\cite{0409593} that,  up to canonical homotopy equivalence, these complexes are
independent of the chosen braid word representing $\beta$. More precisely the
following holds:

\begin{thm}[Rouquier canonicity]
    \label{thm:Rouquier-canonicity} Let $n\in \N_0$ and $\ub_1$ and $\ub_2$ be
    braid words representing the same braid $\beta\in \Braidg_n$, then in
    $\Chb(\SBim_n)$, and equivalently in $\Chb(\DS_n)$, there exist homotopy
    equivalences
    \[\psi_{\ub_1,\,\ub_2}\colon \Rouq(\ub_1) \to \Rouq(\ub_2),\] which form a
transitive system, i.e. if $\ub_3$ is a third braid word representing the same
braid, then we have
\[\psi_{\ub_2,\,\ub_3}\vcomp \psi_{\ub_1,\,\ub_2} \htc \psi_{\ub_1,\,\ub_3}.\]
Moreover, if $\ub_1'$, $\ub_2'$ is another pair of braid words representing a
(different) braid then 
\begin{equation}
\label{eqn:Rouquier-can-tensor}
  \psi_{\ub_1\ub_1',\ub_2\ub_2'} \htc \psi_{\ub_1,\ub_2}\hcomp \psi_{\ub_1',\ub_2'}. 
\end{equation}
\end{thm}
This rephrasing of Rouquier's results from \cite{0409593} also slightly strengthens \cite[Proposition 2.19]{elias2017categorical}.
From now on we may abuse notation and write $\Rouq(\beta)$ instead of $\Rouq(\ub)$.

\begin{lemma}\label{lem:dualrouquier}  The negative Rouquier complex
$F(\sigma^{-1}_i)$ is the dual of the positive Rouquier complex $F(\sigma_i)$.
More generally, the functor $F$ sends mutually inverse braids to
mutually dual Rouquier complexes.
\end{lemma}
\begin{proof}
    The symmetries of the diagrammatic category $\DS_n$ from
    \Cref{rem:dualities} extend to symmetries of $\Chb(\DS_n)$ with one caveat:
    contravariant symmetries require negating both, the quantum
    grading shifts and the homological degrees of complexes. By inspection we have
    $
        F(\sigma^{-1}_i) = r_{yz}(F(\sigma_i))
    $ which is the dual of $F(\sigma_i)$ by \Cref{cor:selfduality}. The claim for general braids follows directly.
\end{proof}
\begin{definition}\label{def:invertible}
Recall that an \emph{invertible object} in a monoidal category is a dualizable
object, for which the unit and counit of the duality are isomorphisms. More
generally, a $1$-morphism in a bicategory is an \emph{(adjoint) equivalence} if
it has a (left or right) adjoint, for which the unit and counit $2$-morphisms
are isomorphisms.
\end{definition}

\begin{corollary}
\label{cor:Rouqinv}
    The Rouquier complexes $\Rouq(\beta)$ for $\beta\in \Braidg_n, n\in \N_0$, are invertible in $\Kb(\SBim_n)$  and
    thus invertible up to homotopy in $\Chb(\SBim_n)$. The same result holds
    equivalently for $\DS_n$ in place of $\SBim_n$.
\end{corollary}

For later use, let us spell out what this means for a Rouquier complex
$X:=F(\ub)$. More generally, we describe here the data of an object $X$ in a
monoidal dg-category (with monoidal unit $\one$), which is \emph{invertible up to
homotopy}. First of all, dualizability of $X$ requires the existence of another
object $X'$ (e.g. $X'=F({\ub}^{-1})$ for 
$X=F({\ub}$)), as well as chain maps:
\begin{gather*}
    \ev_{X} \colon X \hcomp X' \to \one, \;\; \coev_{X} \colon \one \to X'\hcomp X,\;\;
    \ev_{X'} \colon X' \hcomp X \to \one, \;\;\coev_{X'} \colon \one \to X\hcomp X',
\end{gather*} 
which satisfy the snake relations \emph{up to homotopy}:
\begin{equation}\label{homotopysnakes}
  \begin{split}
    (\ev_{X} \hcomp \id_{X})\vcomp (\id_{X}\hcomp \coev_{X}) \simeq \id_{X}, 
    (\id_{X'}\hcomp \ev_{X}) \vcomp(\coev_{X} \hcomp \id_{X'}) \simeq \id_{X'},
    \\
(\ev_{X'} \hcomp \id_{X'})\vcomp(\id_{X'}\hcomp \coev_{X'}) \simeq \id_{X'}, 
    (\id_{X}\hcomp \ev_{X'}) \vcomp(\coev_{X'} \hcomp \id_{X}) \simeq \id_{X}.
  \end{split}
\end{equation}
\begin{definition}\label{def:hinvertibility}
\emph{Invertibility up to homotopy} requires additionally to \eqref{homotopysnakes}, that the pair $\ev_{X}$, $\coev_{X'}$ as well as $\ev_{X'}$,
$\coev_{X}$ form mutually inverse homotopy equivalences:
\begin{gather*}
    \ev_{X} \vcomp\coev_{X'}  \simeq \id_{\one}, \quad
    \coev_{X'} \vcomp\ev_{X}   \simeq \id_{X\hcomp X'},\\
    \ev_{X'} \vcomp\coev_{X}  \simeq \id_{\one}, \quad 
    \coev_{X} \vcomp \ev_{X'}   \simeq \id_{X'\hcomp X}. 
\end{gather*}
\end{definition}

\begin{remark}
For Rouquier complexes $X=\Rouq(\ub)$ and $X':=F({\ub}^{-1})$, the chain maps
$\ev_{X}$, $\coev_{X}$ and their analogues for $X'$ can be constructed as
composites of elementary Reidemeister 2 homotopy equivalences using \Cref{thm:Rouquier-canonicity}, i.e. from the
$\psi_{\ub, 1_n}$ connecting $\ub=\Ag_i\Ag_i^{-1}$ or $\ub=\Ag_i^{-1}\Ag_i$ and
the identity braid element $1_n\in \Braidg_n$. It is interesting to observe that in this setting the snake relations as well as the relations
\begin{equation}\label{onthenose}
    \ev_{X} \vcomp\coev_{X'}  = \id_{\one}, \quad \ev_{X'} \vcomp\coev_{X}  = \id_{\one}
\end{equation}
hold \emph{on the nose} (i.e.~not just up to homotopy).
The relations \eqref{onthenose} are generalizations of the Carter-Saito movie move number 9, see
\cite[Braid movie move 2]{MR2721032} for details.
\end{remark}

\begin{rem}
    \label{rem:Rouquierswap}
    Consider $n \geq 2$ and $R=R_n$ as above. For every $w\in S_n$ we  denote by 
    $R_{\circlearrowleft w}$ the graded $R$-bimodule which is obtained from
    the regular $R$-bimodule by twisting the right-action by $w$: i.e. $r\in R$
    acts on $R_{\circlearrowleft w}$ from the right as multiplication by $w(r)$.
    For non-trivial $w$, the $R$-bimodule $R_{\circlearrowleft w}$ is \emph{not}
    an object of $\SBim_n$. 

    However, there are short exact sequences of $R$-bimodules
    \begin{align*}
        0 \to R_{\circlearrowleft s_i}\langle 1\rangle \xrightarrow{f}  B_i \xrightarrow{\epsilon} R\langle -1 \rangle,\quad
        0 \to R \langle 1\rangle \xrightarrow{\eta} &B_i \to R_{\circlearrowleft s_i} \langle -1 \rangle
    \end{align*}
    where $f\colon 1\mapsto x_i\otimes 1 - 1\otimes x_i$.
    In the bounded derived category of graded $R$-bimodules, the swap bimodule
    $R_{\circlearrowleft s_i}$ is thus equivalent to either of the shifted
    Rouquier complexes $F(\sigma_i)\langle -1\rangle$ and $F(\sigma^{-1})\langle
    1 \rangle$. Similarly, all Rouquier complexes become derived equivalent to
    permutation bimodules. Since there is an isomorphism of bimodules $R_{\circlearrowleft w}\otimes_RR_{\circlearrowleft w'}\cong R_{\circlearrowleft ww'}$ via $f\otimes g\mapsto fw(g)$ for any $w,w'\in S_n$, the swap bimodule gives a symmetric braiding. To obtain an interesting
    (non-symmetric!) braiding, it is thus essential to work up-to-chain-homotopy,
    rather than up-to-quasi-isomorphism. Nevertheless, the comparison with 
permutation bimodules is conceptually important and motivates the grading
    shifts in the next definition.
\end{rem}

\subsection{Braiding 1-morphisms} 
The braiding $1$-morphisms in $\Chbloc(\SBim)$ and
$\Chbloc(\DS)$ are, as in \cite{liu2024braided},  given by (shifted) Rouquier complexes for certain shuffle braids:

\begin{definition}
For $m,n\geq 0$, we define the \emph{Rouquier complexes for cabled crossings}:
\begin{equation}\label{cabledcrossing}
\begin{aligned}
    \cabledcross_{m,n}&:=\Rouq( (\Ag_{n}\cdots\Ag_{1})\cdots (\Ag_{i+n-1}\cdots\Ag_{i}) \cdots (\Ag_{m+n-1}\cdots\Ag_{m}) )\langle -m n \rangle,\\
    \cabledcross'_{m,n}&:=\Rouq( (\Ag\inv_{n}\cdots\Ag\inv_{m+n-1}) \cdots (\Ag\inv_{i}\cdots\Ag\inv_{i+m-1}) \cdots (\Ag\inv_{1}\cdots\Ag\inv_{m}) )\langle m n \rangle.
\end{aligned}
\end{equation}
The underlying braids, and by extension, also the complexes $\cabledcross_{m,n}$
(resp. $\cabledcross'_{m,n}$) will be called \emph{positive} (resp. \emph{negative})
\emph{cabled crossings}. The special cabled crossings $\cabledcross_{m,1}$,
$\cabledcross'_{m,1}$, $\cabledcross_{1,n}$, and $\cabledcross'_{1,n}$ will also be
called \emph{Coxeter braids}, since they are braid lifts of Coxeter words.
\end{definition}
\begin{figure}[h]
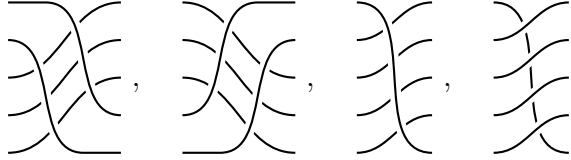

\vspace{-5mm}
    \[
\cabledcrossfig
    \]
    \caption{Cabled crossings $\cabledcross_{2,3}$ and $\cabledcross'_{3,2}$,
    Coxeter braids $\cabledcross_{1,4}$ and $\cabledcross'_{1,4}$. Braid words and their graphical representatives are read right to left.}
    \label{fig:cabledcross}
    \end{figure}
    
\begin{lemma}
    \label{lem:cabled-crossings-from-coxeter}
    Cabled crossings are assembled from Coxeter braids, i.e.  (for $m,n\geq 0$) 
\begin{align*}
\cabledcross_{m,n} 
&= 
(\cabledcross_{1,n}\boxtimes \one_{m-1}) 
\hcomp \cdots \hcomp 
(\one_{m-1-i} \boxtimes \cabledcross_{1,n}\boxtimes\one_{i})
\hcomp \cdots \hcomp
(\one_{m-1}\boxtimes \cabledcross_{1,n})\\
& \hte
(\one_{n-1}\boxtimes \cabledcross_{m,1})
\hcomp \cdots \hcomp 
(\one_{i} \boxtimes \cabledcross_{m,1}\boxtimes\one_{n-1-i})
\hcomp \cdots \hcomp
(\cabledcross_{m,1}\boxtimes \one_{n-1} ),
\\
\cabledcross'_{m,n} 
&= 
(\one_{n-1}\boxtimes \cabledcross'_{m,1})
\hcomp \cdots \hcomp 
(\one_{i} \boxtimes \cabledcross'_{m,1}\boxtimes\one_{n-1-i})
\hcomp \cdots \hcomp
(\cabledcross'_{m,1}\boxtimes \one_{n-1} )\\
& \hte
(\cabledcross'_{1,n}\boxtimes \one_{m-1}) 
\hcomp \cdots \hcomp 
(\one_{m-1-i} \boxtimes \cabledcross'_{1,n}\boxtimes\one_{i})
\hcomp \cdots \hcomp
(\one_{m-1}\boxtimes \cabledcross'_{1,n}).
\end{align*}
\end{lemma}
\begin{proof} The equalities follow immediately from the definitions in  \eqref{cabledcrossing}. The
homotopy equivalences come from applying braid relations, see
\Cref{thm:Rouquier-canonicity}.
\end{proof}

\section{Naturality for the prebraiding}
\label{sec:natprebraid}

In this section we fix $m,n\in \N_0$ and consider the functor $\boxtimes\colon
    \DS_m\times \DS_n \to \DS_{m+n}$ alongside its opposite version
    \begin{equation}\label{eq:parabolicindop}
    \boxtimes^{\opp} \colon \DS_m\times \DS_n \to \DS_{m+n}, \quad (M,N)\mapsto  M \boxtimes^{\opp} N := N \boxtimes M.
    \end{equation}

\subsection{Slide chain maps}
\label{ssec:slides-homotopies}
Next, we provide the naturality data for the braiding. The crucial concept
hereby are the following slide chain maps which will be constructed. The development in this section is entirely parallel to \cite[Section 2.5]{liu2024braided}, but we include a self-contained version for the reader's convenience.

\begin{proposition}
    \label{prop:sliding-objects}  Given a cabled crossing,  \eqref{cabledcrossing},   there is for each pair of $1$-morphisms $Y_1\in \DS_m$, $Y_2\in \DS_n$ a  respective homotopy equivalences of chain complexes in $\Chb(\DS_{m+n})$: 
     \begin{equation}\label{defslidemaps}
     \begin{aligned}
    \slide_{Y_1,Y_2}\colon & \cabledcross_{m,n} \hcomp (Y_1 \boxtimes Y_2) \longrightarrow (Y_1 \boxtimes^\opp Y_2)\hcomp \cabledcross_{m,n} = (Y_2 \boxtimes Y_1)\hcomp \cabledcross_{m,n},\\
\slide'_{Y_1,Y_2}\colon & \cabledcross'_{m,n} \hcomp (Y_1 \boxtimes Y_2)
    \longrightarrow (Y_1 \boxtimes^\opp Y_2)\hcomp \cabledcross'_{m,n} = (Y_2
    \boxtimes Y_1)\hcomp \cabledcross'_{m,n}.
    \end{aligned}
    \end{equation}
    \end{proposition}
    We will also write  $\slide_{Y}:=\slide_{Y_1,Y_2}$, $\slide'_{Y}:=\slide'_{Y_1,Y_2}$ when $Y=Y_1\boxtimes Y_2$ and $m$ and
    $n$ are clear from the context.
\begin{definition}
    The chain maps $\slide_{Y_1,Y_2}$  and $\slide'_{Y_1,Y_2}$ from
    \Cref{prop:sliding-objects} (which are explicitly given in that proof and in 
    \Cref{lem:atomicslide}) will be called \emph{slide chain maps}.
\end{definition}

We  illustrate the domain and codomain of
$\slide_{Y_1,Y_2}$ as follows (here: $m=2$, $n=3$): 
\[
    \slidefig
        \]
          
       The proof of \Cref{prop:sliding-objects}, i.e. the construction of the chain maps $\slide_{Y_1,Y_2}$ and
    $\slide'_{Y_1,Y_2}$ proceeds through a sequence of steps. We start with the  
        \emph{atomic slide chain maps} $\slide_{\one_1,B_1}$ for
$(m,n)=(1,2)$ and $\slide_{B_1,\one_1}$ for $(m,n)=(2,1)$. 
        \begin{proposition}[Atomic slide chain maps]\label{lem:atomicslide} % PW: made this label more descriptive
      There are chain maps 
\[\slide_{\one_1,B_1}:= \hspace{-.5cm}\FigSlide,\quad \slide_{B_1,\one_1}:=\hspace{-.5cm} \FigSlidee \] 
which are invertible up to homotopy. The inverses are given by the chain maps
\[\slide^{-1}_{\one_1,B_1}:= \hspace{-.5cm}\FigSlidei, \quad
\slide^{-1}_{B_1,\one_1}:= \hspace{-.5cm}\FigSlideei \] 
\end{proposition}
Observe that the relevant chain complexes and chain maps for the two cases are
related by swapping the colours red and blue. This is a consequence of the
duality $r_x$ from \Cref{rem:dualities}, which preserves the sign of Rouquier
complexes.

\begin{proof}
The proof is given by an explicit calculation. As an example (the
remaining cases are checked analogously) we show that
$\slide^{-1}_{\one_1,B_1}\circ\; \slide_{\one_1,B_1}$ is homotopic to the
identity by computing their difference and by exhibiting a  null-homotopy, namely 
\[\slide^{-1}_{\one_1,B_1}\circ\; \slide_{\one_1,B_1} - \Id = \left(\FigSlidec\right) = d \left(\FigSlideh \right)\] 
realizing this difference.
  \end{proof}      

  \begin{remark} Readers familiar with the chain maps between Rouquier complexes
    associated to a Reidemeister 3 move will recognize the atomic slide chain maps
    as filtrations-preserving pieces of the former, see e.g. \cite[(3.3) and (3.4)]{MWW}.
\end{remark}

Since every Bott-Samelson bimodule is a
(monoidal and horizontal) composition of generating Bott-Samelson bimodules on
two strands, the homotopy equivalences needed for \Cref{prop:sliding-objects}
can be constructed from the atomic slide chain maps along the following scheme:
  \[
    \begin{tikzpicture}[anchorbase,xscale=-.5,yscale=.5]
        \draw[thick] (-1.5,2) to (1,2) \pr (4,0);
        \draw[thick] (-1.5,3) to (1,3) \pr (4,1);
        \draw[thick] (-1.5,4) to (1,4) \pr (4,2);
        \draw[wh] (1,1) \pr (3,4) to  (4,4);
        \draw[thick] (-1.5,1) to (1,1) \pr (3,4) to (4,4);
        \draw[wh] (1,0) to (2,0) \pr (4,3);
        \draw[thick] (-1.5,0) to (1,0) to (2,0) \pr (4,3);
        \draw[thick,fill=white] (0,-.2) rectangle (1,1.2);
        \draw[thick,fill=white] (-1,1.8) rectangle (0,4.2);
        \node at (.5,.5) {$Y_1$};
        \node at (-.5,3) {$Y_2$};
    \end{tikzpicture}
    \;\; \xrsquigarrow{\tiny\text{reduce to}} \;\;
    \begin{tikzpicture}[anchorbase,xscale=-.5,yscale=.5]
        \draw[thick] (-.5,2) to (1,2) \pr (4,0);
        \draw[thick] (-.5,3) to (1,3) \pr (4,1);
        \draw[thick] (-.5,4) to (1,4) \pr (4,2);
        \draw[wh] (1,1) \pr (3,4) to  (4,4);
        \draw[thick] (-.5,1) to (0,1) to (1,1) \pr (3,4) to (4,4);
        \draw[wh] (2,0) \pr (4,3);
        \draw[thick, gray] (-.5,0) to(0,0) to (1,0) to (2,0) \pr (4,3);
        \draw[thick, fill=white] (1,1.8) rectangle (0,4.2);
        %\node at (.5,3.7) {$Y_2'$};
        \node at (.5,3) {$B_i$};
        %\node at (.5,2.3) {$Y_2$};
    \end{tikzpicture}
    \;\; \xrsquigarrow{\tiny\text{reduce to}} \;\;
    \begin{tikzpicture}[anchorbase,xscale=-.5,yscale=.5]
        \draw[thick] (-.5,2) to (1,2) \pr (4,0);
        \draw[thick] (-.5,3) to (1,3) \pr (4,1);
        \draw[thick, gray] (-.5,4) to (0,4) to (1,4) \pr (4,2);
        \draw[wh] (1,1) \pr (3,4) to  (4,4);
        \draw[thick] (-.5,1) to  (0,1) to (1,1) \pr (3,4) to (4,4);
        \draw[wh] (2,0) \pr (4,3);
        \draw[thick, gray] (-.5,0) to  (0,0) to (1,0) to (2,0) \pr (4,3);
        \draw[thick, fill=white] (1,1.8) rectangle (0,3.2);
        \node at (.5,2.5) {$B_i$};
    \end{tikzpicture}
\]
(Essentially the same argument would work for any monoidal bicategory generated
by a single object and one endomorphism of its tensor square.)        
\begin{proof}[Proof of \Cref{prop:sliding-objects}]
 We will focus on the version for the positive cabled crossing, since the maps  $\slide'_{Y_1,Y_2}
$ for sliding through the inverse braiding $1$-morphisms,  arise from the 
   slide maps  $\slide_{Y_2,Y_1}$ by composing with inverse braiding $1$-morphisms on both
    sides.
    
     We start now by reducing to the case when the $1$-morphism
    $Y_1\boxtimes Y_2$ is a generating one, i.e. of the form $B_i\boxtimes \one$ or $\one \boxtimes B_j$.
    Otherwise, we can decompose into generators:
    \[ Y_1\boxtimes Y_2= (Y_1\boxtimes \one) \hcomp (\one \boxtimes Y_2) =
    (B_{i_1}\boxtimes \one) \hcomp \cdots \hcomp (B_{i_a}\boxtimes \one) \hcomp (\one
    \boxtimes B_{j_1})\hcomp \cdots \hcomp (\one
    \boxtimes B_{j_b}) \] 
    and define
    \begin{align*}
        \slide_{Y_1,\one}&:= 
        (\id_{\one \boxtimes (B_{i_1}\hcomp \cdots \hcomp B_{i_{a-1}})}\hcomp \slide_{B_{i_{a}},\one}) 
        \vcomp \cdots \vcomp
        (\slide_{B_{i_{1}},\one} \hcomp \id_{(B_{i_1}\hcomp \cdots \hcomp B_{i_{a-1}}) \boxtimes \one}),
        \\
        \slide_{\one,Y_2}&:= 
        (\id_{(B_{j_1}\hcomp \cdots \hcomp B_{j_{b-1}})\boxtimes \one }\hcomp \slide_{\one, B_{j_{b}}}) 
        \vcomp \cdots \vcomp
        (\slide_{\one, B_{j_{1}}} \hcomp \id_{\one\boxtimes (B_{j_1}\hcomp \cdots \hcomp B_{j_{b-1}})}),
        \\
        \slide_{Y_1,Y_2} &:= (\id_{\one \boxtimes Y_1} \hcomp \slide_{\one,Y_2})  \vcomp (\slide_{Y_1,\one} \hcomp \id_{Y_2}).
    \end{align*}

    Now we turn to defining $\slide_{B,\one_n}$ and $\slide_{\one_m,B}$, where
    $B$ is a generating $1$-morphism. Here we place subscripts to distinguish the
    identity $1$-morphisms. We first consider the situation of $\slide_{\one_m,B}$ and reduce it to the case
    $m=1$, where the cabled crossing is a Coxeter braid. Indeed, suppose that
    $m>1$. Then we use the first equality from
    \Cref{lem:cabled-crossings-from-coxeter} to define $\slide_{\one_m,B}$ to be
    the composite
    \begin{gather}
        \nonumber
        \big((\slide_{\one_1,B}\boxtimes \id_{\one_{m-1}})
        \hcomp \cdots \hcomp 
        \id_{\one_{m-1-i} \boxtimes \cabledcross_{1,n}\boxtimes\one_{i}}
        \hcomp \cdots \hcomp
        \id_{\one_{m-1}\boxtimes \cabledcross_{1,n}}\big)
        \vcomp \cdots
        \\
        \label{eqn:slidecabledfromcoxeter}
        \vcomp \big(\id_{\cabledcross_{1,n}\boxtimes \one_{m-1}}
        \hcomp \cdots \hcomp 
        (\id_{\one_{m-1-i}} \boxtimes \slide_{\one_1,B}\boxtimes\id_{\one_{i}})
        \hcomp \cdots \hcomp
        \id_{\one_{m-1}\boxtimes \cabledcross_{1,n}}\big)
        \vcomp \cdots\\ \nonumber
        \vcomp \big(\id_{\cabledcross_{1,n}\boxtimes \one_{m-1}}
        \hcomp \cdots \hcomp 
        \id_{\one_{m-1-i} \boxtimes \cabledcross_{1,n}\boxtimes\one_{i}}
        \hcomp \cdots \hcomp
        (\id_{\one_{m-1}}\boxtimes \slide_{\one_1,B}) \big).
    \end{gather}
For the case of $\slide_{B,\one_n}$, we first choose chain maps $\phi$ and $\phi^{-1}$  realising
 the first homotopy equivalence in \Cref{lem:cabled-crossings-from-coxeter}, and then define $\slide_{B,\one_n}$ as the composition
 \begin{gather*}
    \nonumber
    \phi^{-1} \vcomp \big((\one_{n-1}\boxtimes \slide_{B,\one_1})
    \hcomp \cdots \hcomp 
    \id_{\one_{i} \boxtimes \cabledcross_{m,1}\boxtimes\one_{n-1-i}}
    \hcomp \cdots \hcomp
    \id_{\cabledcross_{m,1}\boxtimes \one_{n-1}}\big)
    \vcomp \cdots
    \\
    \vcomp \big(\id_{\one_{n-1}\boxtimes \cabledcross_{m,1}}
    \hcomp \cdots \hcomp 
    (\id_{\one_{i}} \boxtimes \slide_{B,\one_1}\boxtimes\id_{\one_{n-1-i}})
    \hcomp \cdots \hcomp
    \id_{\cabledcross_{m,1}\boxtimes \one_{n-1}}\big)
    \vcomp \cdots
    \\ \nonumber
    \vcomp \big(\id_{\one_{n-1}\boxtimes \cabledcross_{m,1}}
    \hcomp \cdots \hcomp 
    \id_{\one_{i} \boxtimes \cabledcross_{m,1}\boxtimes\one_{n-1-i}}
    \hcomp \cdots \hcomp
    (\slide_{B,\one_1}\boxtimes \one_{n-1}) \big)\vcomp \phi. 
    % \\
    % \phi^{-1} \vcomp (\one_{n-1}\boxtimes \slide_{B,\one_1})
    % \hcomp \cdots \hcomp 
    % (\one_{i} \boxtimes \slide_{B,\one_1}\boxtimes\one_{n-1-i})
    % \hcomp \cdots \hcomp
    % (\slide_{B,\one_1}\boxtimes \one_{n-1} ) \vcomp \phi    
 \end{gather*}
It only remains to construct $\slide_{\one_1,B_i}$ and $\slide_{B_j,\one_1}$ where 
$B_i$ is a generating $1$-morphism of $\DS_n$ and $B_j$ is a generating $1$-morphism of
$\DS_m$. We reduce this problem to the cases $n=2$ and $m=2$
respectively, and thus to the statement from  \Cref{lem:atomicslide}. 
The reduction uses the explicit form \eqref{defslidemaps}, namely  we define $\slide_{\one_1,B_i}$ as
the composite
\begin{align*}
    \Rouq(\Ag_{n}\cdots\Ag_{1}) \hcomp B_i =& \Rouq(\Ag_{n}\cdots\Ag_{i+1})\hcomp \Rouq(\Ag_{i}\Ag_{i-1})\hcomp \Rouq(\Ag_{i-2}\cdots\Ag_{1})\hcomp B_i \\
    \to 
    &\Rouq(\Ag_{n}\cdots\Ag_{i+1})\hcomp \Rouq(\Ag_{i}\Ag_{i-1}) \hcomp B_i \hcomp\Rouq(\Ag_{i-2}\cdots\Ag_{1}) 
   \\ \xrightarrow{\slide}
    &\Rouq(\Ag_{n}\cdots\Ag_{i+1})\hcomp B_{i-1} \hcomp \Rouq(\Ag_{i}\Ag_{i-1}) \hcomp\Rouq(\Ag_{i-2}\cdots\Ag_{1}) 
\\ \to&
 B_{i-1} \hcomp \Rouq(\Ag_{n}\cdots\Ag_{i+1})\hcomp \Rouq(\Ag_{i}\Ag_{i-1}) \hcomp\Rouq(\Ag_{i-2}\cdots\Ag_{1})\\ 
    =& B_{i-1} \hcomp \Rouq(\Ag_{n}\cdots\Ag_{1}),
\end{align*}
where the unlabelled maps are far-commutativity isomorphisms and the labelled
arrow is given by $\id \hcomp \slide_{\one_1,B_1} \hcomp \id $, i.e. it is indeed 
determined by the $n=2$ case. Analogously,   the construction of $\slide_{B_j,\one_1}$ is reduced to the case
$m=2$. 

By
construction, all slide maps in this proof are homotopy
equivalences.
\end{proof}

\begin{remark}
    The construction of the slide chain maps in \Cref{prop:sliding-objects}
    depends on a choice of homotopy equivalence in
    \Cref{lem:cabled-crossings-from-coxeter}. Different choices would produce
    homotopic slide maps. We could specify a choice once
    and for all by prescribing a specific way of rewriting braid words as
    required. However, this will not be necessary.  Instead, we will see in
    \Cref{cor:well-defined-up-to-coherent-homotopy} that the slide
    chain maps are well-defined up to coherent higher homotopy. This is all we need. 
\end{remark}

\subsection{Slide homotopies: the First Naturality Theorem}

Here we formulate the \emph{First Naturality Theorem} for the braiding:
\begin{theorem}[Naturality on Bott-Samelson bimodules up to homotopy]\hfill\\
    \label{thm:naturality}  
    Assuming $2$-morphisms $f_1 \colon Y_1 \to Y'_1$ in
    $\DS_m$ and $f_2 \colon Y_2 \to Y'_2$ in $\DS_m$ we obtain: 
    
     The diagram
    \begin{equation}
        \label{eqn:naturality2}
        \begin{tikzcd}[scale=1,column sep=2.5cm]        
            \cabledcross_{m,n} \hcomp (Y_1 \boxtimes Y_2)
            \ar[r, "\slide_{Y_1,Y_2}"]
            \ar[d, swap,"\id \hcomp (f_1\boxtimes f_2)"]
        &
        (Y_2\boxtimes Y_1)\hcomp \cabledcross_{m,n} 
            \ar[d,"(f_2\boxtimes f_1)\hcomp \id"]
            \ar[dl,swap,Rightarrow,shorten >=5ex,shorten <=5ex, "h_{f_1\boxtimes f_2}"]\\
              \cabledcross_{m,n} \hcomp Y'_1 \boxtimes Y'_2
                \ar[r, swap,"\slide_{Y'_1,Y'_2}"]
            &
            (Y'_2 \boxtimes Y'_1) \hcomp \cabledcross_{m,n} 
      \end{tikzcd}
    \end{equation}
    commutes up to homotopy; a homotopy $h_{f_1\boxtimes f_2}$ will be
    constructed in the proof.
   
    The analogous statement holds for the negative
    cabled crossings $\cabledcross_{m,n}'$, instead of $\cabledcross_{m,n}$, with their associated slide maps.
\end{theorem}

We prove \Cref{thm:naturality} in the next subsection by explicitly constructing homotopies
\[h_{f_1\boxtimes f_2}\in \Hom^{-1}_{\Chb(\DS_{m+n})}(\cabledcross_{m,n} \hcomp
(Y_1\boxtimes Y_2),(Y'_2\boxtimes Y'_1)\hcomp \cabledcross_{m,n})\] such that $d(h_{f_1\boxtimes f_2})
 =  ((f_2\boxtimes f_1)\hcomp \id) \vcomp \slide_{Y_1,Y_2} -
 \slide_{Y'_1,Y'_2}\vcomp (\id \hcomp (f_1\boxtimes f_2))$.  
 
 The construction
 has two parts. First, we show in a sequence of reduction lemmas that the construction of homotopies $h_{f_1\boxtimes f_2}$ reduces to the case where $f_1\boxtimes f_2$ is a diagrammatic generator and $m,n \geq 1$ are minimal for the given type of generator. In the second step, we explicitly describe homotopies for these model situations.  Before we proceed, we indicate the necessity of higher homotopies.
 
 \begin{example}[Higher homotopies]
    \label{exa:hiho}
    The homotopies $h_f$ constructed in the proof of \Cref{thm:naturality}
    depend on a presentation of $f$ as a (linear combination of) $\hcomp$- and
    $\vcomp$-composites of diagrammatic generators. In particular, if $f_i$ are
    composites of diagrammatic generators with associated homotopies $h_{f_i}$
    and $\sum_i \alpha_i f_i=0$ is a relation in the diagrammatic calculus of
    $\DS$, then the corresponding linear combination of homotopies $\sum_i
    \alpha_i h_{f_i}$ need \emph{not} be zero. We encourage the reader to verify (after we have given the proof of \Cref{thm:naturality}) that 
    \[
        \begin{tikzpicture}[anchorbase,xscale=.2,yscale=.2]
            \draw[rd] (0,-0.01) \pu (0,-1.01);
            \fill[rd] (0,-1) circle (2.5mm); 
            \draw[rd] (1,-0.01) \pu (1,-2.01);
            \fill[rd] (1,-2) circle (2.5mm); 
            \end{tikzpicture}
         = \begin{tikzpicture}[anchorbase,xscale=-.2,yscale=.2]
            \draw[rd] (0,-0.01) \pu (0,-1.01);
            \fill[rd] (0,-1) circle (2.5mm); 
            \draw[rd] (1,-0.01) \pu (1,-2.01);
            \fill[rd] (1,-2) circle (2.5mm);  
            \end{tikzpicture} 
            \;\; \text{in} \; \DS \quad \text{but} \quad
        h_{\begin{tikzpicture}[anchorbase,xscale=.2,yscale=.2]
            \draw[rd] (0,-0.01) \pu (0,-1.01);
            \fill[rd] (0,-1) circle (2.5mm); 
            \draw[rd] (1,-0.01) \pu (1,-2.01);
            \fill[rd] (1,-2) circle (2.5mm); 
            \end{tikzpicture}
        }
        \neq
        h_{\begin{tikzpicture}[anchorbase,xscale=-.2,yscale=.2]
            \draw[rd] (0,-0.01) \pu (0,-1.01);
            \fill[rd] (0,-1) circle (2.5mm); 
            \draw[rd] (1,-0.01) \pu (1,-2.01);
            \fill[rd] (1,-2) circle (2.5mm);  
            \end{tikzpicture}
        }
    \]
    Although the difference of the two homotopies is not zero (on the nose), they are nullhomotopic via the \emph{higher homotopy}
    \[
    (\id_{B_1}
        \hcomp h_{\begin{tikzpicture}[anchorbase,xscale=-.2,yscale=.2]
        \draw[rd] (0,-0.01) \pu (0,-1.01);
        \fill[rd] (0,-1) circle (2.5mm); 
        \end{tikzpicture}}) \vcomp         h_{\begin{tikzpicture}[anchorbase,xscale=-.2,yscale=.2]
        \draw[rd] (0,-0.01) \pu (0,-1.01);
        \fill[rd] (0,-1) circle (2.5mm); 
        \end{tikzpicture}}.
    \]
    This is the shadow of a general phenomenon: in
    \Cref{cor:well-defined-up-to-coherent-homotopy} we will see that $h_f$ is
    well-defined and uniquely determined by $f$ up to \emph{coherent higher homotopy}.
\end{example}

\begin{remark}\label{rkprebraiding} \Cref{thm:naturality} is only a partial
    naturality result for the braiding on complexes of Soergel bimodules, since
    only objects from the additive subcategory of Bott--Samuelson bimodules are
    considered. The idea of naturality of braidings restricted to a
    subcategory was formalised in  \cite[Def.~2.4.1]{liu2024braided}  using the
    concept of \emph{prebraidings},  and it was proven  that the slide maps are
    part of such a prebraiding structure.  
    
    More precisely, assume we are given a functor $F\colon{\CS'} \to
    \CS$ of monoidal categories. Then a  \emph{prebraiding} $\beta$ on $F$
    consists of the data of isomorphisms 
    \[F(x)\otimes F(y) \xrightarrow{\beta_{x,y}} F(y) \otimes F(x)\]
    for $x,y\in \overline{\DS}$, that form a \emph{natural} transformation $\otimes\circ (F\times
    F) \Rightarrow \otimes^{\mathrm{op}}\circ (F\times F)$ and satisfy the two
    hexagon axioms of braided monoidal categories, \cite[Def.~8.1.1.]{EGNO}, in
    case $F=\Id$  and obvious analogs of these for general $F$. Intuitively, we
    study braidings on $\CS$ restricted to ${\CS'}$ along $F$  (i.e.
    braidings on $\CS$ in case $F=\id$). 
    
    In  \cite[Theorem 2.5.2]{liu2024braided}, a prebraiding was established for the
    embedding of the bicategory of Bott--Samelson bimodules into the bicategory
    of complexes of Soergel bimodules, \emph{but} viewed as a functor $\iota$
    of ordinary $1$-categories by passing to isomorphism classes of
    objects on the level of homomorphism categories. \Cref{thm:naturality}
    extends this construction by the naturality with respect to (generating)
    morphisms in the morphism categories.  As observed in \cite[Theorem~2.5.6, Corollary~
    7.4.15,  \S8]{liu2024braided}, a precise formulation of prebraidings for
    higher categories can be given in terms of centralizers \cite[Definitions~2.6.1 and ~7.4.1]{liu2024braided} of monoidal subcategory inclusions
    ${\CS'}\subset\CS$. 
    
    Whereas \cite{liu2024braided} provides a fully coherent braiding on a
    monoidal $(\infty,2)$-category version of $\Kb(\SBim)$ with the prebraiding
    on $\iota$ \emph{induced} on the very lowest level, we start now with this
    concrete prebraiding on $\iota$ and work with \emph{concrete models}.
    
    Whereas \cite{liu2024braided} works in $\infty$-categorical versions of
    centralizers, we will use in \Cref{sec:extensiontocomplexes} dg- resp.
    $A_\infty$-versions of centralizers to formulate the precise setup and
    finally establish an extension of the naturality result from
    \Cref{thm:naturality} to complexes of Soergel bimodules, see \Cref{cormain}. 
    \end{remark}

\subsection{Slide homotopies: reduction lemmas and proof of naturality theorem}
\label{sec:slidereduction}
For the following reduction lemmas, there exists analogous statements for negative cabled crossings, which we omit to
streamline the exposition, since they can directly be deduced from the dualities in
\Cref{rem:dualities}.
\begin{lemma}
    \label{lem:h-vcomp}
    Let $f_1\colon Y_1\to Y_1'$, $f_2\colon Y_2\to Y_2'$ be as in \Cref{thm:naturality} with 
    the naturality square \eqref{eqn:naturality2} being commutative up to a homotopy $h_{f_1\boxtimes f_2}$. Assume the same for $f_1'\colon Y_1'\to Y_1''$, $f_2'\colon Y_2'\to Y_2''$ with homotopy $h_{f'_1\boxtimes f'_2}$. Then the naturality square \eqref{eqn:naturality2} for $({f'_1\boxtimes f'_2})\vcomp
   ( {f_1\boxtimes f_2})$ is commutative up to the homotopy
    \[ h_{({f'_1\boxtimes f'_2}) \vcomp ({f_1\boxtimes f_2})} :=  h_{f'_1\boxtimes f'_2}\vcomp (\id \hcomp ({f_1\boxtimes f_2}))  + ((f'_2\boxtimes f'_1)\hcomp
    \id) \vcomp h_{f_1\boxtimes f_2}.\] 
\end{lemma}
\begin{proof} The proof is straightforward when  expanding $d(h_{({f'_1\boxtimes f'_2}) \vcomp ({f_1\boxtimes f_2})})$.
\end{proof}
Since any 2-morphism $f_1\boxtimes f_2$ in $\DS_{m+n}$ is a $\vcomp$-composite of
generating 2-morphisms coming from $\DS_m$ or $\DS_n$, \Cref{lem:h-vcomp} allows us to construct the desired homotopy $h_{f_1\boxtimes f_2}$ in \Cref{thm:naturality} from the homotopies of the generating 2-morphisms. The next three lemmas reduce the
complexity of constructing $h_{f_1\boxtimes f_2}$ by decomposing along $\hcomp$ and $\boxtimes$.
Focussing on the domains, the reduction proceeds schematically as follows: 

\[
    \reducefig
        \]

\begin{lemma}
    \label{lem:h-hcomp}
    Let $Y_1,Y'_1, W_1,Z_1$ be $1$-morphisms in $\DS_m$ and $f_1\colon Y_1 \to Y'_1$ a $2$-morphism. Similarly let $Y_2,Y'_2, W_2,Z_2$ be $1$-morphisms in $\DS_n$ and $f_2\colon Y_2 \to Y'_2$ a $2$-morphism. Set $Y:=Y_1\boxtimes Y_2$, $Y':=Y'_1\boxtimes Y'_2$, $W:=W_1\boxtimes W_2$ and $Z:=Z_1\boxtimes Z_2$. Suppose the naturality square
    \eqref{eqn:naturality2} for $f:=f_1\boxtimes f_2$ commutes up to the homotopy $h_{f}$. Then the naturality square for $\id_{W}\hcomp f \hcomp \id_{Z} \colon W \hcomp Y \hcomp Z
    \to W\hcomp Y' \hcomp Z$ is commutative up to the homotopy
    \[
        h_{\id_{W}\hcomp f \hcomp \id_{Z}}:=   
        (\id_{(W_2 \boxtimes W_1) \hcomp (Y_2\boxtimes Y_1)} \hcomp \slide_{Z})
        \vcomp 
        (\id_{W_2\boxtimes W_1}\hcomp h_f \hcomp \id_{Z})
        \vcomp
        (\slide_{W} \hcomp \id_{Y\hcomp Z})
    \]
\end{lemma}
\begin{proof} The proof is straightforward when expanding $d(h_{\id_{W}\hcomp f \hcomp \id_{Z}})$.
\end{proof}

\begin{lemma}
    \label{lem:h-coxeter}
    Let $f_2\colon Y_2 \to Y'_2$ be a $2$-morphism in $\DS_n$. Suppose the naturality square
    \eqref{eqn:naturality2} for the morphism $\id_{\one_1}
    \boxtimes f_2$ in $\DS_{1+n}$ with the Coxeter braid complex
    $\cabledcross_{1,n}$ commutes up to the homotopy 
    $h_{\id_{\one_1} \boxtimes f_2}$. Then the naturality square for $\id_{\one_m} \boxtimes
    f_2$ in $\DS_{m+n}$ with respect to $\cabledcross_{m,n}$ commutes up to the homotopy
    \begin{gather*}
        h_{\id_{\one_m} \boxtimes f_2}:= 
        \big((h_{\id_{\one_1} \boxtimes f_2}\boxtimes \id_{\one_{m-1}})
        \hcomp \cdots \hcomp 
        \id_{\one_{m-1-i} \boxtimes \cabledcross_{1,n}\boxtimes\one_{i}}
        \hcomp \cdots \hcomp
        \id_{\one_{m-1}\boxtimes \cabledcross_{1,n}}\big)
        + \cdots
        \\
        + \big(\id_{\cabledcross_{1,n}\boxtimes \one_{m-1}}
        \hcomp \cdots \hcomp 
        (\id_{\one_{m-1-i}} \boxtimes h_{\id_{\one_1} \boxtimes f_2}\boxtimes\id_{\one_{i}})
        \hcomp \cdots \hcomp
        \id_{\one_{m-1}\boxtimes \cabledcross_{1,n}}\big)
        + \cdots
        \\
        + \big(\id_{\cabledcross_{1,n}\boxtimes \one_{m-1}}
        \hcomp \cdots \hcomp 
        \id_{\one_{m-1-i} \boxtimes \cabledcross_{1,n}\boxtimes\one_{i}}
        \hcomp \cdots \hcomp
        (\id_{\one_{m-1}}\boxtimes h_{\id_{\one_1} \boxtimes f_2}) \big).
    \end{gather*}
    An analogous description is available for the homotopies $h_{f_1 \boxtimes \id_{\one_n}}$.
\end{lemma}
\begin{proof}
straightforward by expanding $d(h_{\id_{\one_m} \boxtimes f_2})$ and using the
description \eqref{eqn:slidecabledfromcoxeter} of the slide map for the cabled crossing in terms of slide maps for
Coxeter braids.
\end{proof}

\begin{lemma}
    \label{lem:h-smaller-coxeter}
    Let $f'_2\colon Z_1\to Z_2$ be a $2$-morphism in $\DS_k$ with $k\leq n$. Pick $0\leq l \leq n-k$ and consider the $2$-morphism  $f_2 :=
    \id_{\one_l} \boxtimes f'_2 \boxtimes \id_{\one_{n-l-k}}$ in $\DS_n$ and
    $f=\id_{\one_1} \boxtimes f_2$ in $\DS_{1+n}$. Suppose that
    $h_{\id_{\one_1} \boxtimes f'_2}$ is a homotopy for the naturality square
    \eqref{eqn:naturality2} for the interaction of $\id_{\one_1}
    \boxtimes f'_2$ in $\DS_{1+k}$ with the Coxeter braid complex
    $\cabledcross_{1,k}$. 
    
    Then the naturality square for $\id_{\one_1} \boxtimes
    f_2$ and $\cabledcross_{1,n}$ is commutative up to the homotopy $h_{\id_{\one_1} \boxtimes f_2}$ 
    constructed as the composite
 \begin{gather*}
    \cabledcross_{1,l+k+(n-l-k)} 
        \hcomp (\one_{l+1} \boxtimes Z_1 \boxtimes \one_{n-l-k})\\ 
        % = (\one_l \boxtimes \one_k \boxtimes \cabledcross_{1,(n-l-k)}) 
        % \hcomp (\one_l \boxtimes \cabledcross_{1,k} \boxtimes \one_{n-l-k})
        % \hcomp (\cabledcross_{1,l}\boxtimes \one_k \boxtimes \one_{n-l-k})
        % \hcomp (\one_{l+1} \boxtimes Z_1 \boxtimes \one_{n-l-k})\\
        \cong 
        (\one_l \boxtimes \one_k \boxtimes \cabledcross_{1,(n-l-k)}) 
        \hcomp (\one_l \boxtimes (\cabledcross_{1,k}\hcomp Z_1) \boxtimes \one_{n-l-k})
        \hcomp (\cabledcross_{1,l}\boxtimes \one_k \boxtimes \one_{n-l-k})\\
   \quad \quad\quad\quad\quad\quad\quad\quad   \Big\downarrow{{}_{ \id \hcomp (\id_{\one_l} \boxtimes h_{\id_{\one_1}
        \boxtimes f'_2} \boxtimes \id_{\one_{n-k-l}}) \hcomp \id }}\\
        (\one_l \boxtimes \one_k \boxtimes \cabledcross_{1,(n-l-k)}) 
        \hcomp (\one_l \boxtimes (Z_2 \hcomp \cabledcross_{1,k}) \boxtimes \one_{n-l-k})
        \hcomp (\cabledcross_{1,l}\boxtimes \one_k \boxtimes \one_{n-l-k})\\
        \cong
        % (\one_l \boxtimes Z_2 \boxtimes \one_{n-l-k+1})
        % \hcomp (\one_l \boxtimes \one_k \boxtimes \cabledcross_{1,(n-l-k)}) 
        % \hcomp (\one_l \boxtimes \cabledcross_{1,k} \boxtimes \one_{n-l-k})
        % \hcomp (\cabledcross_{1,l}\boxtimes \one_k \boxtimes \one_{n-l-k})\\
        % = 
        (\one_l \boxtimes Z_2 \boxtimes \one_{n-l-k+1})
        \hcomp \cabledcross_{1,l+k+(n-l-k)},
  \end{gather*}
    where the isomorphisms are given far-commutativity. An analogous
    description is available for $2$-morphisms arising from inclusions
    $\DS_k\hookrightarrow \DS_m$.
\end{lemma}
\begin{proof}
    The proof is straightforward when expanding $d(h_{\id_{\one_1} \boxtimes f_2})$.
\end{proof}

\begin{proof}[Proof of \Cref{thm:naturality}] For $2$-morphisms $f_1 \colon Y_1
    \to Y'_1$ in $\DS_m$ and $f_2 \colon Y_2 \to Y'_2$ we have to construct a
    homotopy $h_{f_1\boxtimes f_2}$ such that \[d(h_{f_1\boxtimes f_2}) =
    ((f_2\boxtimes f_1)\hcomp \id) \vcomp \slide_{Y_1,Y_2} -
    \slide_{Y'_1,Y'_2}\vcomp (\id \hcomp (f_1\boxtimes f_2)).\] By linearity, \Cref{lem:h-coxeter}, 
    \Cref{lem:h-smaller-coxeter} and 
    \Cref{lem:h-vcomp}, we may assume that one of the factors $f_1$ or $f_2$ is
    an identity $2$-morphism and the other one is a horizontal composition of
    identity $2$-morphisms and a single diagrammatic generating $2$-morphism.
    Then \Cref{lem:h-hcomp} lets us reduce to the case, when the
    interesting factor (say wlog. $f_2$ in $\DS_n$) is purely a diagrammatic generator. Next, \Cref{lem:h-coxeter} reduces the problem of
    constructing the desired $h_f$ to the case $m=1$. By
    \Cref{lem:h-smaller-coxeter}, we may moreover assume that $n=k$, where $k$ is the
    number of strands on which the diagrammatic generator of type $f$ is
    supported. The homotopies are now provided by \Cref{prop:slidehomotopyforgen} below.
\end{proof}

\subsection{Slide homotopies: explicit constructions}
\label{sec:slidegens}
We finally construct explicit homotopies for the naturality square \eqref{eqn:naturality2} for $f_1=\id_1$, $f_2=f$ where $f$ is a diagrammatic generator. We consider, one by one, the different diagrammatic generators from \eqref{diaggenerators}: the monomial generator, the start dot, the end dot, the merge trivalent vertex, the split trivalent vertex, two kinds of six-valent vertices, and the four-valent vertex.
\begin{proposition}[Slide homotopies for diagrammatic generators]
    \label{prop:slidehomotopyforgen}
Let $f$ be a type of diagrammatic generator of $\DS$ and $k=k(f)\in\N$ be the number of
strands on which the generator is supported. Then the naturality square \eqref{eqn:naturality2} for $(m,n)=(1,k)$ and $f_1=\id$ and $f_2=f$ commutes up to the explicit homotopy $h=h_{\id\boxtimes f}$ given in the proof. 
\end{proposition}
\begin{proof} We now compute the desired homotopies. We leave the verification that these are indeed homotopies with the correct differential to the reader, but refer to \Cref{rk:computesixvalent}.
 For convenience, the final results are collected in several Figures in  \Cref{sec:diags}. \begin{itemize}    
\item     
{\it The polynomial generator on $k=1$ strand.} Here the slide chain maps are both identities. We write the action of polynomials on the left with variables $x_1,x_2$ and on the right with variables $x_1', x_2'$. Then the failure of the naturality square to commute on the nose is the chain map that  acts component-wise by $x_1-x_2'$. This is nullhomotopic with homotopy $h$ given by the start dot.

       \item  {\it The start dot, for $k=2$ strands.} We let $f$ be the start dot and determine the homotopy $h=h_{\id\boxtimes f}$ satisfying 
\[((f\boxtimes \id)\hcomp \id) \vcomp \slide_{Y_1,Y_2} -
        \slide_{Y'_1,Y'_2}\vcomp (\id \hcomp (\id\boxtimes f))  
        =d(h_{\id\boxtimes f}).\] The homotopy is given  in \Cref{fig:startdot}.
      \item   {\it The end dot, for $k=2$ strands.}
        The homotopy is computed in \Cref{fig:enddot}.

       \item  {\it The merge vertex, for $k=2$ strands.}
        The homotopy is zero, see \Cref{fig:merge}.

      \item   {\it The split vertex, for $k=2$ strands.}
        The homotopy is zero, see \Cref{fig:split}.
       
        \item {\it The first 6-valent vertex, for $k=3$ strands.} 
        The homotopy is computed in \Cref{fig:first6valentvertex1} and \Cref{fig:first6valentvertex2}.
        \item
        {\it The second 6-valent vertex, for $k=3$ strands.}
        The homotopy is zero, see \Cref{fig:second6valentvertex1} and \Cref{fig:second6valentvertex2}.
       \item  {\it The four-valent vertex, for $k\geq 4$ strands.} Here we will only consider $k=4$ since
        the cases $k>4$ involve larger complexes, but are essentially analogous.
        Also, since the four-valent vertices are invertible, we only consider one
        version. The naturality for the inverse is a formal consequence. The homotopy is zero and computed in \Cref{fig:4valentvertex1} and \Cref{fig:4valentvertex2}.
\end{itemize}
We computed the homotopy for each possible type of generator.
\end{proof}

\begin{obs}
It is remarkable that many naturality relations (all but those for the polynomial generator, for the start
and end dots, and one version of the six-valent vertex) hold on the nose, not
just up to homotopy. 
\end{obs}

\begin{obs} Some diagrammatic relations are inherited on the nose by the
 $h_f$. For example, the barbell relation \eqref{eq:barbell} in $\DS$  has such a corresponding relation. Namely,  
\[ \begin{tikzpicture}[anchorbase,smallnodes, scale=.2]
        \draw[rd] (1,-0.01) \pu (1,3.01);
        \draw[rd] (0,1.01) \pu (0,2.01);
        \fill[rd] (0,1) circle (2.5mm); 
        \fill[rd] (0,2) circle (2.5mm); 
    \end{tikzpicture} 
    +
    \;\begin{tikzpicture}[anchorbase,smallnodes,scale=.2]
        \draw[rd] (-1,-0.01) \pu (-1,3.01);
        \draw[rd] (0,1.01) \pu (0,2.01);
        \fill[rd] (0,1) circle (2.5mm); 
        \fill[rd] (0,2) circle (2.5mm); 
    \end{tikzpicture}  
=2 \;
\begin{tikzpicture}[anchorbase,smallnodes,scale=.2]
        \draw[rd] (0,-0.01) \pu (0,1.01);
        \fill[rd] (0,1) circle (2.5mm); 
        \fill[rd] (0,2) circle (2.5mm); 
        \draw[rd] (0,2) \pu (0,3.01);
    \end{tikzpicture}
\quad\quad\text{induces the relation}\quad\quad
 h_{\begin{tikzpicture}[anchorbase,smallnodes, scale=.2]
        \draw[rd] (1,-0.01) \pu (1,3.01);
        \draw[rd] (0,1.01) \pu (0,2.01);
        \fill[rd] (0,1) circle (2.5mm); 
        \fill[rd] (0,2) circle (2.5mm); 
    \end{tikzpicture} } 
    +
    h_{\;\begin{tikzpicture}[anchorbase,smallnodes,scale=.2]
        \draw[rd] (-1,-0.01) \pu (-1,3.01);
        \draw[rd] (0,1.01) \pu (0,2.01);
        \fill[rd] (0,1) circle (2.5mm); 
        \fill[rd] (0,2) circle (2.5mm); 
    \end{tikzpicture} } 
=2 h_{ \;
\begin{tikzpicture}[anchorbase,smallnodes,scale=.2]
        \draw[rd] (0,-0.01) \pu (0,1.01);
        \fill[rd] (0,1) circle (2.5mm); 
        \fill[rd] (0,2) circle (2.5mm); 
        \draw[rd] (0,2) \pu (0,3.01);
    \end{tikzpicture}} .
    \]
In general, the relations do not carry over to the  homotopies, see  \Cref{exa:hiho}.
\end{obs}

\begin{remark}\label{rk:computesixvalent}
    The verification of the fact that the homotopy for the first six-valent vertex has the correct
    differential uses two somewhat nontrivial relations in the diagrammatic calculus. The first one 
  \begin{equation}
        \vcenter{\hbox{\includegraphics{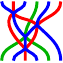}}}   
        -
        \vcenter{\hbox{\includegraphics{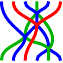}}}  
        =
        \vcenter{\hbox{\includegraphics{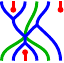}}}   
        -
        \vcenter{\hbox{\includegraphics{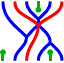}}}  
    \end{equation}
    is a consequence of the parabolic relation \eqref{eq:para} of type $A_3$ (called 
    Zamolodchikov relation in  \cite[\S1.4.3.]{EW}) that can be obtained by composing with sixvalent vertices.
    
         The second
        relation is the rotationally $\Z/3\Z$-symmetric relation that can be proved by sliding the trivalent vertex through one of the sixvalent vertices
\begin{equation}
    \secondrel.
\end{equation}
\end{remark}

\subsection{Higher homotopies}
\label{ssec:hhomotopies}

We need the following well-known auxiliary results.
\begin{lemma}
    \label{lem:pre-post-comp}
    Let $\CS$ be a dg-category and $f\colon Y_1 \to Y_2$ a homotopy equivalence
    (i.e.~a closed degree zero morphism inducing an isomorphism in the homotopy
    category). For objects $X$, $Z$ in $\CS$ we get homotopy equivalences of
    morphism complexes:
    \[
        f^*\colon \Hom_\CS(X,Y_1)
        \xrightarrow{g\mapsto fg}
        \Hom_\CS(X,Y_2), \qquad 
       f_*\colon \Hom_\CS(Y_2,Z)
        \xrightarrow{g'\mapsto g'f}
        \Hom_\CS(Y_1,Z)
    \]
\end{lemma}
\begin{proof} Pre- and post-composition with $f$ gives (degree zero) chain maps
 since $f$ is closed.
 Let $f'\colon Y_2 \to Y_1$ be a homotopy inverse to $f$, witnessed by homotopies $h_i\in \End(Y_i)$, $i=1,2$, i.e. $d_1 h_1+ h_1 d_1 = f'f-\id_{Y_1},\quad d_2 h_2+ h_2 d_2 = ff'-\id_{Y_2}$. 
 We check the second statement in the case of postcomposition and claim that the requisite homotopies are given by postcomposing with $h_1$ and $h_2$. For  $g\in \Hom_\CS(X,Y_1)$ we get
 \begin{align*}
(f'f-\id_{Y_1})^*(g)& = (f'f-\id_{Y_1})g  \\
&=(d_1 h_1+ h_1 d_1)g 
= d_1 h_1 g + (-1)^{|g|} h_1 g d_X + (h_1 d_1 g - (-1)^{|g|} h_1 g d_X)
\\
&= d_{\Hom}(h_1 g) + h_1 d_{\Hom}(g)
 =(d_{\Hom}\circ h_1^*)(g) + (h_1^* \circ d_{\Hom})(g).
 \end{align*}
 Analogously one checks $(ff'-\id_{Y_2})^* = d_{\Hom}\circ h_2^* + h_2^* \circ d_{\Hom}$.
 \end{proof}

%The proof is standard, but it is useful to record that for a (not necessarily closed) morphism $k\colon Y_1 \to Y_2$ we have
%\[
%d(k_*) = d(k)_*\quad , \quad d(k^*) = d(k)^*
%\]
%since composition is a chain map (also with one argument fixed).

% \begin{align*}
% d(k_*)(g) 
% &=  d_{XY_2}(k_*(g)) - (-1)^{|k|} k_*(d_{XY_1}(g)) \\
% &=  d_{XY_2}(kg) - (-1)^{|k|} k(d_{XY_1}(g)) \\
% &=  \left(d_{Y_2}kg - (-1)^{|kg|} kg d_{X}\right) - (-1)^{|k|} \left( kd_{Y_1}g - (-1)^{|G|} kg d_{X} \right)\\
% &=  d_{Y_2}kg  - (-1)^{|k|} kd_{Y_1}g \\
% &=  d_{XY_2}(k)_*(g) 
% \end{align*}

\begin{lemma}
    \label{lem:invertible-tensoring}
Let $(\CS,\hcomp)$ be a monoidal dg-category, $Y_1$, $Y_2$ objects and $X$ an
    object that is invertible up to homotopy (see \Cref{def:hinvertibility} with examples in \Cref{cor:Rouqinv}).
    Then we have mutually inverse homotopy equivalences of morphism complexes:
    \[
        \Hom_\CS(Y_1,Y_2)
        \xrightleftharpoons[\psi]{\phi}
        \Hom_\CS(X \hcomp Y_1,X\hcomp Y_2)
    \]
    
defined by $\phi(f) = \id_X \hcomp f$ and $\psi(g)=(
\ev_{X'} \hcomp \id_{Y_2})(\id_{X'}\hcomp g)(\coev_{X} \hcomp \id_{Y_1})$. Similarly, we have
homotopy equivalences  $\Hom_\CS(Y_1,Y_2) \hte \Hom_\CS(Y_1\hcomp X,Y_2\hcomp X)$.
\end{lemma}
For the proof we use that for the horizontal composition (monoidal product) with any fixed morphism $k\colon X_1 \to X_2$ we have
\[ d(k \hcomp -) =  d(k) \hcomp - \colon   \Hom_\CS(Y_1,Y_2) \to \Hom_\CS(X_1\hcomp Y_1,X_1\hcomp Y_2).\]
In particular, if $k$ is closed, then horizontal composition with $k$ defines a chain map
\[k \hcomp - \colon \Hom_\CS(Y_1,Y_2) \to \Hom_\CS(X_1\hcomp Y_1,X_1\hcomp Y_2), \quad f \mapsto k \hcomp f.\]
\begin{proof}
The assignments $\phi$ and $\psi$ define chain maps since they are given by
tensoring and composing with closed degree zero morphisms. We compute for $f\in
\Hom_\CS(Y_1,Y_2)$:
\begin{align*}
(\psi \circ \phi - \id)(f)
& =
(\ev_{X'} \hcomp \id_{Y_2})(\id_{X'\hcomp X}\hcomp f)(\coev_{X}\hcomp \id_{Y_1} ) - f
\\
& = (\ev_{X'}\vcomp \coev_{X}-\id ) \hcomp f 
 = d(h_R) \hcomp f
 = d(h_R \hcomp -)(f).
\end{align*}
For the opposite composite $\phi \circ \psi$ we compute as follows:\footnotesize
\begin{align*}
    (\phi \circ \psi %- \id
    )(g)
     &=
    \id_X \hcomp(\ev_{X'} \hcomp \id_{Y_2})(\id_{X'}\hcomp g)(\coev_{X}\hcomp \id_{Y_1} )% - g
    \\
     &=
    (\id_X \hcomp \ev_{X'} \hcomp \id_{Y_2})(\id_{X\hcomp X'} \hcomp \id_{X} \hcomp \id_{Y_2}) (\id_{X\hcomp X'}\hcomp g)(\id_X \hcomp \coev_{X}\hcomp \id_{Y_1} )% - g
    \\
    &=
   (\id_X \hcomp \ev_{X'} \hcomp \id_{Y_2})(\coev_{X'}\vcomp \ev_{X})  \hcomp \id_{X} \hcomp \id_{Y_2}) (\id_{X\hcomp X'}\hcomp g)(\id_X \hcomp \coev_{X}\hcomp \id_{Y_1} )% - g
  \\ &\quad \BLUE{-
   (\id_X \hcomp \ev_{X'} \hcomp \id_{Y_2})(d(h_1))  \hcomp \id_{X} \hcomp \id_{Y_2}) (\id_{X\hcomp X'}\hcomp g)(\id_X \hcomp \coev_{X}\hcomp \id_{Y_1} )}% - g
    % \\
    % & = (\ev_{X'}\vcomp \coev_{X}-\id ) \hcomp f 
    % \\
    % & = d(h_R) \hcomp f
    % \\
    % & = d(h_R \hcomp -)(f)
    \end{align*}
    {\normalsize

    For the last line, highlighted in blue, we have used the homotopy $h_1$ relating $\coev_{X'}\vcomp \ev_{X}\simeq \id_{X\hcomp X'}$. This blue term will be part of the differential of the homotopy and put aside for the moment. We continue rewriting the remaining component:
    }
    \footnotesize
    \begin{align*}
       & \hspace{-.5cm}(\id_X \hcomp \ev_{X'} \hcomp \id_{Y_2})((\coev_{X'}\vcomp \ev_{X})  \hcomp \id_{X} \hcomp \id_{Y_2}) (\id_{X\hcomp X'}\hcomp g)(\id_X \hcomp \coev_{X}\hcomp \id_{Y_1} )% - g
      \\ 
      &=
      (((\id_X \hcomp \ev_{X'})(\coev_{X'} \hcomp \id_{X})) \hcomp \id_{Y_2}) g
      (((\ev_{X} \hcomp \id_{X})(\id_X \hcomp \coev_{X}))\hcomp \id_{Y_1} )% - g   
      \\
      &=
      (\id_X \hcomp \id_{Y_2}) g (\id_X \hcomp \id_{Y_2})\\
     & \quad \BLUE{-
      ( d(h_2) \hcomp \id_{Y_2}) g
      (((\ev_{X} \hcomp \id_{X})(\id_X \hcomp \coev_{X}))\hcomp \id_{Y_1} )}
    \BLUE{-
      (\id_X \hcomp \id_{Y_2}) g
      (d(h_3)\hcomp \id_{Y_1} ).
      }% - g
        % \\
        % & = (\ev_{X'}\vcomp \coev_{X}-\id ) \hcomp f 
        % \\
        % & = d(h_R) \hcomp f
        % \\
        % & = d(h_R \hcomp -)(f)
        \end{align*}
        \normalsize
    Here we used the zigzag relations up to  homotopies $h_1$, $h_2$ for the evaluation and
    coevaluation of $X$ and $X'$ respectively. Thus we have $(\phi \circ \psi -
    \id )(g) = d(H(g))$ for a homotopy $H(g)$ constructed from the homotopies
    for the invertible object $X$ using vertical and horizontal composition with
    fixed degree zero morphisms and $g$.
\end{proof}

\begin{corollary}
    \label{cor:end-invertible}
    Let $(\CS,\hcomp)$ be a monoidal dg-category with unit $\one$. Then we have
    quasi-isomorphism of dg-algebras respectively chain complexes
    \[\End_\CS(\one)\qis \End_\CS(X),\quad \Hom_\CS(X,Y) \qis \End_\CS(\one)\]
for any objects $X$ which is invertible up to homotopy and $Y$ homotopy equivalent to $X$.

\end{corollary}
\begin{proof}
    \Cref{lem:invertible-tensoring} gives homotopy equivalences
    \[
        \Hom_\CS(\one,\one)
        \hte
        \Hom_\CS(X \hcomp \one,X \hcomp \one)
        \hte
    \Hom_\CS(X,X)
    \]
    and the left-to-right map $\phi(f)=f\hcomp \id_X$ is clearly an algebra
    homomorphism, yielding an algebra isomorphism on homology. \Cref{lem:pre-post-comp} implies the second statement.
\end{proof}

\begin{proposition}
    \label{prop:cohomologymorphismcomplex}
    Let $Y_1,Y'_1$ be $1$-morphisms in $\DS_m$ and $Y_2,Y'_2$ $1$-morphisms in $\DS_n$. Then we have a
    quasi-isomorphism of morphism complexes:
  \begin{equation}
    \label{eq:homcomplex}
      \Hom_{\Ch(\DS_{m+n})}(X_{m,n}\hcomp (Y_1\boxtimes Y_2),(Y'_2\boxtimes Y'_1)\hcomp X_{m,n})
      \qis 
      \Hom_{\Ch(\DS_{m+n})}(Y_1\boxtimes Y_2,Y'_1\boxtimes Y'_2).
  \end{equation}
      In particular, these morphism complexes have cohomology concentrated in
      degree zero. 
\end{proposition}
\begin{proof}
    By \Cref{lem:pre-post-comp} the slide chain homotopy
    equivalences from \Cref{prop:sliding-objects} provide a
    homotopy equivalence:
    \begin{gather*}
        \Hom_{\Ch(\DS_{m+n})}(X_{m,n}\hcomp (Y_1\boxtimes Y_2),(Y'_2\boxtimes Y'_1)\hcomp X_{m,n})\\
      \qquad \qquad\qquad \qquad\hte
      \Hom_{\Ch(\DS_{m+n})}(X_{m,n}\hcomp (Y_1\boxtimes Y_2),X_{m,n}\hcomp (Y'_1\boxtimes Y'_2)). \end{gather*}
      Now we apply
      \Cref{lem:invertible-tensoring} to ``cancel'' the invertible
      factor $X_{m,n}$.
\end{proof}

\begin{lemma}[Compatibility of slides and tensorators]
    \label{lem:slidetensorator}  
Let $m,n\in \N_0$ and consider a tensorator $2$-morphism $\boxtimes_{c_{\underline{i}},c_{\underline{j}}}\colon
      (c_{\underline{i}} \boxtimes \id) \hcomp (\id \boxtimes c_{\underline{j}})\to(\id \boxtimes c_{\underline{j}})\hcomp (c_{\underline{i}} \boxtimes \id)
    $ in $\DS_{m,n}$ as in \Cref{pseudo7}. We also consider the tensorator $\boxtimes_{c_{\underline{j}},c_{\underline{i}}}$ in $\DS_{n,m}$. Then we have
    \begin{equation}
        \label{eqn:naturalitytensorator}
        \begin{tikzcd}[scale=1,column sep=2.5cm]        
            \cabledcross_{m,n} \hcomp (c_{\underline{i}} \boxtimes c_{\underline{j}})
            \ar[r, "\slide_{c_{\underline{i}},c_{\underline{j}}}"]
            \ar[d, swap,"\id \hcomp (\boxtimes_{c_{\underline{i}},c_{\underline{j}}})"]
        &
        (c_{\underline{i}}\boxtimes c_{\underline{j}})\hcomp \cabledcross_{m,n} 
            \ar[d,"(\boxtimes_{c_{\underline{j}},c_{\underline{i}}})\hcomp \id"]
            \ar[dl,swap,Rightarrow,shorten >=5ex,shorten <=5ex, "h_{\boxtimes_{c_{\underline{i}},c_{\underline{j}}}}"]\\
              \cabledcross_{m,n} \hcomp c_{\underline{i}} \boxtimes c_{\underline{j}}
                \ar[r, swap,"\slide_{c_{\underline{j}},c_{\underline{i}}}"]
            &
            (c_{\underline{i}} \boxtimes c_{\underline{j}}) \hcomp \cabledcross_{m,n} 
      \end{tikzcd}
    \end{equation}
    commutes up to homotopy; a homotopy $h_{\boxtimes_{c_{\underline{i}},c_{\underline{j}}}}$ will be
    constructed in the proof.
\end{lemma}

\begin{proof}
     By the previous reduction steps, it suffices to prove this in the case when
     $c_{\underline{i}}=c_i$ and $c_{\underline{j}}=c_j$ are words of length
     one, which implies that $\boxtimes_{c_{\underline{i}},c_{\underline{j}}}$
     is a single four-valent vertex. In this case, it follows from
     \Cref{prop:cohomologymorphismcomplex} that the hom complex in
     \eqref{eqn:naturalitytensorator} has cohomology concentrated in homological
     degree zero which is, moreover, $1$-dimensional in bidegree zero. Thus, the
     two chain maps being compared in \eqref{eqn:naturalitytensorator} have to
     be homotopic to scalar multiples of each other. In fact, the scalars are
     one, as can be seen by postcomposing with suitable end-dot morphisms and
     using their naturality squares from \cref{thm:naturality}.  
    \end{proof}

\begin{theorem}
    \label{thm:naturalityfull}  
    Let $m,n\in \N_0$ and $f\colon Y \to Y'$ be a morphism in $\DS_{m,n}$ from \Cref{def:parabolicsub} and denote by $\swap_{m,n}(f)\colon \swap_{m,n}(Y) \to \swap_{m,n}(Y')$ the swapped morphism in $\DS_{n,m}$, see \Cref{def:swap}.
    Then the diagramm
    \begin{equation}
        \label{eqn:naturality}
        \begin{tikzcd}[scale=1,column sep=2.5cm]        
            \cabledcross_{m,n} \hcomp (Y)
            \ar[r, "\slide_{Y}"]
            \ar[d, swap,"\id \hcomp f"]
        &
        \swap_{m,n}(Y)\hcomp \cabledcross_{m,n} 
            \ar[d,"\swap_{m,n}(f)\hcomp \id"]
            \ar[dl,swap,Rightarrow,shorten >=5ex,shorten <=5ex, "h_{f}"]\\
              \cabledcross_{m,n} \hcomp Y'
                \ar[r, swap,"\slide_{Y'}"]
            &
            \swap_{m,n}(Y') \hcomp \cabledcross_{m,n} 
      \end{tikzcd}
    \end{equation}
    commutes up to homotopy; a homotopy $h_f$ will be
    constructed in the proof.
\end{theorem}
\begin{proof}
This is a combination of \Cref{thm:naturality} and \Cref{lem:slidetensorator}.
\end{proof}

\begin{corollary}
    \label{cor:well-defined-up-to-coherent-homotopy}
    The homotopies constructed in \Cref{thm:naturality} and \Cref{thm:naturalityfull}  are well-defined and unique up to coherent higher homotopy.
\end{corollary}
\begin{proof}
    Assume we are given $2$-morphisms $f\colon Y \to Y'$ in $\DS_{m,n}$. Then $h:=h_{f}$ is an
    element of degree one of the morphism complex \eqref{eq:homcomplex} such
    that $    d(h) =  (\swap_{m,n}(f) \hcomp \id) \vcomp \slide_{Y} -
    \slide_{Y'}\vcomp (\id \hcomp f)$. If $h'$ is another such homotopy, then
    the difference $h-h'$ is closed and thus exact, since the 
    first cohomology of \eqref{eq:homcomplex} vanishes by \Cref{prop:cohomologymorphismcomplex}.
    Consequently, there exists a
    higher homotopy $h^{(2)}$ with $d(h^{(2)})=h-h'$. Now we can
    iterate: $h^{(2)}$ may not be unique, but any different choice would have
    to be homotopic up to an even higher homotopy $h^{(3)}$, etc.

\end{proof}

\section{Naturality for the braiding}
\label{sec:extensiontocomplexes}
So far we have only provided naturality data for sliding (diagrammatic) Bott--Samelson bimodules (and
morphisms between them) through the proposed braiding given by Rouquier
complexes of cabled crossings. To extend this to complexes,
we are forced to study higher homotopies in a conceptual way.

\subsection{The 2-sided bar construction for monoidal dg-categories}
We briefly review the bar construction for (monoidal) dg-categories with the focus in explaining how it
can be used to formalize what it means for a transformation between dg-functors
to be \emph{natural up to homotopy}.  For background on bar constructions, see e.g. \cite{Positselski}, we will use the conventions from \cite[\S5.1]{GHW}. The case of interest for us is the prospective
braiding on chain complexes over the diagrammatic categories, which we can view as a family of functors related by transformations given in terms of slide maps from \Cref{thm:naturality} which we like to show are  natural up to homotopy.

Let $\CS$ be a dg-category over a field $\k$. We will consider $\CS$ as a
locally-unital dg-algebra with orthogonal idempotents $1_S$ indexed by objects
$S\in \Ob(\CS)$:
\[
\CS = \bigoplus_{S,T\in \Ob(\CS)} 1_T \CS 1_S,\qquad 1_T \CS 1_S:=\Hom_{\CS}(S,T)  
\]
with multiplication given by the composition of morphism.  Any dg-subcategory $\IS\subset \CS$ corresponds to a dg-subalgebra. 
\begin{example}
    \label{exa:subcat}
If $\CS$ is a dg-category, then we write $\IS_\CS$ for the dg-subcategory with
the same objects as $\CS$, but whose only non-zero morphisms are the scalar
multiples of identity morphisms. Then $\IS:=\IS_\CS$ is a dg-subalgebra of $\CS$. 
\end{example}

For simplicity we often write $T \CS S:=1_T \CS 1_S$.  
We can consider $\CS$ as a dg-bimodule over itself, which we denote by ${}_{\CS}\CS_{\CS}$. By shifting all
hom complexes we obtain dg-bimodules ${}_{\CS}\Sigma^k(\CS)_{\CS}$ for any $k\in
\Z$, where  the bimodule structure is given by $f_L \cdot f \cdot f_R:=
(-1)^{k|f_L|} f_L\circ f \circ f_R$.

Given $\IS\subset \CS$ as in \Cref{exa:subcat}, we can consider tensor products of the form ${}_\CS\CS\otimes_{\IS}\CS
\otimes_{\IS} \cdots \otimes_{\IS} \CS \otimes_{\IS} \CS_\CS$ and denote
elements of such a tensor product by $f_0\bbar f_1\bbar \cdots \bbar f_{r}\bbar
f_{r+1}$.  Note these tensor products are dg-bimodules for $\CS$, and in particular complexes. 

\begin{definition}
    Let $\CS$ be a dg-category and $\IS\subset \CS$ a dg-subcategory.  The 
    \emph{2-sided bar complex} of the pair $(\CS,\IS)$ is the dg-bimodule for $\CS$ given by the total complex
    $\barcx(\CS,\IS)$ of the following bicomplex:
\[ 
\cdots 
\xrightarrow{} 
{}_\CS\CS\otimes_{\IS}\CS \otimes_{\IS} \CS \otimes_{\IS} \CS_\CS    
\xrightarrow{} 
{}_\CS\CS\otimes_{\IS}\CS \otimes_{\IS} \CS_\CS    
\xrightarrow{} 
{}_\CS\CS \otimes_{\IS} \CS_\CS
\]
where the horizontal differential is given by:
\[
    f_0\bbar f_1\bbar \cdots \bbar f_{r}\bbar f_{r+1} 
    \xmapsto{d_{\mathrm{bar}}} 
    \sum_{i=0}^{r} (-1)^i f_0\bbar \cdots\bbar f_i \circ f_{i+1}\bbar\cdots \bbar f_{r+1}
    \]
Formally, we have $\barcx(\CS,\IS):=\bigoplus_{r\geq 0}
\Sigma^r(\underbrace{\CS\otimes_\IS \cdots \otimes_\IS \CS}_{r+2 \text{
factors}})$ with the differential given by the sum of $d_{\mathrm{bar}}$ and the
internal differentials on the summands $\Sigma^r(\CS\otimes_\IS \cdots
\otimes_\IS \CS)$, see e.g. \cite[\S5.1]{GHW} for explicit formulas. 
\end{definition}

We will be interested in the special case $\IS:=\IS_\CS$ and abbreviate $\barcx(\CS):=\barcx(\CS,\IS_\CS)$.
The 2-sided bar complex $\barcx(\CS)$ is then,  considered as a dg-bimodule for $\CS$, 
spanned by configurations $f_0\bbar f_1\bbar \cdots \bbar f_{r}\bbar f_{r+1}$
where $f_i\colon Y_{i+1}\to Y_{i}$ is a sequence of composable morphisms in
$\CS$. If moreover $\DS$ is a full subcategory of $\CS$, then $\barcx(\CS,\IS_\DS)$ is the
subcomplex of $\barcx(\CS)$ spanned by configurations $f_0\bbar f_1\bbar \cdots
\bbar f_{r}\bbar f_{r+1}$ of composable morphisms $f_i\colon Y_{i+1}\to
Y_{i}$, where $Y_1,\dots, Y_{r+1}$ are contained in $\DS$, although $Y_0,Y_{r+2}$ need not to be.

In the following let $\AS$ be a $\k$-linear category. As usual we consider $\AS$
    as a dg-category with morphism complexes concentrated in homological degree
    zero, embedded in the dg-category $\CS:=\Ch^b(\AS)$ as complexes
    concentrated in homological degree zero.    

\begin{proposition}\label{prop:barcxdefr}
     Consider $\AS\subset \CS:=\Ch^b(\AS)$. Then the inclusion of complexes     
        \[\barcx(\CS,\IS_\AS) \hookrightarrow \barcx(\CS,\IS_\CS)=\barcx(\CS)
    \] 
    is the section of a deformation retract of dg-bimodules over $\CS$.
\end{proposition}
\begin{proof}
This is \cite[\S5.3]{GHW}, since $\CS$ is the pretriangulated hull of $\AS$. 
\end{proof}

\subsection{The cabled crossing as a (twisted) centralizing object}
Here we refine the data described in \Cref{ssec:slides-homotopies} and organise
them into a more conceptual framework  as indicated already in \Cref{rkprebraiding}.
We  show that even complexes of diagrammatic Bott--Samelson bimodules slide through the Rouquier complex associated to the cabled crossing.  

For a monoidal dg-category $\CS$ the 2-sided bar construction $\barcx(\CS)$ is
not just a dg-bimodule for $\CS$, but additionally carries the structure of a dg-algebra structure with multiplication given by the shuffle product, \cite{EZ, LQ}, see e.g. \cite[Section 6.1]{GHW}
for explicit formulas.

Given an object $Z$ of $\CS$, consider the $\CS$-dg-bimodule
\begin{equation}\label{eq:Z}
    {}_\CS Z_\CS :=\bigoplus_{X,Y\in \CS} (Y\otimes Z)\CS(Z \otimes X).
    \end{equation}
as an algebra object in $\CS$-dg-bimodules, see \cite[Section 6.2]{GHW}, where ${}_\CS Z_\CS$ is denoted
$\mathcal{X}_{12}(Z)$. 
The following dg-versions can be found in \cite{GHW} and generalize the notion of centralizers  from \cite[Def. 2.6.1]{liu2024braided}, see also \Cref{rkprebraiding}.
\begin{itemize}[leftmargin=5mm]
\item The data of a collection of morphisms $\tau_X\colon Z\otimes X \to X \otimes Z$ that are \emph{natural in $X$ up to coherent homotopy} is the data of a $\CS$-bimodule map $\tau \colon \barcx(\CS) \to {}_\CS Z_\CS$.
\item  The resulting natural transformation $Z\otimes - \to - \otimes Z$ is 
     \emph{strictly compatible with the monoidal structure} in $\CS$ if  $\tau$ is additionally a morphism of dg-algebras.

    In this case, the pair $(Z,\tau)$ can be considered as an object of the \emph{dg-Drinfeld center} $\dgZ(\CS)$.
 \item  Given a monoidal subcategory $\MS\subset \CS$,  then a dg-bimodule map
    $\tau \colon \barcx(\MS) \to {}_\MS Z_\MS$ that is also a morphism of dg-algebras  
    captures the data of a collection of morphisms $Z\otimes
    X \to X \otimes Z$ in $\CS$ that are \emph{natural up to coherent homotopy and strictly compatible
    with the tensor product in $\MS$}, i.e. now only for objects $X\in\MS$.  
    
        Such pairs $(Z,\tau)$ are the objects 
    in the  \emph{dg-monoidal centralizer} $\dgZ_{\CS}(\MS)$ of $\MS$ in $\CS$ defined in \cite[Definition 6.3]{GHW}. Note that $\dgZ_{\CS}(\CS)=\dgZ(\CS)$.  
\end{itemize}

We introduce the following twisted algebra object: 
\begin{definition}
Let $\MS$ and $\CS$ be monoidal dg-categories and $\phi,\psi\colon \MS\to
\CS$ a pair of monoidal dg-functors. We denote by $_{\MS}^{\phi} \CS^{\psi}_{\MS}$ 
the $\CS$-dg-bimodule $\CS$ considered as an $\MS$-dg-bimodule
by pulling-back the $\CS$-action along the functors (for the left action along $\phi$, for the right action along $\psi$.)
Directly generalizing \eqref{eq:Z} we define the $\MS$-dg-bimodule
\[
    _{\MS}^{\phi} Z^\psi_{\MS}:=\bigoplus_{X,Y\in \MS} (\phi(Y)\otimes Z)\CS(Z \otimes\psi( X)).
    \]
which is an algebra object in $\MS$-dg-bimodules.
\end{definition}

We finally introduce a \emph{twisted} and \emph{weaker} version of a 
dg-monoidal centralizer.
\begin{definition} 
    \label{def:twistedcentralizer}
An \emph{object of the twisted $A_\infty$-monoidal centralizer}
$\AiZ_\CS(\phi,\psi)$ is a pair $(Z,\tau)$, where $Z$ is an object of $\CS$ and
$\tau$ is a morphism $\tau\colon \barcx(\MS)\to {}_\MS^{\phi}
Z^\psi_\MS$ of $\MS$-dg-bimodules  that extends to an $\Ai$-morphism between dg-algebras.
\end{definition}

\begin{remark}
The concept of twisted centers for ordinary monoidal categories appears already in  \cite{Shimizu}, and for $2$-categories in \cite{Halbigetal} .
\end{remark}
We now connect this back to categories of Soergel bimodules, see also \Cref{rkprebraiding}. 

\begin{nota}\label{notationdgcats}
For $m,n\in\N_0$ we set 
$\CS_{m+n}=\Chb(\DS_{m+n})$ and consider
\begin{itemize}
\item the monoidal dg-category $\CS_{m,n}=\Chb(\DS_{m,n})$ where $\DS_{m,n}=\DS_m\boxtimes \DS_n$ is defined as in \Cref{def:parabolicsub}.
\item the monoidal dg-functor $\psi\colon \CS_{m,n}\to \CS_{m+n}$ determined by extending the standard inclusion
 $\DS_m\boxtimes \DS_n \to \DS_{m+n}$ induced by the $\boxtimes$-product from \eqref{parabinddiag},
\item the monoidal dg-functor $\phi\colon \CS_{m,n}\to \CS_{m+n}$ determined by extending the opposite inclusion
$\DS_m\boxtimes \DS_n \to \DS_{m+n}$ induced by the $\boxtimes^\opp$ from \eqref{eq:parabolicindop}.
\end{itemize}
\end{nota}

We finally extend the cabled crossing $\cabledcross_{m,n}$ with the help of the slide maps from \Cref{prop:sliding-objects} and the homotopies from \Cref{thm:naturality} to an object in a twisted $A_\infty$-monoidal centralizer. We refer to \cite{KellerAinf} for basic definitions on $A_\infty$-algebras. We obtain the \emph{Second Naturality Theorem}: 
\begin{theorem}[Naturality of the braiding on chain complexes]
    \label{thm:cabledcross-data}
With notation from \Cref{notationdgcats} the following hold for  $m,n\in\N_0$:
\begin{enumerate}
    \item Consider the Rouquier complex $X_{m,n}$ of the cabled crossing $\cabledcross_{m,n}$. 
    Then there is a morphism of $\CS_{m,n}$-bimodules
    \[\tau_{m,n}\colon \barcx(\CS_{m,n})\to {}^\phi_{\CS_{m,n}} (X_{m,n})^\psi_{\CS_{m,n}}\]
      extending the assignments  $\tau_{m,n}(\id_{Y_1}\bbar\id_{Y_1} ):= \slide_{Y_1}, \tau_{m,n}(\id_{Y_1}\bbar f_1 \bbar \id_{Y_2}) := h_{f_1}$, where $Y_1$, $Y_2$ are objects and $f_1\colon Y_2 \to Y_1$ a morphism in the subcategory $\DS_{m,n}$ of $\CS_{m,n}$. 
        \item The morphism $\tau_{m,n}$ extends to an $\Ai$-morphism $\tau^\infty_{m,n}$ of dg-algebras.  
\end{enumerate}
\begin{corollary}\label{cormain}
The cabled crossing $\cabledcross_{m,n}$ extends to an
object $(\cabledcross_{m,n}, \tau_{m,n})\in\AiZ_{\CS_{m+n}}(\phi,\psi)$. 
\end{corollary}
\end{theorem}

Intuitively, \Cref{thm:cabledcross-data} and its summary in \Cref{cormain}
capture that all of $\CS_{m,n}=\Chb(\DS_{m,n})$, in particular arbitrary
bounded complexes, can slide through $\cabledcross_{m,n}$ in a natural way, up
to coherent homotopy.

\begin{proof}[Proof of \Cref{thm:cabledcross-data}] We will carefully prove (1) and discuss a proof outline of (2). For (1) we first define $\tau_{m,n}$ on the
    subcomplex 
    \[\barcx(\CS_{m,n},\IS_{\DS_{m,n}})\subset \barcx(\CS_{m,n})\] that
    is spanned by those sequences of composable morphisms $f_0\bbar f_1 \bbar
    \cdots \bbar f_{r+1}$ with $f_i\colon Y_{i+1} \to Y_{i}$ such that
    $Y_1,\dots Y_{r+1}$ are in the subcategory $\DS_{m,n}$ of $\CS_{m,n}$,
    although $Y_0$ and $Y_{r+2}$ need not be in this subcategory. 
    
    Since $\tau_{m,n}$ is a bimodule map, we only need to define it on elements
    of the form $\id_{Y_1} \bbar f_1 \bbar \cdots \bbar f_r \bbar
    \id_{Y_{r+1}}$. The morphisms $f_i$ for $1\leq i\leq r$ are
    closed since they belong to the full subcategory $\DS_{m,n}$. We define
    $\tau$ by induction on $r$. For $r=0$ and $r=1$ we set
    \[
        \tau_{m,n}(\id_{Y_1}\bbar\id_{Y_1} ):= \slide_{Y_1},\quad
        \tau_{m,n}(\id_{Y_1}\bbar f_1 \bbar \id_{Y_2}) := h_{f_1}
    \]
    for any $f_1\colon Y_2 \to Y_1$, where $h_{f_1}$ was constructed in \Cref{thm:naturalityfull}. The properties of
    $h_{f_1}$ established there guarantee that 
    \begin{align*} \tau_{m,n}(d(\id_{Y_1}\bbar f_1 \bbar \id_{Y_2}))
         &= \phi(f_1) \circ \tau_{m,n}(\id_{Y_2}\bbar \id_{Y_2}) - \tau_{m,n}(\id_{Y_1}\bbar \id_{Y_1}) \circ \psi(f_1)  \\
         &= \phi(f_1) \circ \slide_{Y_2} - \slide_{Y_1} \circ \psi(f_1)  = d(h_{f_1}) = d(\tau_{m,n}(\id_{Y_1}\bbar f_1 \bbar \id_{Y_2})).
    \end{align*}
    This shows that $\tau_{m,n}$ is a chain map on the locus where it has already been defined. 
    Now we proceed by induction. Suppose $\tau_{m,n}$ has already been
    defined as a chain map up to length $r-1>0$ in the bar complex. We like to define it
    for length $r$, say on $\id_{Y_1} \bbar f_1 \bbar \cdots \bbar f_r \bbar
    \id_{Y_{r+1}}$. Thus, we are looking for 
    \begin{equation}\label{complex}
    \tau_{m,n}(\id_{Y_1} \bbar f_1 \bbar \cdots \bbar f_r \bbar
    \id_{Y_{r+1}})   \in \Hom^{-r}(X_{m,n}\hcomp Y_{r+1},\phi(Y_{1}) \hcomp X_{m,n})
    \end{equation}
    such that $[d, \tau_{m,n}(\id_{Y_1} \bbar f_1 \bbar \cdots \bbar f_r \bbar
    \id_{Y_{r+1}})] = \tau_{m,n}( d(\id_{Y_1} \bbar f_1 \bbar \cdots \bbar f_r
    \bbar \id_{Y_{r+1}}))$. The right-hand side here is already defined (because
    $d$ lowers the degree) and closed (because $\tau_{m,n}$ is a chain map). By
    \Cref{prop:cohomologymorphismcomplex}, the relevant complex of
    \eqref{complex} has trivial homology outside of degree zero, so the
    element is also exact. We choose $\tau_{m,n}(\id_{Y_1} \bbar f_1 \bbar
    \cdots \bbar f_r \bbar \id_{Y_{r+1}})$ to be any antiderivative
    (i.e.~preimage under the differential; the choice does not matter) and
    proceed with the induction. Finally, we extend $\tau_{m,n}$ from
    $\barcx(\CS_{m,n},\IS_{\DS_{m,n}})$ to $ \barcx(\CS_{m,n})$ by pre-composing
    with the deformation retract $\barcx(\CS_{m,n}) \to
    \barcx(\CS_{m,n},\IS_{\DS_{m,n}})$ from \Cref{prop:barcxdefr}. This
    completes the proof of (1).

For (2) we have to consider the compatibility of $\tau_{m,n}$ with the
    multiplication. By definition, the multiplication on the target
    ${}^\phi_{\CS_{m,n}} (X_{m,n})^\psi_{\CS_{m,n}}$ is given as follows:
\begin{align*}
    &\Hom^{-r}(X_{m,n}\hcomp \psi(Y_{r+1}),\phi(Y_{1}) \hcomp X_{m,n})
    \times 
    \Hom^{-s}(X_{m,n}\hcomp \psi(Y'_{s+1}),\phi(Y'_{1}) \hcomp X_{m,n})
    \\&\to 
    \Hom^{-r-s}(X_{m,n}\hcomp \psi(Y_{r+1}\hcomp Y'_{s+1}),\phi(Y_{1}\hcomp Y'_{1}) \hcomp X_{m,n})\\
    (F,G) &\mapsto F \ast G := (\id_{\phi(Y_1)}\hcomp G) \vcomp (F \hcomp \id_{\psi(Y'_{s+1})})
\end{align*}
Here and in the following we suppress the coherence maps for the monoidal functors $\phi$ and $\psi$.

Since $\tau_{m,n}$ is a dg-bimodule map, it is enough to check the required compatibility for elements of the form $(\id_{Y_1} \bbar f_1 \bbar
\cdots \bbar f_r \bbar \id_{Y_{r+1}},\id_{Y'_1} \bbar g_1 \bbar \cdots
\bbar g_s \bbar \id_{Y'_{s+1}})$. 

If $r=s=0$, then 
\begin{align} \label{eqn:slide}
 \nonumber\tau_{m,n}(\id_{Y_1} \bbar\id_{Y_1} \ast \id_{Y'_1} \bbar\id_{Y'_1}) 
&= \tau_{m,n}(\id_{Y_1\hcomp Y'_1} \bbar\id_{Y_1\hcomp Y'_1}) 
= \slide_{Y_1\hcomp Y'_1} 
\\ \nonumber &\htc (\id_{\phi(Y_1)}\hcomp \slide_{Y'_1})\vcomp (\slide_{Y_1}\hcomp \id_{\psi(Y'_1)})
= \slide_{Y_1} \ast \slide_{Y'_1} 
\\&= \tau_{m,n}(\id_{Y_1} \bbar\id_{Y_1}) \ast \tau_{m,n}(\id_{Y'_1} \bbar\id_{Y'_1})
\end{align}
where we used for the second line the definition of the slide map of a $\hcomp$-composite in
terms of the slides of the factors, which we gave in the proof of
\Cref{prop:sliding-objects}. In this line we typically only get homotopies,
instead of an equality, because of the choices we made in the above proof. For
general $r,s\geq 0$ such problems become unavoidable, given that we have chosen
the elements $\tau_{m,n}(\id_{Y_1} \bbar f_1 \bbar \cdots \bbar f_r \bbar
\id_{Y_{r+1}})$ as \emph{arbitrary} antiderivatives of certain exact elements. At best
we can, thus, hope for $\tau_{m,n}$ to be a \textit{dg-algebra map up to coherent
homotopy}, namely an $\Ai$-morphism. 

Recall, \cite{KellerAinf}, that an $\Ai$-morphism between dg-algebras $(A,m_A,d_A)$ and $(B,m_B,d_B)$ is a
sequence of degree zero maps $\eta_k\colon (A[1])^{\otimes k}\to B[1]$ for $k\geq 1$ satisfying the conditions 
\begin{align}
\nonumber \sum_{i=1}^k \eta_k \circ (\id^{\otimes i-1}\otimes d_A \otimes \id^{\otimes k-i})
&+    \sum_{i=1}^{k-1} \eta_{k-1} \circ (\id^{\otimes i-1}\otimes m_A \otimes \id^{\otimes k-1-i})
\\
\label{eqn:Aimorph}
&= d_B \circ \eta_k + \sum_{a+b=k} m_B\circ (\eta_a\otimes \eta_b)
\end{align}
for all $k\geq 1$. Specifically, $\eta_1$ is a chain map and
intertwines the multiplication up to a homotopy specified by $\eta_2$, is
associative up to a homotopy specified by $\eta_3$ etc. 

We now claim that $\tau = \tau_{m,n}$ can be extended to an $\Ai$-morphism:
\[
    \tau^\infty \colon   
A:=\barcx(\CS_{m,n})\to {}^\phi_{\CS_{m,n}} (X_{m,n})^\psi_{\CS_{m,n}}=:B
\] 
with the component $k=1$ given by $\tau^\infty_1=\tau$. As before, we will
perform the extension for $k\geq 2$ first on the dg-subalgebra $\barcx(\CS_{m,n},\IS_{\DS_{m,n}})$
and on a basis with elements of the form $\id_{Y_1} \bbar f_1 \bbar \cdots \bbar
f_r \bbar \id_{Y_{r+1}}$, where the $f_i$ are necessarily closed. Then
\eqref{eqn:Aimorph} simplifies to:

\begin{equation}
\label{eqn:Aimorphclosed}
d_B \circ \tau^\infty_k 
=  \sum_{i=1}^{k-1} \tau^\infty_{k-1} \circ (\id^{\otimes i-1}\otimes m_A \otimes \id^{\otimes k-1-i}) - \sum_{a+b=k} m_B\circ (\tau^\infty_a\otimes \tau^\infty_b)
\end{equation}

So let $k\geq 2$ and suppose $\tau^\infty_l$ has been defined for all $l<k$. Then the right-hand side
of \eqref{eqn:Aimorphclosed} is closed, as is straightforward to compute.
% To see this, we compute:
% \begin{align*}
%     d_B \circ \sum_{a+b=k} m_B\circ (\tau^\infty_a\otimes \tau^\infty_b)
%     =& 
%     \sum_{a+b=k} m_B\circ ((d_B\circ \tau^\infty_a)\otimes \tau^\infty_b)\\
%     &+ \sum_{a+b=k} m_B\circ (\tau^\infty_a\otimes (d_B\circ \tau^\infty_b))\\
%     =&
%     \sum_{a+b=k}\sum_{i=1}^{a-1} m_B\circ ((\tau^\infty_{a-1} \circ (\id^{\otimes i-1} \otimes m_A \otimes \id^{\otimes a-i-1}))\otimes \tau^\infty_b)\\
%     &+ \sum_{a+b=k} \sum_{i=1}^{b-1} m_B\circ (\tau^\infty_a\otimes (\tau^\infty_{b-1} \circ (\id^{\otimes i-1} \otimes m_A \otimes \id^{\otimes b-i-1})))\\
%     =&
%     \sum_{j=1}^{k-1} \sum_{a+b=k-1} m_B\circ (\tau^\infty_{a}\otimes \tau^\infty_b)\circ (  \id^{\otimes j-1} \otimes m_A \otimes \id^{\otimes k-j-1})\\
%     =&
%      \tau^\infty_{k-1} \circ (\sum_{i=1}^{k-2} \id^{\otimes i-1}\otimes m_A \otimes \id^{\otimes k-2-i}) \circ (\sum_{j=1}^{k-1}\id^{\otimes j-1} \otimes m_A \otimes \id^{\otimes k-j-1})\\
%     +&d_B \circ \sum_{j=1}^{k-1}   \tau^\infty_{k-1}\circ (\id^{\otimes j-1} \otimes m_A \otimes \id^{\otimes k-j-1})
% \end{align*}
% The term in the second-to-last row vanishes due to associativity of the
% multiplication $m_A$ (note that here we work in the shifted setting of $A[1]$).
% Thus, the first left-hand side of the first row equals the last row, i.e. $d_B$
% annihilates the right-hand side of \eqref{eqn:Aimorphclosed}. 
Since the relevant complexes of $B$ have no cohomology in non-zero degree for
elements of the form $\id_{Y_1} \bbar f_1 \bbar \cdots \bbar f_r \bbar
\id_{Y_{r+1}}$, the right-hand side of \eqref{eqn:Aimorphclosed} is also exact
when $k\not=2$. Exactness for $k=2$, we have seen in \eqref{eqn:slide}. Thus,
the construction of $\tau^\infty_k$ is unobstructed. We choose an arbitrary
antiderivative to define the value of $\tau^\infty_{k}$ on the basis elements
$\id_{Y_1} \bbar f_1 \bbar \cdots \bbar f_r \bbar \id_{Y_{r+1}}$ and then
continue inductively. 

Finally, we extend $\tau^\infty$ from $\barcx(\CS_{m,n},\IS_{\DS_{m,n}})$ to
$\barcx(\CS_{m,n})$ by precomposing with the deformation retract from
\Cref{prop:barcxdefr}. For this we note that the section
$\barcx(\CS_{m,n},\IS_{\DS_{m,n}})\to \barcx(\CS_{m,n})$ of the deformation
retract is a dg-algebra map, and thus the retract can be extended to an
$\Ai$-morphism \cite{MR1856940}. Now the composite is also an $\Ai$-morphism
between dg-algebras, as required in (2).
 \end{proof}

We can now provide slide chain maps for the arbitrary complexes over the diagrammatic categories, as well as corresponding slide homotopies for morphisms between such.

 \begin{definition}[General slide maps and homotopies]
    \label{def:genslides}
    We retain the notation from \Cref{thm:cabledcross-data}. For every object $Y$ of
    $\CS_{m,n}$, we define the slide chain map\footnote{If $Y$ is an
    indecomposable object concentrated in degree zero, then this agrees with
    the previous construction in \S~\ref{ssec:slides-homotopies}.}
    \[
        \slide_{Y}:= \tau_{m,n}(\id_{Y}\bbar\id_{Y}) 
        \colon  \cabledcross_{m,n} \hcomp \psi(Y) 
        \longrightarrow \phi(Y)\hcomp \cabledcross_{m,n}   
     \]
    and for every closed morphism $f\colon Y \to
    Y'$ in $\CS_{m,n}$ the slide homotopy
    \[h_f:= \tau_{m,n}(\id_{Y'}\bbar f \bbar \id_{Y}) \]
 \end{definition}

 \begin{corollary}
    \label{cor:slideandhomotopy}
    For every object $Y$ of $\CS_{m,n}$ the slide chain map $\slide_{Y}$ is a
    homotopy equivalence and for every closed morphism $f\colon Y \to Y'$ in
    $\CS_{m,n}$, the naturality square of the form \eqref{eqn:naturality}
    commutes up to the homotopy $h_f$.
 \end{corollary}
\begin{proof}
    Since $Y$ has a dual in $\CS_{m,n}$, the chain map
    $\slide_{Y}:= \tau_{m,n}(\id_{Y}\bbar\id_{Y})$ is a homotopy equivalence, see e.g. \cite[Lemma
    6.14]{GHW}. The statement about $h_f:= \tau_{m,n}(\id_{Y'}\bbar f \bbar \id_{Y})$ follows from the description of the bar differential and the fact that $\tau_{m,n}$ is a chain map. 
\end{proof}

This completes the proof of \hyperlink{thmA}{Main Theorem}.

Next we illustrate in a concrete example, how (higher) homotopies for sliding objects in the additive category become relevant for assembling the slide maps for complexes. 

\begin{example} \label{exa:cone}
Consider a morphism $f\colon Y \to Y'$ in $\DS_{m,n}$ and its cone
\[\cone(f):=\bigl( 0\to Y \xrightarrow{f} Y'\to 0\bigr)\] in $\CS_{m,n}$. Let $h_f$ denote the slide homotopy from \eqref{eqn:naturality}. Then the slide chain map for $\cone(f)$ is given by:

\[
    \begin{tikzcd}[scale=1,column sep=2.5cm]
        \bigl( \cabledcross_{m,n} \hcomp Y
        \ar[d, "\slide_{Y}"]
        \ar[r, "\id \hcomp f"]
        \ar[dr, "-h_{f}"]
        &
        \cabledcross_{m,n} \hcomp Y'\bigr)
        \ar[d, "\slide_{Y'}"]
        \\
        \bigl(Y \hcomp \cabledcross_{m,n} 
        \ar[r, "f \hcomp\id"]
    &
    Y' \hcomp \cabledcross_{m,n} \bigr)
  \end{tikzcd}
\]
For longer complexes, the slide chain maps are assembled from slide chain maps
of objects in the additive category as well as higher slide homotopies for the
components of the differential, compare \cite[Remark 6.14]{GHW}.
\end{example}

\begin{remark}As a very concrete instance of \Cref{exa:cone}, one can recover
Reidemeister 3 chain maps from
\begin{itemize}
\item The atomic slide chain maps from \Cref{lem:atomicslide},
\item The slide chain map for an identity bimodule, which is the identity,
\item The slide homotopies for start-dots (or end-dots) from \Cref{fig:startdot} (or \Cref{fig:enddot}),
\end{itemize}
cf.~\cite[(3.3) and (3.4)]{MWW}, where this argument is applied in reverse to
show that particular Reidemeister 3 chain maps preserve a certain filtration. In
fact, naturality provides two chain maps realising the Reidemeister 3, depending
on which crossing the naturality argument is being applied to. These are
homotopic, but not equal, see \cite[Remark 2.4.4]{liu2024braided}. 
\end{remark}

We constructed the braiding equivalences and finally conjecture that the braiding data provided here indeed gives $\Kbloc(\DS)$ the structure of a braided monoidal $2$-category:

 \begin{conjecture}
    \label{conj:main}
    The Rouquier complexes $m\boxtimes n \to n\boxtimes m$, together with the
    naturality data provided by the bimodule morphisms $\tau_{m,n}$, equip the
    locally graded $\k$-linear semistrict monoidal $2$-category $\Kbloc(\DS)$
    of complexes of (diagrammatic) Bott-Samelson bimodules of type A and chain maps up
    to homotopy with a braided monoidal structure in the sense of \cite[Lemma
    7]{MR1402727} and the slightly stricter notion by Crans~\cite{MR1626844}.
\end{conjecture}
\begin{remark}
If the conjecture holds, it provides a low-tech approach to an important consequence of the
main result of \cite{liu2024braided}: that $\Kbloc(\SBim)$ admits the structure
of a braided monoidal bicategory. Note, however, that the work here and
in \cite{liu2024braided} are independent in the sense that neither implies the
other: here we work in a semistrict $2$-categorical setting, but truncated at
the level of chain maps, while in \cite{liu2024braided} we work in a fully weak, but homotopy-coherent setting.
\end{remark}
The only missing part for the conjecture to hold is the verification of coherences: 
\begin{theorem}[Braiding]
    \label{thm:last}
    \Cref{conj:main} holds up to the coherence conditions 
   \begin{equation}
    \label{eq:coherencecond}
    \begin{aligned}
        (\bullet\otimes \bullet)\otimes \!\rightarrow\;,\quad 
        (\bullet\;\otimes \!\to\!)\otimes \bullet\;,\quad 
        (\!\to\!\otimes \;\bullet)\otimes \bullet,\\ 
        \bullet\otimes (\bullet\;\otimes \!\to)\;,\quad 
        \bullet\otimes (\!\to\!\otimes \;\bullet)\;,\quad 
        \to\!\otimes (\bullet\otimes \bullet)\\ 
    \end{aligned}
\end{equation} in the notation from  \cite[Lemma 7]{MR1402727}. 
Moreover, \eqref{eq:coherencecond} hold for complexes concentrated in a single degree.
\end{theorem}

\begin{proof}
We check the required data and properties for a braided monoidal $2$-category
    following \cite[Lemma 7]{MR1402727}:
    \begin{enumerate}
        \item We gave $\DS$ a semistrict monoidal structure in
        \Cref{thm:ssmtwocat} and by \Cref{thm:ssmtwocatb} it is inherited by
        $\Kbloc(\DS)$.
        \item We use the Rouquier complex of the cabled crossing $X_{m,n}\colon m \boxtimes n \to n \boxtimes m$ as underlying $1$-morphisms of the braiding,
        \item The $2$-morphisms required here are slide chain maps of the form $\slide_{Y_1\boxtimes \one_n}$ from \Cref{def:genslides}. To show  \cite[Lemma 7, (4)]{MR1402727} we use $\slide_{\one_m\boxtimes Y_2}$.
        \item[(5)] The $2$-morphisms required here are provided by Rouquier canonicity, see \Cref{thm:Rouquier-canonicity}. Likewise we deal with \cite[Lemma 7, (6)]{MR1402727}.
    \end{enumerate}
    The $2$-morphisms have to satisfy certain conditions (for us: certain composites of chain maps are homotopic), which we list in the same order as in \cite[Lemma 7]{MR1402727}:
    \begin{itemize}
        \item The compatibility of slide maps in both arguments and the
        tensorators in $\DS$ was discussed in \Cref{lem:slidetensorator} and lifted to $\Kbloc(\DS)$ as part of \eqref{eqn:naturality} in \Cref{cor:slideandhomotopy}. 
        \item The next two requirements are the naturality squares for slides, i.e. \eqref{eqn:naturality2} for $\DS$, which were lifted to $\Kbloc(\DS)$ as part of \eqref{eqn:naturality} in \Cref{cor:slideandhomotopy}.
        \item The following two requirements posit that slide maps for
        $\hcomp$-composites are homotopic to composites of slide maps, which is a consequence of \Cref{thm:cabledcross-data}.(2).
        \item Six more requirements express the compatibility of composites of two slide maps across two bundles of strands with a single slide map associated to the merged bundle of strands. For objects in $\DS_{m,n}$, these relations are built into the construction of the slide maps, for non-trivial complexes over $\DS_{m,n}$, they remain conjectural.
        \item Three compatibility conditions that follow from Rouquier canonicity.
        \item The $S^+=S^-$ relation holds by Rouquier canonicity, see also \cite[Remark 2.4.4]{liu2024braided}.
    \end{itemize}
The data (1) to (6) additionally satisfy the normalization conditions of the
slightly stricter notion of braided $2$-category by Crans \cite{MR1626844},
which expresses that the unit for the tensor product is also the unit for the braiding.
\end{proof}

\newpage
\section{Diagrammatic computation of slide homotopies}
\label{sec:diags}
We display the diagrammatic calculations of slide homotopies for diagrammatic generators $f$ that constitute the heart of the proof of \Cref{prop:slidehomotopyforgen}. In each case we compute the difference of two chain maps: the composite of a slide chain map with $f$ left of the crossings, and the composite of $f$ on the right of the crossings with a slide chain map. After displaying simplified expressions, we provide a nullhomotopy for the difference.

\begin{figure}[h]
%     \noindent\makebox[\linewidth]{\rule{\paperwidth}{.4pt}} 
% \vspace{5mm}
    \begin{gather*}
        \eqnstartdothomotopy
    \end{gather*} 
    \caption{(Non-trivial) Homotopy for the start dot.} 
    \label{fig:startdot}
\end{figure} 

\begin{figure}[!b]
    % \noindent\makebox[\linewidth]{\rule{\paperwidth}{.4pt}} 
    % \vspace{5mm}
    \begin{gather*}
        \eqnenddothomotopy
    \end{gather*}
    \caption{(Non-trivial) Homotopy for the end dot.}
    \label{fig:enddot}
\end{figure}
\begin{figure}[!b]
%     \noindent\makebox[\linewidth]{\rule{\paperwidth}{.4pt}} 
% \vspace{5mm}
    \begin{gather*}
        \eqnmergehomotopy
    \end{gather*}
    \caption{(Trivial) Homotopy for the merge vertex.}
    \label{fig:merge}
\end{figure}
\begin{figure}[!b]
%     \noindent\makebox[\linewidth]{\rule{\paperwidth}{.4pt}} 
% \vspace{3mm}
    \begin{gather*}
        \eqnsplithomotopy
    \end{gather*}
    \caption{(Trivial) Homotopy for the split vertex.}
    \label{fig:split}
\end{figure} 
\begin{figure}[!b]
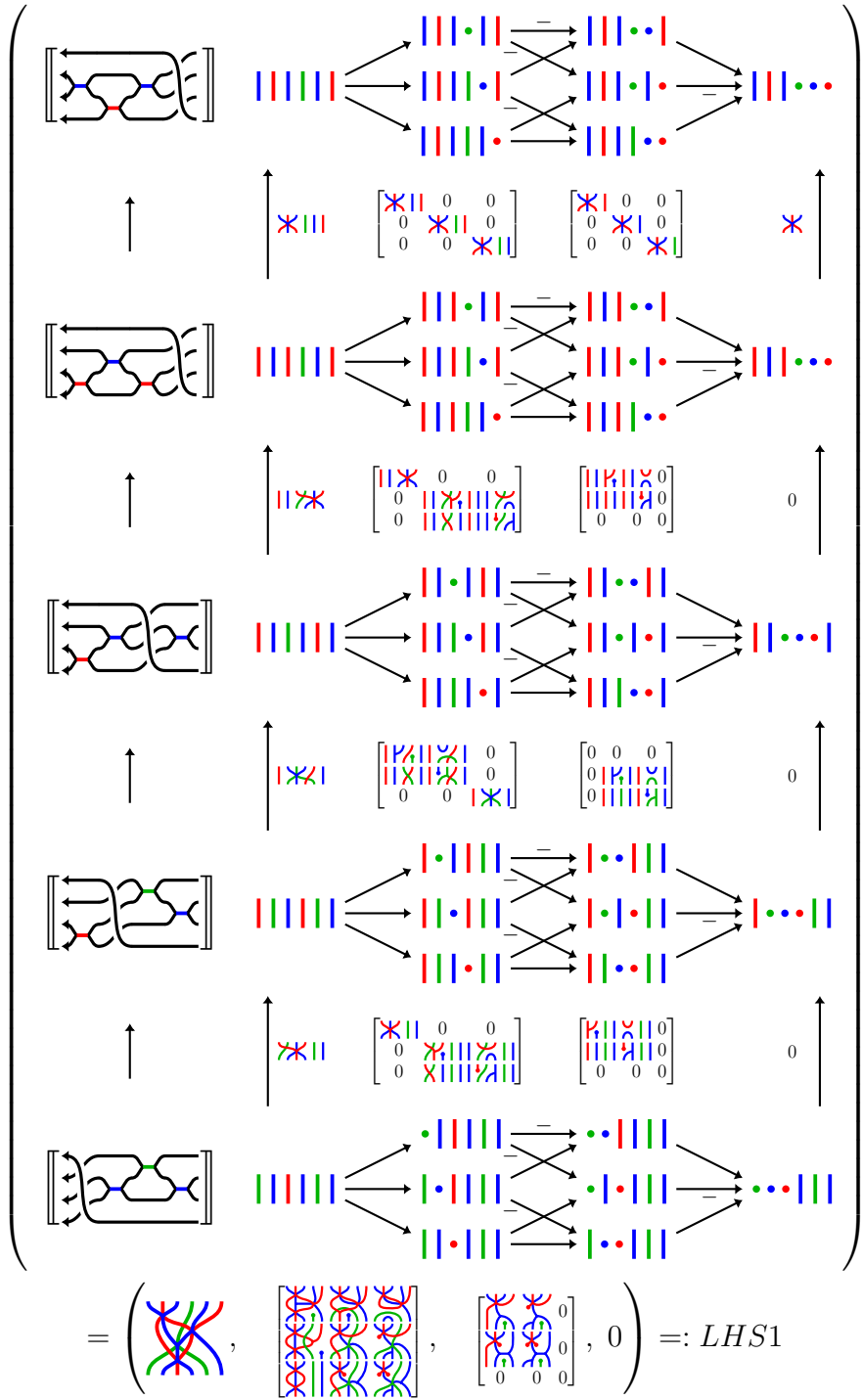

    % \noindent\makebox[\linewidth]{\rule{\paperwidth}{.4pt}} 
    % \vspace{3mm}
    \begin{gather*}
        \sixvlhs
    \end{gather*}
    \caption{Computing the composite of the first six-valent vertex with a slide map.}
    \label{fig:first6valentvertex1}
\end{figure}
\begin{figure}[!b]
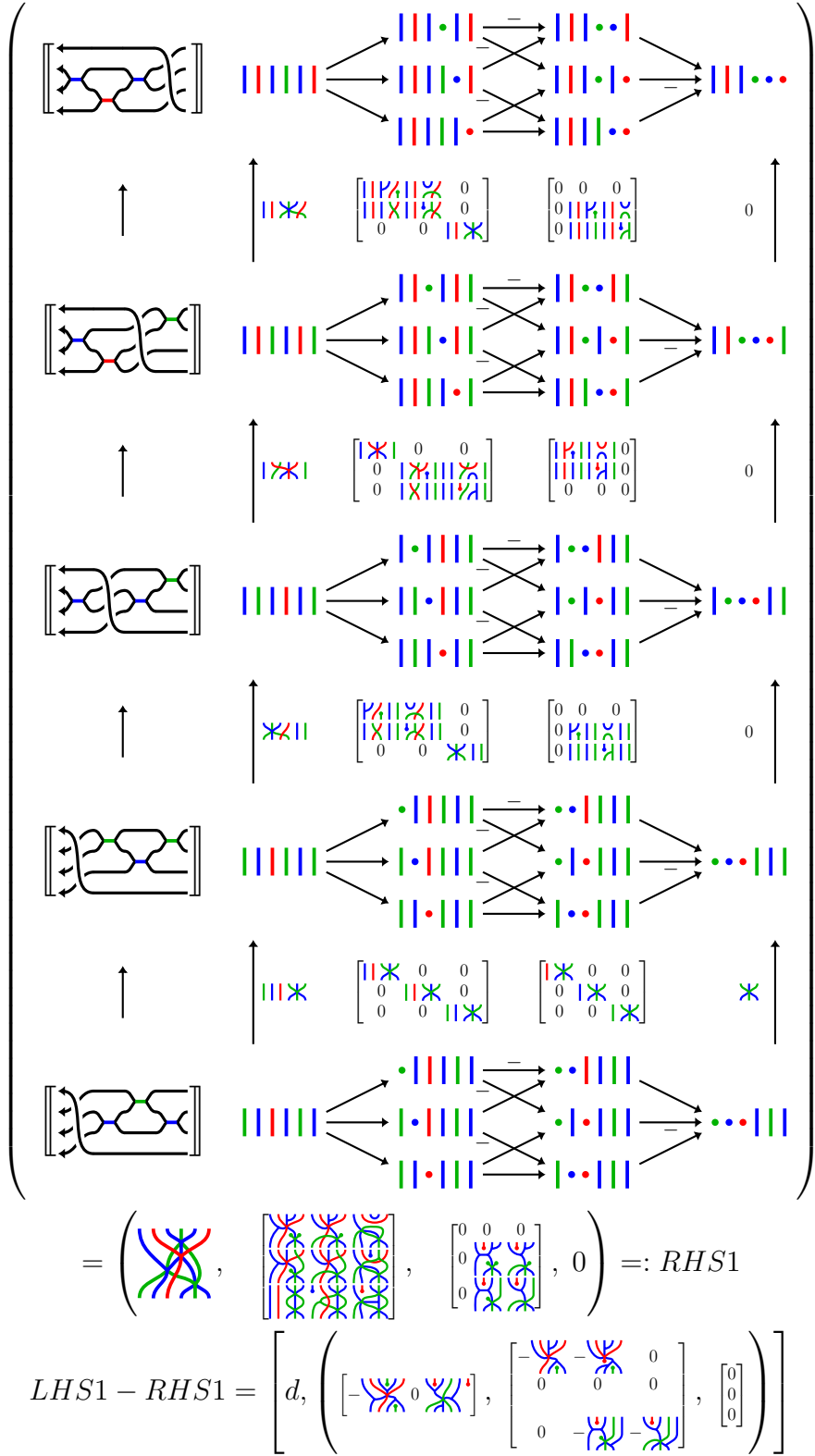

    % \noindent\makebox[\linewidth]{\rule{\paperwidth}{.4pt}} 
    % \vspace{3mm}
    \begin{gather*}
        \sixvrhs
    \end{gather*}
    \caption{Computing the composite of a slide map with the first six-valent vertex and  the (non-trivial) homotopy for the first six-valent vertex.}
    \label{fig:first6valentvertex2}
\end{figure}
\begin{figure}[!b]
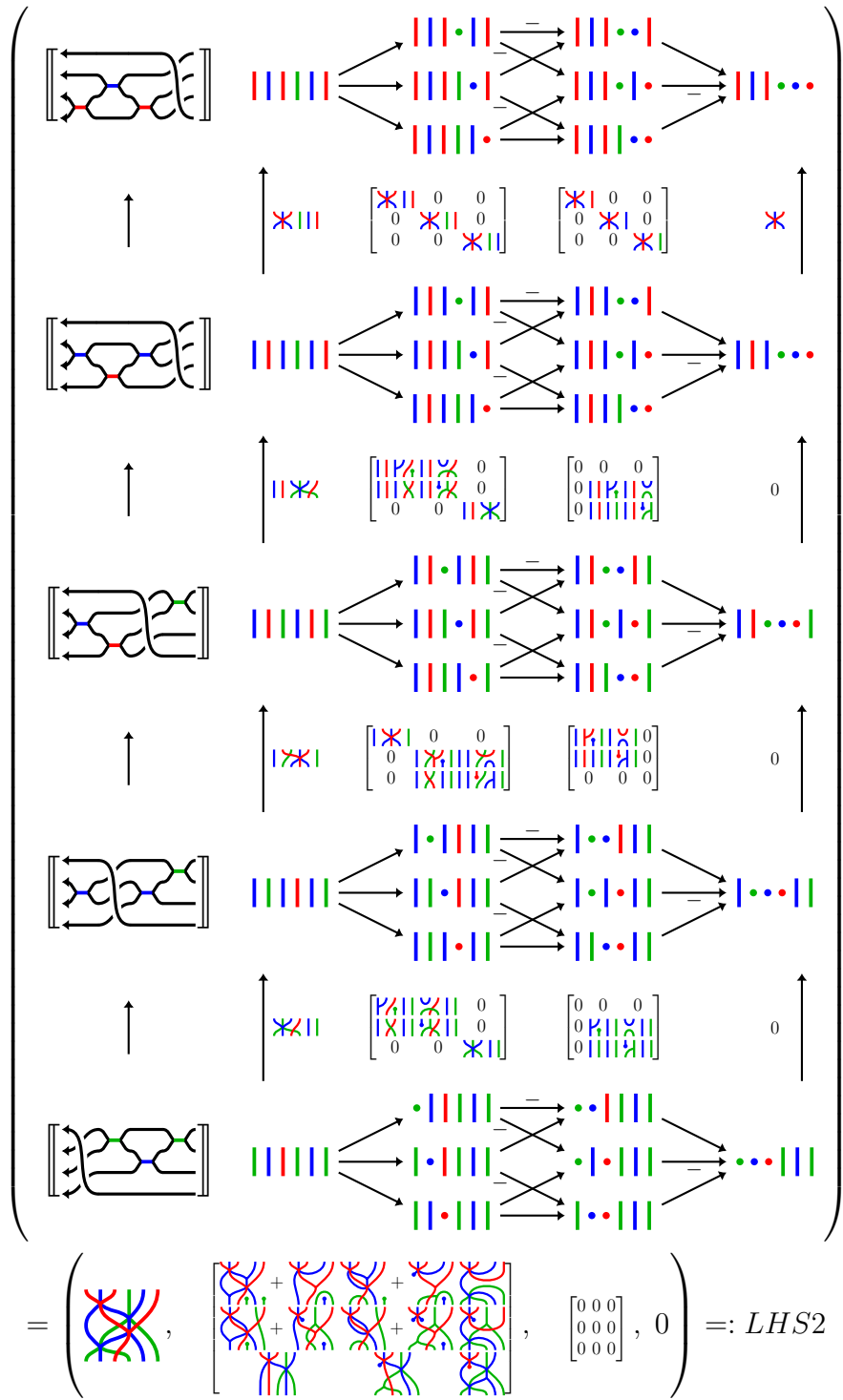

    % \noindent\makebox[\linewidth]{\rule{\paperwidth}{.4pt}} 
    % \vspace{3mm}
    \begin{gather*}
        \sixvvlhs
    \end{gather*}
    \caption{Computing the composite of the second six-valent vertex with a slide map.}
    \label{fig:second6valentvertex1}
\end{figure}
\begin{figure}[!b]
    % \noindent\makebox[\linewidth]{\rule{\paperwidth}{.4pt}} 
    % \vspace{3mm}
    \begin{gather*}
        \sixvvrhs
    \end{gather*} 
    \caption{Computing the composite of a slide map with the second six-valent vertex and the (trivial) homotopy for the second six-valent vertex.}
    \label{fig:second6valentvertex2}
\end{figure}
    \begin{figure}[!b]
        % \noindent\makebox[\linewidth]{\rule{\paperwidth}{.4pt}} 
        % \vspace{3mm}
   \begin{gather*}
        \fourvlhs
    \end{gather*}
    \caption{Computing the composite of the four-valent vertex with a slide map.}
    \label{fig:4valentvertex1}
\end{figure}
\begin{figure}
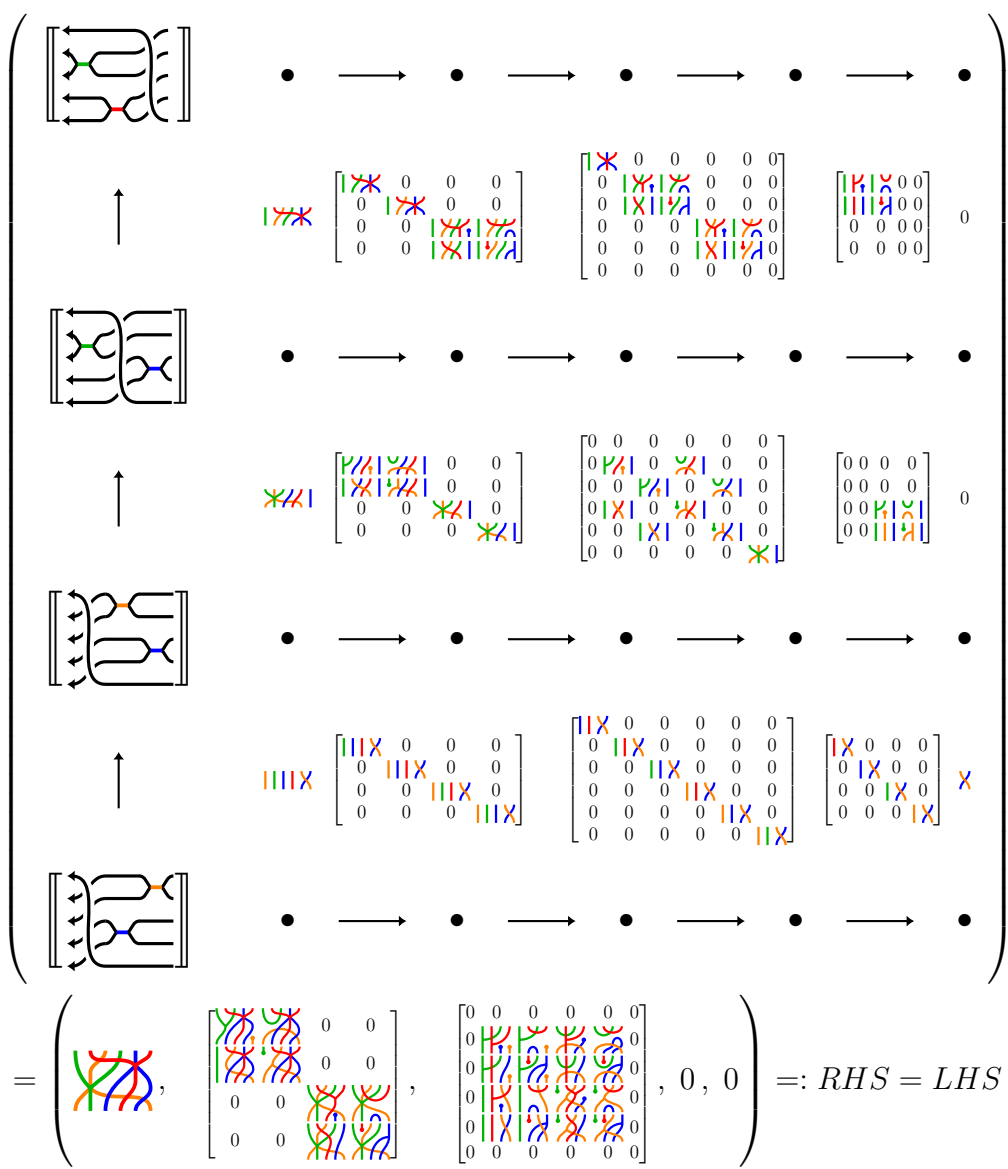

    % \noindent\makebox[\linewidth]{\rule{\paperwidth}{.4pt}} 
    % \vspace{3mm}
    \begin{gather*}
        \fourvrhs
    \end{gather*}
    \caption{Computing the composite of a slide map with the four-valent vertex and finding the (trivial) homotopy for the four-valent vertex.}
    \label{fig:4valentvertex2}
\end{figure}

\clearpage

\renewcommand*{\bibfont}{\small}
\setlength{\bibitemsep}{0pt}
\raggedright
\printbibliography

\end{document}